\documentclass[11pt]{amsart}
\usepackage[utf8]{inputenc}
\usepackage{amssymb,amsmath,amsthm,mathrsfs,eufrak}
\usepackage{xypic}
\usepackage{comment}
\usepackage{hyperref}
\usepackage{enumerate}
\usepackage{color}
\definecolor{darkgreen}{RGB}{47,139,79}
\definecolor{darkblue}{RGB}{36,24,130}
\usepackage{hyperref}
\hypersetup{
    colorlinks=true, 
    linktoc=all,     
    linktocpage,
    citecolor=darkgreen,
    filecolor=black,
    linkcolor=darkblue,
    urlcolor=black
}

\let\oldtocsection=\tocsection
\let\oldtocsubsection=\tocsubsection
\renewcommand{\tocsection}[2]{\hspace{0em}\oldtocsection{#1}{#2}}
\renewcommand{\tocsubsection}[2]{\hspace{1em}\oldtocsubsection{#1}{#2}}
\DeclareRobustCommand{\SkipTocEntry}[5]{}

\title{Homological stability  for classical groups}
\author{David Sprehn}
\author{Nathalie Wahl}

\newcommand{\im} {\operatorname{Im}}

\newcommand{\aut} {\operatorname{Aut}}
\newcommand{\Aut} {\operatorname{Aut}}

\newcommand{\rank} {\operatorname{rank}}

\newcommand{\iso} {\cong}

\newcommand{\semidirect} {\rtimes}

\newcommand{\inv} {\ensuremath ^{-1}}
\newcommand{\bbR} {\ensuremath{\mathbb{R}}}
\newcommand{\bbC} {\ensuremath{\mathbb{C}}}

\newcommand{\bbQ} {\ensuremath{\mathbb{Q}}}
\newcommand{\bbZ} {\ensuremath{\mathbb{Z}}}

\newcommand{\F} {\ensuremath{\mathbb{F}}}

\newcommand{\set}[2] {\left\lbrace {#1} \ \middle\arrowvert\ {#2} \right\rbrace}

\newcommand{\abs}[1] {\left| {#1} \right|}
\newcommand{\induction}[3] {\left. {#1} \right\uparrow_{#2}^{#3}}
\newcommand{\+} {\quad\mbox{and}\quad}

\DeclareMathAlphabet\mathbfcal{OMS}{cmsy}{b}{n}

\DeclareMathOperator\id{id}
\DeclareMathOperator\codim{codim}

\DeclareMathOperator\Cone{Cone}
\DeclareMathOperator\colim{colim}
\DeclareMathOperator\G{\mathrm{GL}}
\DeclareMathOperator\A{\mathrm{A}}
\DeclareMathOperator\Gi{\mathbf{G}}
\DeclareMathOperator\Ai{\mathbf{A}}
\DeclareMathOperator\GA{\Gamma}
\DeclareMathOperator\bP{\overline{\mathcal{T}}}
\DeclareMathOperator\PP{\mathcal{T}}
\DeclareMathOperator\PPi{\mathbfcal{T}} 
\DeclareMathOperator\pP{\mathcal{P}}
\DeclareMathOperator\HH{\mathcal{H}}
\DeclareMathOperator\R{\mathfrak{R}} 
\DeclareMathOperator\g{g} 
\DeclareMathOperator\K{\mathfrak{K}} 
\newcommand{\Si}{\Sigma}
\newcommand{\Forms}{\operatorname{Form}}
\newcommand{\GL}{\operatorname{GL}}
\newcommand{\OO}{\operatorname{O}}
\newcommand{\Sp}{\operatorname{Sp}}
\newcommand{\UU}{\operatorname{U}}
\newcommand{\Z}{\mathbb{Z}}
\newcommand{\St}{\operatorname{St}}

\theoremstyle{definition}
\newtheorem{thm}{Theorem}[section]
\newtheorem{prop}[thm]{Proposition}
\newtheorem{lemma}[thm]{Lemma}
\newtheorem{lem}[thm]{Lemma}
\newtheorem{cor}[thm]{Corollary}

\newtheorem{rem}[thm]{Remark}
\newtheorem{defn}[thm]{Definition}
\newtheorem{Def}[thm]{Definition}

\numberwithin{equation}{section}

\newtheorem{Th}{Theorem}

\newtheorem{Co}[Th]{Corollary}

\newtheorem{Thp}{Theorem}

\newcommand{\al}{\alpha}
\newcommand{\s}{\sigma}

\newcommand{\eps}{\varepsilon}

\newcommand{\op}{\oplus}
\newcommand{\x}{\times}
\newcommand{\minus}{\backslash}

\newcommand{\inc}{\hookrightarrow}

\newcommand{\rar}{\longrightarrow}
\newcommand{\sta}{\stackrel}

\newcommand{\note}[1]
{{\bf [N: #1]}}

\begin{document}

\begin{abstract}
We prove a slope 1 stability range for the  homology of the
symplectic, orthogonal and unitary groups with respect to the
hyperbolic form, over any fields other than $\F_2$, improving the
known range by a factor 2 in the case of finite fields.  Our result
more generally applies to the automorphism groups of vector spaces
equipped with a possibly degenerate form (in the sense of Bak, Tits
and Wall). For finite fields of odd characteristic, and more generally
fields in which $-1$ is a sum of two squares, we deduce a stability
range for the orthogonal groups with respect to the Euclidean form, and a corresponding result for the unitary groups.

In addition, we include an exposition of Quillen's unpublished slope 1
stability argument for the general linear groups over fields other
than $\F_2$, and use it to recover also the improved range of
Galatius--Kupers--Randal-Williams in the case of finite fields, at the characteristic. 
\end{abstract}

\maketitle

\section{Introduction}

The homology of the classical groups has long been known to stabilize, in quite large generality, see eg., \cite{Bet,Cathelineau,Charney,Essert,Frie,Maazen,Mirzaii-vdK,Sah,vdK,Vogtmann}.  
Our main result is an improvement on the known stability range for the  symplectic, unitary and orthogonal groups over finite fields. 
The proof is based on an unpublished argument of Quillen in the case of $\GL_n$, which we also present here. The proof works for the groups defined over any field, finite or infinite, as long as it is not $\F_2$: 
\begin{Th}\label{thm:intro stability}
Let $\F$ be a field other than $\F_2$.  The stabilization maps
\begin{align*}
H_d(\GL_n(\F);\Z) &\rar H_d(\GL_{n+1}(\F);\Z)\\
H_d(\Sp_{2n}(\F);\Z) &\rar H_d(\Sp_{2n+2}(\F);\Z)\\
H_d(\UU^\s_{n,n}(\F);\Z)&\rar H_d(\UU^\s_{n+1,n+1}(\F);\Z)\\
H_d(\OO_{n,n}(\F);\Z) &\rar H_d(\OO_{n+1,n+1}(\F);\Z)
\end{align*}
are isomorphisms in degrees $d\leq n-1$ and surjections in degree $d=n$. 

Moreover, if $\F$ is the finite field $\F_{\ell^r}$  for $\ell$ a prime number with $\ell^r\neq 2$, 
$$H_d(\G_n(\F_{\ell^r});\F_\ell)\rar  H_d(\G_{n+1}(\F_{\ell^r});\F_\ell)$$ 
is an isomorphism for $d\le n+r(\ell-1)-3$ and a surjection in degrees  $d= n+r(\ell-1)-2$. 
\end{Th}

The case of $\GL_n$ 
is the stability range obtained by Quillen in his note books \cite[p10]{QuillenNotes} (which though starts with a couple of bleached pages!), with 
the improvement of Galatius--Kupers--Randal-Williams \cite{GKRW} in the case of finite fields $\F_{\ell^r}$ at the characteristic $\ell$; the improved range is obtained here by inputing a twisted coefficient computation \cite[Lem 5.2]{GKRW} into Quillen's argument
\footnote{The paper \cite{GKRW} gives in addition a slope $\frac{2}{3}$ range for $\GL_n(\F_2)$ with $\F_2$ coefficients, which we do not recover with our methods.}.  
For the other groups in the case $\F$ a finite field, our result improves the best known stability ranges by  a factor of two. (See Essert~\cite[Thm 3.8]{Essert} who obtains a bound $d\leq \frac{n}{2}-1$, slightly improving \cite[Thm 8.2]{Mirzaii-vdK}, and \cite{LoovdK} for symplectic group. Mirzaii \cite{Mir} has a slope 1 range for finite fields but away from the characteristic.)   
In the case of an infinite field,
our ranges match those obtained by  Essert~\cite[Thm 3.8]{Essert} and Sah~\cite[App.~B]{Sah} (see also \cite{Mir}), except that Sah also gets injectivity in degree $n$ for the general linear groups, which does not hold in our generality,  
see Remark~\ref{rem:range}.

Our result in the case of symplectic, unitary and orthogonal groups, follows from a more general stability result, stabilizing the automorphism group of any {\em formed space} with the {\em hyperbolic form}, as we will explain now. We start by recalling the definition of these objects, and explain how they relate to the above groups.  


\subsection{Automorphisms of formed spaces}

To prove our result, we use the general framework of forms originating in the work of Bak, Tits and Wall (see eg.~\cite{Bak69,Tits,Wallaxiomatic,WallII}), and later developed also by Magurn--van der Kallen--Vaserstein \cite{MVV}. Given a field $\F$,  a (possibly trivial) involution $\s:\F\to \F$, an element $\eps\in \F$ such that $\eps\bar \eps=1$, for $\bar\eps:=\s(\eps)$, and a certain  additive subgroup $\Lambda \le \F$, a {\em form} $q$ on an $\F$--vector space $E$ is an element of the quotient 
\[\Forms(E,\s,\eps,\Lambda):= \operatorname{Sesq}_{\sigma}(E)/X(E,\sigma,\eps,\Lambda)  \]
of the sesquilinear forms on $E$ by a subgroup $X(E,\sigma,\eps,\Lambda)$ that depends on the chosen $\s,\eps$ and $\Lambda$ (see Definition~\ref{def:form}). We call the pair $(E,q)$ a {\em formed space}. 
Direct sum induces a monoidal structure on the category of formed spaces and form-preserving maps. 

 To a form $q\in \Forms(E,\s,\eps,\Lambda)$, one associates two new maps 
$\omega_q:E\x E\to \F$ and $Q_q:E\to \F/\Lambda$ defined by 
 \[ \omega_q(v,w)=q(v,w)+\eps\overline{q(w,v)} \ \ \ \textrm{and}\ \ \  Q_q(v)=q(v,v)+\Lambda\in\F/\Lambda, \]
where $\omega_q$ is an ``$\eps$-skew symmetric sesquilinear form'' by construction, and $Q_q$ should be thought of as  a ``quadratic refinement'' of $\omega_q$.  
We show in our companion paper \cite[Prop 2.7]{Forms} that the form $q$ is always determined by either $\omega_q$ or $Q_q$. As $\omega_q$ and $Q_q$ for various choices of $\s$, $\eps$ and $\Lambda$ recover the classical  notions of alternating, Hermitian and quadratic forms, this allows to study the symplectic, unitary  and orthogonal groups, all within one framework. More precisely, 
for $(E,q)$ a formed space, let 
$$\Gi(E)=\Gi(E,q):=\aut(E,q)$$
denote its automorphism group, that is the group of bijective linear maps preserving the form, where we drop $q$ from the notation when $q$ is unambiguous. 
The classical groups appearing in Theorem \ref{thm:intro stability} are automorphisms of the formed space $\HH^{\op n}$ for various choices of the parameters $\s,\eps$ and $\Lambda$, where  $\HH=(\F^2,q_{\HH})$ is  the ``hyperbolic'' formed space, with 
$$q_{\HH}:\F^2\x \F^2\to \F$$
defined by $q_{\HH}(v,w)=\overline{v_1} w_2$, for $\overline{v_1}:=\s(v_1)$ as above:

\begin{Def}\label{def:classicalgroups}

\noindent
(Symplectic) For $\sigma=\id$, $\eps=-1$, $\Lambda=\F$, we have 
$\omega_{q_{\HH}}(v,w)=v_1w_2-v_2w_1$ is a non-degenerate alternating form 
and the {\em symplectic group} is 
\[ \Sp_{2n}(\F) =\Aut(\F^{2n},\omega_{q_{\HH}}^{\op n})= \Gi(\HH^{\op n}). \]

\smallskip

\noindent
(Unitary) For $\sigma\neq \id$, $\eps=1$, $\Lambda=\F^\s$, we have 
$\omega_{q_{\HH}}(v,w)=\overline{v_1}w_2+\overline{v_2}w_1$ is a non-degenerate Hermitian form, 
and the {\em unitary group} is
\[ \UU_{n,n}(\F)=\UU_{n,n}(\F,\s)=\Aut(\F^{2n},\omega_{q_{\HH}}^{\op n})=\Gi(\HH^{\op n}). \]

\smallskip

\noindent
(Orthogonal) For $\sigma=\id$, $\eps=1$, $\Lambda=0$, we have 
$Q_{q_{\HH}}(v)=v_1v_2$ is a quadratic form 
with associated non-degenerate symmetric bilinear form
$\omega_{q_{\HH}}(v,w)=v_1w_2+v_2w_1$,
and the {\em orthogonal group} is  
\[ \OO_{n,n}(\F) =\Aut(\F^{2n},Q_{q_{\HH}}^{\op n})=\Gi(\HH^{\op n}). \]
\end{Def}
So, up to the choice of parameters $(\s,\eps,\Lambda)$, the symplectic, unitary and orthogonal groups of Theorem \ref{thm:intro stability} are really the same group $\Gi(\HH^{\op n})$. The fact that $\Gi(\HH^{\op n})$ is the same subgroup of $\GL(\F^{2n})$ as $\Aut(\F^{2n},\omega_{q_{\HH}}^{\op n})$ in the first two cases and $\Aut(\F^{2n},Q_{q_{\HH}}^{\op n})$ in the third is given by \cite[Thm~A and Prop~2.7]{Forms}, and the fact that the form $q_{\HH}^{\op n}$ is non-degenerate follows from  Lemma~\ref{lem:properties of HH}.  (A form $q$ is called non-degenerate if the associated form $\omega_q$ is non-degenerate, see Definition~\ref{def:kerneletc}.)

\medskip

With this language in place, Theorem~\ref{thm:intro stability} is actually a consequence of the following more general result: 

\begin{Thp}[see Theorem~\ref{thm:mainstability}]\label{thm:stabilityintro'}
Let $E=(E,q)$ be a formed space over a field $\F\neq \F_2$.  Then 
$$\begin{array}{ll}
H_d(\Gi(E\op \HH),\Gi(E);\Z)=0 & \textrm{for}\ \ d\le \g(E)\\
\end{array}$$ 
where  the {\em genus} $\g(E)$ is the maximal dimension of an isotropic subspace in $E$ modulo the radical
(see Definition~\ref{def:kerneletc}).
\end{Thp}
Note that $E$ in the theorem is allowed to be degenerate, and again only the field $\F_2$ is excluded; this assumption is used in the form
$\F_{\F^\times}=0$, 
which is true for all fields other than $\F_2$. 
Three out of four cases in Theorem~\ref{thm:intro stability} are obtained by taking $(E,q)=\HH^{\op n}$ with the three choices for $(\s,\eps,\Lambda)$ of Definition~\ref{def:classicalgroups}, given that $\HH^{\op n}$ has genus $n$ (by Lemma~\ref{lem:properties of HH}). 
 The remaining case of the general linear groups in Theorem~\ref{thm:intro stability}   is not a corollary of the above theorem, since the trivial form on $\F^n$ has genus zero; this case, also treated in Theorem~\ref{thm:mainstability},  thus
receives a certain amount of separate treatment. Theorem~\ref{thm:mainstability}  contains in addition  a stability result for certain affine groups, and the improved range for linear groups over finite fields at the characteristic, as in Theorem~\ref{thm:intro stability}. The proof of this improved range does not carry over to likewise give an improved range in the case of the symplectic, orthogonal and unitary groups, the main obstacle being the connectivity of the building used (see Remark~\ref{rem:range}).

\subsection{Euclidean orthogonal and unitary groups}

We consider in addition stability for 
the Euclidean orthogonal and unitary groups $\OO(n,\F)$ and $\UU(n,\F,\sigma)$ associated to the form which has an orthonormal basis
 (see Definition~\ref{def:Euclidean}): 
\begin{Th}
\label{thm:introeucli}
(1) Let $\F$ be a field with $\operatorname{char}(\F)\neq 2$ and so that $-1=a^2+b^2$ for some $a,b\in \F$. 
Then the stabilization map 
$$ H_d(\OO(n,\F);\Z)\rar H_d(\OO(n+1,\F);\Z)$$ 
is an isomorphism in degrees
$ d \leq \frac{n-4}{2}$  
and a surjection in degree
$ d = \frac{n-3}{2}$.

\smallskip

\noindent
(2) Let $\F\neq \F_2$, and $\sigma=\overline{(\cdot)}$ a nontrivial involution of $\F$, such that $-1=\bar a a+\bar b b$ for some $a,b\in \F$.
Then the stabilization map 
$$ H_d(\UU(n,\F,\sigma);\Z)\rar H_d(\UU(n+1,\F,\sigma);\Z)$$
 is an isomorphism in degrees
$d \leq \frac{n-4}{2}$,
and a surjection in degree $d = \frac{n-3}{2}$.

\smallskip

\noindent
If  $b$ in (1) or (2) can be chosen to be 0,  the corresponding bounds can be improved by 1.
\end{Th}

Cathelineau~\cite{Cathelineau} obtained such a stability result for $\OO(n,\F)$ with range $d\leq n-1$ under the assumption that $\F$ is an infinite Pythagorean field\footnote{A {\em Pythagorean field} is a field in which every sum of two squares is a square.} of characteristic not  2 (see also \cite{Vog82} for a related result), and  Sah~\cite{Sah} for  $\UU(n,\F)$ with  a range $d\leq n-1$ in the case  $\F=\bbR$ or $\bbC$. See also \cite{Coll}. 

Part (1) of the theorem applies in particular to finite fields $\F_q$ for $q$ odd and part (2) to $\F_{q^2}$ with the  nontrivial involution for any $q\neq 2$ (see Lemmas~\ref{lem:sumsquares} and \ref{lem:norm}). 
Besides finite fields, Theorem~\ref{thm:introeucli} in the orthogonal case also applies to some infinite fields which are not Pythagorean, such as $\bbQ(i)$.

\smallskip

Theorem~\ref{thm:introeucli} is actually a consequence of our main stability theorem, Theorem~\ref{thm:stabilityintro'}. Just like in the previous case, there is a more general statement behind: 
\begin{Thp}\label{thm:intronondg} 
Let $(\F,\s)$ be a field with involution satisfying that $-1=\bar a a+\bar b b$ for some $a,b\in \F$, and let $(E,q)$ be a non-degenerate formed space. 
Then  the stabilization map 
$$H_d(\Gi(E^{\op n});\Z) \rar  H_d(\Gi(E^{\op n+1});\Z)$$
 is injective in degrees  $d\leq k(\lfloor \frac{n}{2}\rfloor-e)-1$,
and a surjection in degrees $d\leq k(\lfloor\frac{n-1}{2}\rfloor-e)$, where $k=\dim(E)$ and $e=0$ if $b$ can be chosen to be $0$ in the above equation, and $e=1$ otherwise. 
\end{Thp}

Theorem~\ref{thm:introeucli} is obtained by applying the above result to the Euclidean formed space $\mathcal{E}=(\F,q_{\mathcal{E}})$, with $q_{\mathcal{E}}(a,b)=\bar a b$. This formed space is non-degenerate as long as $\eps\neq -1$, which is the reason for the assumptions in Theorem~\ref{thm:introeucli}---see Remark~\ref{rem:eucli} more details about this, for why there is no Euclidian symplectic group. 

\begin{rem}
One may wonder whether Theorems~\ref{thm:introeucli} and~\ref{thm:intronondg}, having slope 2 ranges, could be proved using a more direct and simpler stability argument. The groups $\Gi(E^{\op n})$ fit eg.~into  the framework of  \cite{RWW} and stability with slope 2 will follow if one can show that a certain associated building is at least slope 2 connected  (with no a priori non-degeneracy assumption, or assumption on the field). The connectivity of this building is however not known to us. In the case of the groups $\OO(n,\F)$, this building is closely related to the poset of orthonormal frames studied in \cite{Vog82}, where a connectivity bound dependent on the Pythagorian number of the field\footnote{the smallest integer $p=p(\F)$ such that every sum of squares in $\F$ can be written as a sum of $p$ squares} is given (see Corollary 1.8 in that paper). The connectivity of a similar building is proved by Cathelineau \cite[Prop 4.2]{Cathelineau} under the hypothesis that $\F$ is Pythagorian, that is with Pythagorian number 1. 
\end{rem}

\subsection{Vanishing ranges} 

The homology of the general linear groups, symplectic, unitary and orthogonal groups over finite fields was completely computed by Quillen \cite{QuillenK} and Fiedorowicz-Priddy \cite{FP}\footnote{See Theorems~IV.2.1, IV.5.2, IV.7.2, V.2.1, V.3.1, V.5.1 in \cite{FP}.} 
away from the characteristic of the field. At the characteristic, the stable homology of these groups was shown to vanish by the same authors (see also \cite{Friedlander,Hiller}), but 
their unstable homology remains rather mysterious. 
Our stability range in the case of finite fields thus translate to improved vanishing ranges for the homology of these groups at the characteristic. We make this explicit here in the symplectic, unitary and orthogonal case. The general linear group case is \cite[Cor C]{GKRW}.

Combining Theorem~\ref{thm:intro stability} with the vanishing of the stable homology \cite[Thm III.4.6]{FP} gives: 

\begin{Co}\label{cor:C}
(1)~For $\F_{\ell^r}\neq \F_2$, $$H_d(\Sp_{2n}(\F_{\ell^r});\F_\ell) = 0 \ \ \ \textrm{for}\ \ \  0<d<n$$

\noindent
(2)~For $\F_{\ell^{2r}}\neq \F_2$, $$H_d(\UU_{n,n}(\F_{\ell^{2r}});\F_\ell) =0\ \  \ \textrm{ for}\ \ \  0<d<n.$$

\noindent
(3)~For $\ell$ odd $$H_d(\OO_{n,n}(\F_{\ell^r});\F_\ell) = 0\ \ \ \textrm{ for}\ \ \  0<d<n.$$
\end{Co}

Using other results in the litterature, one can complete the above corollary as follows: 

\begin{rem} 
(1) (Symplectic and unitary groups with $\F=\F_2$) The stability result \cite[Thm 3.8]{Essert} applies to the case $\F=\F_2$ and gives 
$$H_d(\Sp_{2n}(\F_{2});\F_2) = H_d(\UU_{n,n}(\F_{2});\F_2) =0 \ \ \ \textrm{for}\ \ \  0<d<\frac{n-1}{2}.$$

\smallskip

\noindent (2) (Orthogonal groups in characteristic 2)
The first homology group $H_1(\OO_{n,n}(\F_{\ell^r});\F_\ell) =\Z/2$ when $\ell=2$ \cite[II.7.19]{FP}, so there is no vanishing range in characteristic two for orthogonal groups. There is however a stable vanishing for the commutator subgroup $\OO'_{n,n}(\F_{\ell^r})$ of $\OO_{n,n}(\F_{\ell^r})$ \cite[Thm III.4.6]{FP}. (This subgroup identifies with the subgroup $D\OO_{n,n,}(\F_{\ell^r})$ of matrices with trivial Dickson invariant, see \cite[II.7.19]{FP}.)  Stability for the commutator subgroups follows from a twisted stability theorem for the groups, as in \cite[Cor 3.9]{RWW}, and in fact, the groups $\OO'_{n,n}(\F_{\ell^r})$ occur as a special case of Theorem 5.16 in \cite{RWW}, which, combined with the vanishing result of \cite{FP} gives that 
  $$H_d(\OO'_{n,n}(\F_{2^r});\F_2) = 0\ \ \ \textrm{ for}\ \ \  0<d<\frac{n-5}{3}.\footnote{where we used that the unitary stable rank of a field is  1 \cite[Rem 6.4]{Mirzaii-vdK}}$$
\end{rem}

In finite characteristic, there is one isomorphism class of non-degenerate symplectic and Hermitian form in each dimension,  and only one orthogonal group
in odd dimensions,  see \cite[Sec II.4,6,7]{FP}. In even dimensions and odd characteristic, there are two non-degenerate quadratic forms that are not isomorphic, corresponding to the groups usually called 
$\OO^+(2n,\F_{\ell^r})=\OO(2n,\F_{\ell^r})$ and $\OO^-(2n,\F_{\ell^r})$, see \cite[II.4.5]{FP}. (See Remark~\ref{rem:eucli}(3) for the characteristic 2 case.) The group $\OO_{n,n}(\F_{\ell^r})$ defined above is isomorphic to the first or the second of these groups depending on whether $(-1)^n$ has a square root in $\F_{\ell^r}$ or not.  There is though only one stable orthogonal group \cite[Prop II.4.11]{FP}. (See also Theorem~\ref{thm:stable group}.)

Our stability results, again combined with the vanishing results of \cite{FP}, also prove:

\begin{Co}\label{cor:D}
(1) For $\ell^{r}\neq 2$, $$H_d(\UU(n,\F_{\ell^{2r}});\F_\ell)=0 \ \ \ \textrm{for}\ \ \  0<d<\frac{n}{2}.$$ 

\noindent
(2) For $\ell$ an odd prime,
$$H_d(\OO^{\pm}(n,\F_{\ell^r});\F_\ell)=  0 \ \ \ \textrm{for}\ \ \  0<d<\frac{n-1}{2}$$
where $\OO^\pm$ stands for either $\OO^+$ or $\OO^-$ when $n$ is even, and for $\OO$ when $n$ is odd. 
\end{Co}

There is some overlap between the two corollaries as $\UU_{n,n}(\F_{\ell^{2r}})\cong \UU(2n,\F_{\ell^{2r}})$ and $\OO_{n,n}(\F_{\ell^r})$ is isomorphic to either $\OO^+(2n,\F_{\ell^r})$ or $\OO^-(2n,\F_{\ell^r})$. 
For the unitary groups, both results gives the same range, whereas for the orthogonal groups, the first result is better when there is overlap. 
(See Remark~\ref{rem:better} for more details.)

\subsection{Buildings} 

The proofs of the homological stability results proceed by studying the Tits building of a vector space, and the isotropic building of a formed space, together with their actions by the relevant automorphism groups.  These buildings are regarded as posets: the poset of nontrivial subspaces of a vector space, and the poset of isotropic subspaces of a formed space (containing the radical). In both cases, we need relative variants of the posets; see Definitions~\ref{def:GL building} and~\ref{def:iso buildings}.
These posets are highly connected, and indeed satisfy the Cohen-Macaulay property, i.e.~all their intervals are homotopy-spherical: this is well-known for the general linear building, and was proven in a special case by Vogtmann in~\cite{Vogtmann} for the isotropic building. We show in Theorems~\ref{thm:PE is C-M}  and~\ref{thm:PEU is CM}    that Vogtmann's argument extends to our more general set-up. 
Note that it is crucial for obtaining a slope 1 stability with our argument that the building is at least slope 1 connected. 

\subsection{Spectral sequence argument} 

We prove the stability theorem using spectral sequences associated to the action of the groups on the buildings: Following Quillen, we associate a small chain complex to the Cohen-Macaulay posets (see Appendix~\ref{app:CM}), which gives a double complex using the group action and hence a spectral sequence (see Section~\ref{sec:SS}). 
Stabilizers of the action appear in the spectral sequence and must be treated in parallel; these are affine versions of the groups. These stabilizer subgroups are more complicated than those arising when using simplicial complexes of unimodular vectors or split versions of these (as for example in~\cite{Charney,Maazen,Mirzaii-vdK,RWW}),  see  Proposition~\ref{prop:stabilizers}. 
Also, the fact that we work with posets rather than simplicial complexes gives more complicated twisted coefficients in the spectral sequence, in the present case twisted by the Steinberg module, the top homology of the general linear building (the same Steinberg module in all cases!).  But these two difficulties are in fact, together, the reason why we can improve the slope in the argument: indeed, they combine to give the vanishing of certain groups in the spectral sequence in the form of the vanishing of  the Steinberg coinvariants. This vanishing is well-known for the non-relative building, see eg.~\cite[Thm.~1.1]{APS}, and we give here a proof for the relative case  (see Theorem~\ref{thm:relcoinvars}), again following ideas outlined by Quillen, using a join decomposition of the building (Theorem~\ref{thm:join decomposition}). 
In the case of finite fields, this vanishing was improved at the characteristic by Galatius--Kupers--Randal-Williams to a vanishing of the twisted homology in a range \cite[Lem 5.2]{GKRW}, yielding the improved stability range under the same assumption. 

\medskip

It might be possible to apply the above argument to eg.~$\GL_n(R)$ for $R$ belonging to a certain class of rings. 
In addition to the coinvariant vanishings mentioned above, the main properties of the buildings that are used in the proof of the stability theorem are summarized in Section~\ref{sec:props}.

One can also ask whether other families of groups would fit into the arguments presented here, and 
look for a general set-up making such a spectral sequence work, for example for groups arising as automorphism groups of ``uncomplemented versions'' of the  complemented categories  of \cite{PutSam} or the homogeneous categories of \cite{RWW}, as suggested by the examples considered here. We do not know whether such a general set-up can be formulated. We did attempt to formulate an abstract list of ingredients making the above spectral sequence argument work, and found a list that was long and unenlightning.

\medskip

\noindent
{\em Organization of the paper: } In Section~\ref{sec:forms}, we give a brief introduction to the framework of forms. 
In Section~\ref{sec:buildings}, we define and study the general linear buildings with the associated Steinberg modules, while Section~\ref{sec:isotropicbuilding} is concerned with the buildings of isotropic subspaces of formed spaces. Section~\ref{sec:groupsandbuildings} defines and studies the groups acting on the buildings. 
Section~\ref{sec:stability} gives the proof of Theorems~\ref{thm:intro stability} and \ref{thm:stabilityintro'} through the more general Theorem~\ref{thm:mainstability}, and deduces Corollary~\ref{cor:C},   while Section~\ref{sec:eucli} treats the case of the Euclidean orthogonal and unitary groups, proving Theorem~\ref{thm:intronondg} and Corollary~\ref{cor:D}. Finally,  Appendix~\ref{app:isotropic} gives some basic properties of subspaces of formed spaces, while the short  Appendix~\ref{app:CM} associates a chain complex to a Cohen-Macaulay poset.

\addtocontents{toc}{\SkipTocEntry}
\subsection*{Acknowledgements}
The authors were both supported by the Danish National Research Foundation through
the Centre for Symmetry and Deformation (DNRF92). The second author was also supported by the  European Research
Council (ERC) under the European Union's Horizon 2020 research and innovation programme (grant agreement No.~772960), the Heilbronn Institute for Mathematical Research, and 
would like to thank the Isaac Newton Institute for Mathematical Sciences, Cambridge, for support and hospitality during the programme Homotopy Harnessing Higher Structures, where this paper was finalized (EPSRC grant numbers EP/K032208/1 and EP/R014604/1). 
The authors would like to thank S{\o}ren Galatius, Alexander Kupers and Oscar Randal-Williams for useful discussions around Quillen's stability argument for general linear groups, and the referee for suggesting using Lemma 5.2 of \cite{GKRW} to improve our stability range.

\tableofcontents

\section{Forms and classical groups}\label{sec:forms}

We give here a brief introduction to the framework of forms of Bak,
Tits and  Wall \cite{Bak69,Tits,Wallaxiomatic,WallII})  (see also \cite{MVV}) restricting to the case of fields, referring to \cite{Forms} for more details.  
We introduce the notion of kernel, radical and genus for a vector
space equipped with a form, and study the hyperbolic form. 

\medskip

Fix a field $\F$ and a field involution $\s:\F\to\F$ (possibly the identity). Throughout, we will write $\bar c=\sigma(c)$ in analogy with complex conjugation.
Given a scalar $\eps\in\F$ satisfying
$\eps\bar\eps=1$, 
define
\begin{align*} \Lambda_{min} &= \set{a-\eps\bar a}{a\in\F} \\
 \Lambda_{max} &= \set{a\in\F}{a+\eps\bar a=0}. \end{align*}
These are additive subgroups of $\F$. Let $\Lambda$ be an abelian groups such that
\[ \Lambda_{min}\leq\Lambda\leq\Lambda_{max}. \]
In fact, for most triples $(\F,\s,\eps)$, we actually have $\Lambda_{min}=\Lambda_{max}$ so 
that there is a unique $\Lambda$ determined by the triple (see
\cite[Prop~A.1]{Forms}). 

For an $\F$--vector space $E$,  recall that a map $f:E\times E\to\F$ is called {\em sesquilinear} if $f$ is biadditive and 
satisfies that  $f(av,bw)=\bar af(v,w)b$ for all $a,b\in \F$ and
$v,w\in E$. The set of sesquilinear forms on a vector space $E$ forms a vector space $\operatorname{Sesq}_\sigma(E)$ with addition and scalar multiplication defined pointwise.

\begin{Def}\label{def:form}
A $(\sigma,\eps,\Lambda)$-\textit{quadratic form} on $E$ (or just a {\em form} for short) is an element of the quotient
\[ \Forms(E,\sigma,\eps,\Lambda)=\operatorname{Sesq}_{\sigma}(E)/X(E,\sigma,\eps,\Lambda),  \]
where $\operatorname{Sesq}_\sigma(E)$ denotes the vector space of sesquilinear forms on $E$
and $X(E,\sigma,\eps,\Lambda)$ is the additive subgroup of all those $f$ satisfying
\[ f(v,v)\in\Lambda \ \ \ \ \textrm{and}\ \ \  \  f(w,v)=-\eps\overline{f(v,w)}. \]
A \textit{formed space} $(E,q)$ is a finite-dimensional vector space $E$ equipped with a form $q$.
\end{Def}
The set of forms on $E$ is contravariantly functorial in $E$: from a linear map $f:E\to F$ and a form $q$ on $F$, we produce a form on $E$ by setting  $f^*q(v,w)=q(f(v),f(w))$
for $v,w\in E$.

\medskip

As described in the introduction, to a form $q$ we can associate two new objects which capture the features of $q$ in a more accessible way: a sesquilinear map $\omega_q: E\x E\to \F$ defined by 
$\omega_q(v,w)=q(v,w)+\eps\overline{q(w,v)}$, 
which satisfies that  $\omega(v,w)=\eps \overline{\omega(w,v)}$, 
and a set map $Q_q: E\to \F/\Lambda$
$Q_q(v)=q(v,v)+\Lambda\in\F/\Lambda$, 
which satisfies that $Q_q(av)=a\bar aQ_q(v)$.  
Theorem~2.5 in \cite{Forms}  shows that the maps $\omega_q$ and $Q_q$  are independent of the choice of representative of $q$, that 
they satisfy the equation
\begin{equation}\label{equ:Qomega}
Q_q(v+w)-Q_q(v)-Q_q(w)=\omega_q(v,w) \in \F/\Lambda.
\end{equation}
and that  the form $q$ is completely determined by the pair $(\omega_q, Q_q)$.  More than that, Proposition~2.4 in that paper shows that in fact one of $\omega_q$ and $Q_q$ is always enough to determine $q$.

\medskip

Formed spaces form a category $\textbf{Forms}(\F,\sigma,\eps,\Lambda)$, in which the morphisms 
\[ (E,q_E)\to (E',q_{E'}) \]
are
the linear maps $f:E\to E'$ preserving the form, i.e.~such that the form
\[ f^*q_{E'}=q_{E'}(f(-),f(-)) \] 
coincides with $q_E$
as forms on $E$. 
This condition is well-defined, and is equivalent to preserving both $\omega_q$ and $Q_q$ by \cite[Thm~2.5]{Forms}. 
The morphisms in $\textbf{Forms}(\F,\sigma,\eps,\Lambda)$ are called \textit{isometries}.
The category of formed spaces is symmetric monoidal under the direct sum operation
\[ (E,q_E)\oplus(E',q_{E'}) = (E\oplus E',q_{E\oplus E'}) \]
where
\[ q_{E\oplus E'}((v,w),(v',w'))=q_E(v,v')+q_{E'}(w,w'). \]
The monoidal  category of  formed spaces over a finite
field $\F$ with standard parameters $(\s,\eps,\Lambda)$ is studied in quite a lot of details in \cite[Chap II]{FP}.

\medskip

Given a vector space $E$, let $^\s\! E^*$ denote the vector space of
$\s$-skew-linear maps $E\to \F$, that is, additive maps $f:E\to \F$
such that $f(av)=\bar a f(v)$, with vector space structure defined
pointwise from that of $\F$.

\begin{defn}\label{def:kerneletc} Let $E=(E,q)$ be a formed space. \\
(1) The {\em kernel} of $E$ is defined to be the kernel
$\K(E) = \ker(\flat_q)$
of the associated linear map
\[ \flat_q: E\to {^\sigma\! E^*},\ v\mapsto\omega_q(-,v). \]
We say that $(E,q)$ is {\em non-degenerate} if $\K(E)=0$.

\smallskip

\noindent
(2) The \textit{orthogonal complement} 
$U^\perp$ of a subspace $U\leq E$ is defined to be the subspace consisting of all $v\in E$ such that $\omega(v,U)=0$.  That is, $U^\perp$ is the kernel of the composition
\[ E\xrightarrow{\flat_q}{^\sigma\! E^*}\xrightarrow{incl^*}{^\sigma U^*}. \]

\smallskip

\noindent
(3) The \textit{radical} of $E$ is defined to be the set
\[ \R(E)=\set{v\in\K(E)}{Q_q(v)=0}. \]

\smallskip

\noindent
(4) A subspace $U$ of $E$ is called \textit{isotropic} if $q|_U=0$ (or equivalently, $\omega_q|_U=0$ and $Q_q|_U=0$).

\smallskip

\noindent
(5) The {\em genus} $\g(E)$ is defined to be the maximum dimension of an isotropic subspace minus the dimension of the radical $\R(E)$.
\end{defn}

\begin{rem}
(1) The radical $\R(E)$ is  a subspace of $E$, because $Q_q$ is additive on $\K(E)$ and satisfies $Q_q(cv)=\bar ccQ_q(v)$ for $c\in\F,v\in E$.
The radical is isotropic.
If the characteristic of $\F$ is not 2, then $\R(E)=\K(E)$, since equation (\ref{equ:Qomega}) in that case gives that $\omega_q|_{\K(E)}=0$ implies $Q_q|_{\K(E)}=0$. 

\smallskip

\noindent
(2) Note that orthogonality is a symmetric relation: $\omega_q(v,w)=0$ if and only if $\omega_q(w,v)=0$, but note also that the orthogonal complement $U^\perp$ defined above will usually not be a complement in the sense of vector spaces. In fact, we will typically be interested in $U^\perp$ when $U$ is isotropic, in which case we actually have $U\subset U^\perp$! Section~\ref{sec:complements} in the appendix is concerned with properties of orthogonal complements that are used in the paper. 
\end{rem}

\begin{lemma}\label{lem:props of sum}
Let $E=(E,q)$ and $E'=(E',q')$ be two formed spaces.  Then
\begin{enumerate}
\item\label{lem:props of sum:K} $\K(E\oplus E') = \K(E)\oplus\K(E')$
\item\label{lem:props of sum:R} If $\K(E')=0$, then $\R(E\oplus E')=\R(E)$.
\item\label{lem:props of sum:g} $\g(E\oplus E') \geq \g(E)+\g(E')$.
\end{enumerate}
\end{lemma}
\begin{proof}
Statement~(\ref{lem:props of sum:K}) follows from the fact that, for all $e\in E,e'\in E'$, 
\[ \flat_{q\op q'}(e+e')=\flat_q(e)+\flat_q(e')\in {^\sigma\!
    E^*}\oplus{^\sigma\! E'^*}. \]
Statement~(\ref{lem:props of sum:R}) follows from (\ref{lem:props of sum:K}) and (\ref{lem:props of sum:g}) for the fact that $U\op U'$ is isotropic whenever $U\leq E$ and $U'\leq E'$ both are.
\end{proof}

\begin{rem}
It is not true in general that $\R(E\oplus E') = \R(E)\oplus\R(E')$, or that
$\g(E\oplus E') = \g(E)+\g(E')$.  Take for example the case $(\sigma,\eps,\Lambda)=(\textrm{id},1,0)$ and let $E=(\F,q)$ with $q(a,b)=ab$, and $E'=(\F,q')$ with $q'(a,b)=-ab$.
If the characteristic of $\F$ is two, this is a counterexample to the first statement as $\R(E\op E') = \langle(1,1)\rangle$.
If not, it is a counterexample to the second one as $\langle(1,1)\rangle$ is isotropic.
\end{rem}

\subsection{The hyperbolic formed space $\HH$}
 
We will be particularly interested in the following basic two-dimensional hyperbolic formed space:
\begin{defn}\label{def:hyperbolic plane}
Given $(\F,\s)$, define a formed space
$\HH=(\F^2,q_{\HH})$ by setting  \[\quad q_{\HH}(v,w)=\overline{v_1}w_2. \] 
\end{defn}

The following result gives properties of the formed space $\HH$, which show that it is a natural ``atomic object'' to stabilize with. These properties will be used in Section~\ref{sec:stability}, where our main homological stability result is proved.

\begin{lemma}\label{lem:properties of HH}
The formed space $\HH=(\F^2,q_{\HH})$ of Definition~\ref{def:hyperbolic plane}  has kernel $\K(\HH)=\R(\HH)=0$ and genus $\g(\HH)=1$ and, given any formed space $(E,q)$, 
we have that $$\R(E\op \HH)=\R(E) \ \ \ \textrm{and}\ \ \  \g(E\op\HH)=\g(E)+1.$$ 
In particular, $\g(\HH^{\op n})=n$. 
\end{lemma}
\begin{proof}
Let $e_1,e_2$ be the standard basis for $\HH$, 
and $\alpha_1,\alpha_2$ the dual basis for ${^\sigma\!\HH^*}$.  
We have 
\[ \omega_q(v,w) = \overline{v_1}w_2 + \eps\overline{v_2}w_1, \]
so $\flat_q(e_1)=\eps\alpha_2 \ \text{and}\ \flat_q(e_2)=\alpha_1$.
Hence $\flat_q$ is bijective and $\K(\HH)=0$.
That shows that $\HH$ itself is not isotropic, so $\g(\HH)<2$.  But the span of $e_1$ is isotropic, so $\g(\HH)=1$.

Let $(E,q)$ be a formed space. By Lemma~\ref{lem:props of sum} and the above calculation, \hbox{$\R(E\op \HH)=\R(E)$} and $g((E\op \HH)\ge g(E)+1$. In Lemma~\ref{lem:dimmax}, we will see that all the maximal isotropic subspaces of $E\op \HH$ have the same dimension, so it is enough to consider the dimension of a maximal isotropic subpace containing $U\op L$ for $U\le E$ and $L\le \HH$ maximal isotropic. Now, using Lemma~\ref{lem:capperp}, we have $$(U\op L)^\perp\, =\, U^\perp\cap L^\perp\, =\, (U^\perp\cap E) \, \op\,  (L^\perp\cap \HH)\, =\, (U^\perp\cap E) \, \op\,  L$$
from which it follows that $(U\op L)^\perp/U\op L \, =\,  (U^\perp \cap E)/U$. 
By maximality of $U$, we have that the right hand side has no isotropic vector. Hence the same holds for the left hand side. 
\end{proof}

Finally, we show how $\HH$ relates to other formed spaces. 
Let $\textbf{IVect}_\F$ denote the category of 
 finite-dimensional $\F$-vector spaces and isomorphisms, and  $\textbf{IForms}(\F,\sigma,\eps,\Lambda)$ the category of formed spaces and bijective isometries. 
There are functors 
$$\xymatrix{F: \textbf{IForms}(\F,\sigma,\eps,\Lambda) \  \ \ar@<1ex>[rr] && \ar@<1ex>[ll] \ \ \textbf{IVect}_\F : \Phi}$$
with $F$ the forgetful functor and $\Phi$ defined on  objects by 
\[ \Phi(V)=(V\oplus {^\s V^*},q) \ \ \ \textrm{with }\ \  q((v,\alpha),(w,\beta))=\beta(v),\]
and on morphisms by $$\Phi(g)=g\op (g^{-1})^*.$$

The isomorphism $\F^2\cong \F\op {^\s\F^*}$ taking $(a,b)$ to $(a,b\s)$ defines an isomorphism of formed spaces
\[ \HH\iso\Phi(\F). \]
More generally, choosing a basis for $V$ determines an isomorphism
\[ \Phi(V)\iso\HH^{\op \dim V}. \]

The following proposition will be used in Section~\ref{sec:eucli}: 

\begin{prop}\label{prop:hat+-} 
Let $(E,q)$ be a nondegenerate formed space over a field $\F$.
Then there is an isomorphism 
\[ 
(E,q)\oplus(E,-q)\iso \HH^{\op \dim E}, \]
which is natural with respect to bijective isometries of $E$. 
\end{prop}

This result is essentially a special case of \cite[Thm 3]{Wallaxiomatic}. 
We give a translation of Wall's proof to our set-up.

\begin{proof} We show that $(E,q)\oplus(E,-q)\iso \Phi(E)$. 
Define  $\phi_E\colon E\oplus E\to E\oplus\, ^{\s}E^*$ by $$\phi_E(u,v)=(u-g_q(u-v),\flat_q(u-v))$$
for $\flat_q$ as in Definition~\ref{def:kerneletc}   
and
$g_q: E \ \stackrel{q(-,-)}\rar\  {^\s E^*} \ \stackrel{\flat_q^{-1}}\rar\  E$ 
taking $v$ to $v'$ if $q(w,v)=\omega_q(w,v')$ for every $w\in E$. 
One checks that $\phi_E$ is a vector space isomorphism and, writing $\Phi(E)=(E\oplus\, ^{\s}E^*,\bar q)$, 
that it is an isometry by checking that the form $\phi_E^*\bar q-(q\op (-q))\in X(E\op E,\s,\eps,\Lambda_{min})\le  X(E\op E,\s,\eps,\Lambda)$.  

For naturality, suppose that $\alpha:(E,q)\to (E',q')$ is a bijective isometry. We need to check that $(\alpha\op (\alpha^{-1})^*)\circ \phi_E=\phi_{E'}\circ (\alpha\op \alpha)$. This 
is true because $\alpha$ preserves $q$ and $\omega_q$, and hence also $\flat_q$ and $g_q$.  
\end{proof}

For further use, we note that the formed spaces $(E,q)$ and $(E,-q)$ (or their squares) are isomorphic under certain assumptions on the field: 

\begin{lemma}\label{lem:--++}
Let $(E,q)$ be a formed space over a field $\F$. \\ 
(1)~If there exists $a\in \F$ such that $\bar a a=-1$, then $(E,q)\cong (E,-q)$. 

\smallskip

\noindent
(2)~If there exists $a,b\in \F$ such that $\bar a a+\bar b b=-1$, then $(E\op E,-q\op -q)\cong (E\op E,q\op q)$.
\end{lemma}

\begin{proof}
To prove (1), we use the isomorphism $\alpha: E\to E$ given by multiplication by $a$ for $a$ such that $\bar a a=-1$. Then $q(\alpha(v),\alpha(w))=-q(v,w)$ by our choice of $a$. 

For (2), we pick $a,b\in \F$ such that $\bar a a+\bar b b=-1$ and let 
\[ \beta=\begin{bmatrix}\bar a&-\bar b\,\\b&a\end{bmatrix}\colon E^{\op 2}\to E^{\op 2}. \]
It is bijective since it has inverse
\[ \begin{bmatrix} -a&-\bar b\,\\b&-\bar a\,\end{bmatrix}. \]
It induces an isometry $(E\op E,-q\op -q)\to (E\op E,q\op q)$  since
\begin{align*}
&q\op q\big((\bar av-\bar bw,bv+aw),(\bar av'-\bar bw',bv'+aw')\big) \\
&= q(\bar av-\bar bw,\bar av'-\bar bw') + q(bv+aw,bv'+aw') = -q(v,v')-q(w,w'). \qedhere
\end{align*}
\end{proof}

\section{The buildings of a vector space}\label{sec:buildings}

In this section, we define and study the buildings and relative
buildings of subspaces of a vector space, and collect their basic
properties. 
We show that the known  vanishing of the coinvariants of the top homology of the standard building also holds in the relative case. 
Before introducing the building, we recall some basic definitions about posets.

\medskip

The dimension of a poset is defined to be the maximal length $d$ of a chain $a_0<a_1<\dots<a_d$ of elements in the poset. We will here only consider posets in which any maximal chain has the same finite length; these are called {\em graded posets}.  In a graded poset, we define the {\em rank} of an element $a$ as the length of any maximal chain ending at $a$. 

A poset $\pP$ is said to be {\em $n$-connected} if its realization is
$n$-connected, and 
{\em spherical} if it is $(\dim \pP-1)$-connected.
For $x,y\in \pP\cup \{-\infty,\infty\}$, we define the {\em interval $(x,y)$} as the subposet  $$(x,y)=\{z\in \pP\ |\ x<z<y\}\ \ \ \textrm{with in particular}\ \   (-\infty,y)=\pP_{<y} \  \ \textrm{and}\ \  (x,\infty)=\pP_{>x}.$$

\begin{Def}\label{def:CM}
A graded poset $\pP$ is {\em Cohen-Macaulay} if for every $-\infty\le x\le y\le \infty$, the interval $(x,y)$ is spherical, i.e.~if $\pi_i(|(x,y)|)=0$ for $i<\dim(x,y)$.
\end{Def}
Being Cohen-Macaulay implies being spherical, which is the case $x=-\infty$ and $y=\infty$.

\medskip

We recall the Tits building of a vector space, and define its relative analogue that will be relevant to us.

\begin{Def}\label{def:GL building}
Let $V$ be a finite-dimensional vector space over a field $\F$. Define the {\em building} of $V$ as 
\[ \PP(V)=\{0<W<V\}, \]
to be the poset of nontrivial proper subspaces of $V$.
For vector spaces $V_0\le V$, we define the {\em relative building}  as 
\[ \PP(V, V_0) = \set{W<V}{W+V_0=V}. \]
\end{Def}

The building and relative building have the following properties: 

\begin{thm}\label{thm:GL building properties} 
Let $V$ be a finite-dimensional vector space and $V_0$ a subspace.  Then $\PP(V)$ and $\PP(V,V_0)$ are Cohen-Macaulay and for  $W\in\PP(V)$  and  $W'\in\PP(V,V_0)$, 
$$\begin{array}{rcl|rcl}
\dim\PP(V) &=& \dim V-2 & \dim\PP(V,V_0) &=& \dim V_0-1, \\
\PP(V)_{<W} &\cong & \PP(W) & \PP(V,V_0)_{<W'} &\cong & \PP(W',W'\cap V_0), \\
\PP(V)_{>W} &\cong & \PP(V/W) & \PP(V,V_0)_{>W'} &\cong & \PP(V/W'). \\
\rank(W) &=& \dim W-1 & \rank(W') &=& \dim(W'\cap V_0)
\end{array}$$
\end{thm}

\begin{proof}
The Cohen-Macaulay property of $\PP(V)$ is the Solomon-Tits theorem; for instance see~\cite[p.~118]{QuillenHomotopy}.
For $\PP(V,V_0)$, Vogtmann~\cite[Cor.~1.3]{Vogtmann} shows that it is spherical, which implies the Cohen-Macaulay property by the non-relative case, once we have checked that the intervals are non-relative buildings as stated.

This isomorphism $\PP(V,V_0)_{<W'}\cong \PP(W',W'\cap V_0)$ takes $X\leq W'$ to itself. This is well-defined because if $X\leq W'$,
\[ (X+V_0)\cap W' = X+(V_0\cap W'), \]
so the latter equals $W'$ if and only if
$X+V_0\geq W'$, which holds 
if and only if
$X+V_0=V$ since $W'+V_0=V$. The isomorphism $\PP(V,V_0)_{>W'}\cong\PP(V/W')$ is given by taking $X\in \PP(V,V_0)_{>W'}$ to $X/W'\in \PP(V/W')$. 
The statements for $P(V)$ are checked likewise. 
\end{proof}

In Section~\ref{sec:stability}, we will also need augmented variants of the buildings $\PP(V)$ and $\PP(V,V_0)$:
\begin{Def}\label{def:GL relative building bar}
For $V_0\le V$, define 
\[ \bP(V)=\{0<W\leq V\}, \]
\[ \bP(V, V_0) = \set{W\leq V}{W+V_0=V}. \]
\end{Def}
These differ from the buildings we studied so far by the addition of a maximal element. (In particular they are contractible.) 
The addition of a maximal element increases the dimensions by one, and does not affect the Cohen-Macaulay property:

\begin{thm}\label{thm:barCM}
For $V_0\le V$ finite-dimensional vector spaces, the buildings $\bP(V)$ and $\bP(V,V_0)$ are Cohen-Macaulay, and:
$$\begin{array}{rcl|rcl}
\dim\bP(V) &=& \dim V-1 & \dim\bP(V,V_0) &=& \dim V_0, \\
\bP(V)_{<W} &\cong & \PP(W) & \bP(V,V_0)_{<W} &\cong & \PP(W,W\cap V_0), \\
\bP(V)_{>W} &\cong & \bP(V/W) & \bP(V,V_0)_{>W} &\cong & \bP(V/W).
\end{array}$$
\end{thm}

\begin{proof}
This follows from Theorem~\ref{thm:GL building properties} and the fact that any interval containing the added maximal element is contractible. 
\end{proof}

\subsection{A join decomposition of the relative building}\label{sec:join}

The reduced homology of the building,
\[ \St(V)=\tilde H_{\dim V-2}(\PP(V)), \]
is known as the {\em Steinberg module} over $\GL(V)$.
It is well-known that the $\G(V)$-coinvariants of this module vanish when $\dim V\geq2$:
If $\dim(V)\geq2$, then
\[ \St(V)_{\G(V)}=0;  \]
this is e.g.~a special case of~\cite[Thm.~1.1]{APS}. 
In Section~\ref{sec:vanishing of relative coinvariants}, we will prove the analogous result for the relative building $\PP(V,V_0)$, using a join decomposition of the building which we now describe.

Given two posets $\pP$ and $\mathcal{Q}$, their {\em join} $\pP*\mathcal{Q}$ is the poset 
which, as a set, is the disjoint union $\pP\sqcup \mathcal{Q}$, with the usual order relations of $\pP$ and $\mathcal{Q}$, plus the additional relation that $a<b$ for all $a\in \pP$ and all $b\in \mathcal{Q}$.  
Note that the realization of a join of two posets is the (topological) join of their realizations.

\begin{thm}\label{thm:join decomposition} 
For any subspace $U\le V_0$, there is poset map 
\[ \phi:\PP(V,V_0)\to \PP(V/U,V_0/U)\ast \PP(V,U) \]
which is natural with respect to triples $(V\geq V_0\geq U)$ and induces a homotopy equivalence on realizations.
\end{thm}

This result is stated  in Quillen's notes \cite{QuillenNotes} with $U$ a line. See also \cite[Lem 3.8]{GKRW}. 

\medskip

Let
\[ \St(V,V_0):=\tilde H_{\dim V_0-1}(\PP(V,V_0)) \]
denote the {\em relative Steinberg module}.  The proposition has the following consequence: 
\begin{cor}\label{cor:join-decomp-homology}
There is an isomorphism
\[ \St(V,V_0)\iso \St(V/U,V_0/U)\otimes_\bbZ \St(V,U)\]
which is natural with respect to triples $(V\geq V_0\geq U)$. 
\end{cor}

\begin{proof}
This follows from Theorem~\ref{thm:join decomposition} 
using the K\"unneth Theorem 
for reduced homology of joins~\cite[eq.~2.3]{Hall},
along with the fact that both $\PP(V/U,V_0/U)$ and $\PP(V,U)$ are spherical (Theorem~\ref{thm:GL building properties}).
\end{proof}

To prove Theorem~\ref{thm:join decomposition}, we will make use of the following (probably well-known) lemma:

\begin{lem}\label{lem:monotonic}
Let $f:\pP\to \pP$ be a poset map which is weakly monotonic, i.e., $f(x)\geq x$ for all $x\in \pP$ or $f(x)\leq x$ for all $x\in \pP$.  Then $f$ induces a homotopy equivalence
\[ |\pP|\xrightarrow{\sim}|\im(f)|. \]
\end{lem}
\begin{proof}
Let $i:\im(f)\to \pP$ be the inclusion.  Then $|f|$ and $|i|$ are homotopy inverses, because
$f\circ i\geq\id_{\im(f)}$ and $i\circ f\geq\id_{\pP}$ in the first case, and the same but reversing the inequalities in the second case.
\end{proof}

\begin{proof}[Proof of Theorem~\ref{thm:join decomposition}]
The map $\phi\colon \PP(V,V_0)\to \PP(V/U,V_0/U)\ast \PP(V,U)$ is defined by
\[ \phi(W)=\begin{cases}
W/U\in \PP(V/U,V_0/U) & \textrm{if }W+U<V, \\
W\in \PP(V,U) & \textrm{if }W+U=V.
\end{cases} \]
The map $\phi$ is clearly well-defined as $W+V_0=V$ implies $W/U+V_0/U=V/U$. 
Let us verify that it is a poset map.  Suppose that $W\leq W'$.
It is clear that $\phi(W)\leq\phi(W')$ if both $W+U=V$ and $W'+U=V$, or if neither hold.  If exactly one holds, it must be that $W'+U=V$ and $W+U<V$: in this case $\phi(W)<\phi(W')$ in the join poset, by definition thereof.

As it does not appear straightforward to describe a homotopy inverse to $\phi$, we will show that $\phi$ is a homotopy equivalence using the 
Quillen fiber lemma \cite[Prop.~1.6]{QuillenHomotopy}. 
In our situation, $\mathcal{Q}=\PP(V/U,V_0/U)\ast \PP(V,U)$, so there are two cases.  If $q=W\in \PP(V/U,V_0/U)$,
then $\phi\inv(\mathcal{Q}_{\leq W})$ is contractible because it has a maximal element: $W+U$.  (In this case, $W+U<V$.)

Next suppose that $q=W\in \PP(V,U)$.  That is, assume $W+U=V$.
As a subposet of $\PP(V,V_0)$,
\[ \phi\inv(\mathcal{Q}_{\leq W})=\set{X\in \PP(V,V_0)}{X+U<V\textrm{ or }X\leq W} \]
We will show this poset is contractible by replacing it with its image under a monotonic self-map twice, which does not change the homotopy type by Lemma~\ref{lem:monotonic}, 
after which it will become clearly contractible.
Define \[ f:\phi\inv(\mathcal{Q}_{\leq W})\to\phi\inv(\mathcal{Q}_{\leq W}) \]
by 
\[ f(X)=\begin{cases}
X & \textrm{if } X\leq W. \\
X+U & \textrm{if } X+U<V\textrm{ but }X\not\leq W ,
\end{cases} \]
In checking that $f$ is a poset map, the only interesting case is when $X\leq X'$, and $X\leq W$ but $X'\not\leq W$ (the other way around being impossible).  In that case
\[ f(X)=X\leq X'+U=f(X'), \]
as needed. Now $f$ is clearly monotonic, so $\phi\inv(\mathcal{Q}_{\leq W})\simeq \im(f)$ by Lemma~\ref{lem:monotonic}. 

We have that 
\[ \im(f)=\set{Y\in \PP(V,V_0)}{U\leq Y\textrm{ or }Y\leq W}. \]
Indeed,  if $U\leq Y$ or $Y\leq W$ with $Y\in \PP(V,V_0)$, then $Y\in\phi\inv(\mathcal{Q}_{\leq W})$ and $Y=f(Y)$.
Now we define a new map
\[ g:\im(f)\to\im(f) \]
by  $g(Y)=Y\cap W$. 
Here it would be clear that $g$ is a monotone poset map.  However, this time we must verify that $g(Y)\in \PP(V,V_0)$ so that $g(Y)$ is indeed inside $\im(f)$. So we need to check that 
\[ (Y\cap W)+V_0=V. \]
In case $Y\leq W$, this is obvious (as $Y+V_0=V$), so assume instead that $U\leq Y$.  Now recall that $W+U=V$.
Then we have
\begin{align*}
(Y\cap W)+V_0 &= (Y\cap W)+U+V_0 \\
&= Y\cap(W+U)+V_0 \ = Y+V_0 \ = V.
\end{align*}
So $g$ is a well-defined monotonic map.  The image is
$\im(g)=\set{X\in\PP(V,V_0)}{X\leq W}$
which is contractible because it has a maximal element: $W$.
\end{proof}

\subsection{Vanishing of relative Steinberg coinvariants}\label{sec:vanishing of relative coinvariants}
The goal of this section is to prove a relative analogue of the vanishing of the Steinberg invariants. 
The relative Steinberg module is the module $ \St(V,V_0):=\tilde H_{\dim V_0-1}(\PP(V,V_0)) $ considered in the previous section. 
Before stating the relative vanishing result, we define the groups acting on this relative module.   

\medskip

Let $V_0\le V$ be a subspace as above, and consider  the group
\begin{align*}
\A^T(V,V_0)&=\set{g\in \GL(V)}{gV_0=V_0\textrm{ and }g\textrm{ induces }\id_{V/V_0} \textrm{ on } V/V_0}. 
\end{align*}
This group can be equivalently described as that subgroup of $\GL(V)$ of elements $g$ such that $g(v)-v\in V_0$ for all $v\in V$. And if we identify $V$ with $V/V_0\op V_0$, in
block matrix form it is the group 
$$\begin{bmatrix}I_{V/V_0}&0 \\ \ast&\GL(V_0)\end{bmatrix}.$$
This group acts on the relative complex $\PP(V,V_0)$ as it preserves $V_0$. The main result of the section is the following: 

\begin{thm}\label{thm:relcoinvars}
If $0<V_0\leq V$ are vector spaces over a field $\F\neq\F_2$, then
\[ \St(V,V_0)_{\A^T(V,V_0)}=0. \]
\end{thm}

In the case of $\F=\F_{p^r}$ a finie field, Lemma 5.2 in \cite{GKRW} gives the more general vanishing of $H_d(\A^T(V,V_0);\St(V,V_0)\otimes \F_p)$ in the range $d<r(p-1)-1$. 

\begin{proof}
Let $n=\dim V$ and assume first that $V_0=L$ is a line.  In that case, $\PP(V,L)$ is discrete (the set of complementary hyperplanes), and its reduced homology is given by a short exact sequence
\[ 0\to \St(V,L)\to \bbZ \PP(V,L)\to\bbZ\to0. \]
We have \[ \A^T(V,L)\iso\begin{bmatrix}I_{n-1}&0\\\ast&\GL_1\F\end{bmatrix}
=\F^{n-1}\semidirect \F^\times. \]
Now, $\F^{n-1}\semidirect \F^\times$ acts on $\bbZ \PP(V,L)$ via its action on the set $\PP(V,L)$ of hyperplanes complementary to $L$.  The latter action is transitive with isotropy subgroup $\F^\times$.
In other words,
\[ \bbZ \PP(V,L)\iso\bbZ\uparrow_{\F^\times}^{\F^{n-1}\semidirect \F^\times}. \]
So by Shapiro's lemma
\[ H_i(\F^{n-1}\semidirect \F^\times;\bbZ \PP(V,L))=H_i(\F^\times;\bbZ). \]
Hence, by applying $H_*(\F^{n-1}\semidirect \F^\times;-)$ to the above short exact sequence
we get a long exact sequence
\begin{align*}
\cdots \to H_1(\F^\times;\bbZ))\xrightarrow{a}
H_1(\F^{n-1}\semidirect \F^\times;\bbZ) &\xrightarrow{b}
H_0(\F^{n-1}\semidirect \F^\times;\St(V,L)) \\
&\xrightarrow{c}
H_0(\F^\times;\bbZ)\xrightarrow{d}
H_0(\F^{n-1}\semidirect \F^\times;\bbZ) \to0.
\end{align*}
Clearly $d$ is an isomorphism, so $c=0$ and $b$ is surjective.  It will suffice to show that $a$ is surjective; then $b=0$ and we can conclude $H_0(\F^{n-1}\semidirect \F^\times;\St(V,L))=0$ as desired.

To that end, consider the spectral sequence for the group extension
\[ 1\to \F^{n-1} \to \F^{n-1}\semidirect \F^\times\to \F^\times\to1. \]
It has
\begin{align*}
E^2_{01}&=H_0(\F^\times; H_1(\F^{n-1};\bbZ)) =(\F^{n-1})_{\F^\times}.
\end{align*}
Now the action of $\F^\times$ on $\F^{n-1}$ is via its $\F$--vector space structure, and we claim that for any $\F$--vector space $M$, we have 
that $M_{\F^\times}=0$.
Indeed, since $\F\neq\F_2$, there exists a scalar $x\neq0,1$.  That is: $x,(1-x)\in \F^\times$.
Then for any $y\in M_{\F^\times}$,
\[ y=x\cdot y+(1-x)\cdot y=y-y=0. \]
It now follows that projection $\F^{n-1}\semidirect \F^\times\to \F^\times$ and hence also inclusion $\F^\times\to \F^{n-1}\semidirect \F^\times$ induces an isomorphism on $H_1(-;\bbZ)$.  So the map $a$ above is indeed surjective as we needed.

We consider now the case where $\dim(V_0)>1$.  Pick a line $L<V_0$ and note that 
\[ \A^T(V,L)\leq\A^T(V,V_0). \]
Hence it suffices to show that $\St(V,V_0)_{\A^T(V,L)}=0$.  But by Corollary~\ref{cor:join-decomp-homology},
\[ \St(V,V_0)\iso \St(V/L,V_0/L)\otimes \St(V,L), \]
equivariantly with respect to $\A^T(V,L)$, whose elements preserve both $L$ and $V_0$.
Since $\A^T(V,L)$ acts trivially on $\St(V/L,V_0/L)$,
\[ \St(V,V_0)_{\A^T(V,L)}\iso \St(V/L,V_0/L)\otimes \St(V,L)_{\A^T(V,L)}. \]
However, $ \St(V,L)_{\A^T(V,L)}=0$
by the previously checked case.
\end{proof}

\section{The buildings of a formed space}\label{sec:isotropicbuilding}
Recall that a formed space $E=(E,q)$ is a vector space $E$ equipped with a form $q$. 
We will study the automorphism groups of formed spaces through their building of isotropic subspaces. 
Just as in the case of vector spaces, we will need a relative version of the building.  
We start by defining and giving the first basic properties of the buildings. In Section~\ref{sec:isoCMboth}, we will prove that the building and relative building are Cohen-Macaulay.

\smallskip

Recall from Definition~\ref{def:kerneletc} that $\R(E)$ denotes the radical of $E$, and that a subspace $U\le E$ is isotropic if $q$ vanishes on $U$.

\begin{Def}\label{def:iso buildings}
Let $E=(E,q)$ be a formed space as in Definition~\ref{def:form}.
We define its building to be the poset
\[ \PPi(E) = \set{W<E}{W\textrm{ isotropic},\ \R(E)<W} \]
ordered by inclusion.
For an isotropic subspace $U\in\PPi(E)$, we define the relative building to be
\[ \PPi(E,U) = \set{W<E}{W\textrm{ isotropic},\ \R(E)<W,\ W+U^\perp=E}. \]
\end{Def}
Note that $\PPi(E,U)$ is an upper-closed subposet of $\PPi(E)$, 
i.e.~$\PPi(E,U)\subset\PPi(E)$ and for any  $W\in \PPi(E,U)$, the upper interval $\PPi(E)_{>W}\subset \PPi(E,U)$.

Note also that the buildings $\PP(V)$ and $\PP(V,V_0)$ associated to vector spaces are not special cases of this construction.  Indeed, if $E=(E,0)$ is a formed space with trivial form, then $\R(E)=E$ and $\PPi(E)$ is empty.

\begin{lem}\label{lem:modR(E)} The buildings of isotropic subspaces  are insensitive to the radical of $E$:
\begin{align*}
\PPi(E) \cong  \PPi(E/\R(E)) \ \ \ \textrm{and}\ \ \ \PPi(E,U) \cong  \PPi(E/\R(E), U/\R(E))
\end{align*}
where $E/\R(E)$ is equipped with the unique form such that the projection $E\to E/\R(E)$ is an isometry. 
\end{lem}
\begin{proof}
The existence of a unique from on $E/\R(E)$ such that the projection $E\to E/\R(E)$ is an isometry is given by  Lemma~\ref{lem:induced form1}. As the projection is an isometry, it takes isotropic subspaces to isotropic subspaces. And if $U/\R(E)\subset E/\R(E)$ is isotropic, then so is $U+\R(E)$, as $\R(E)$ is the radical of $E$. Finally, for $\R(E)<W\le E$ isotropic, $W+U^\perp=E$ if and only if 
$W/\R(E) + (U/\R(E))^\perp=E/\R(E)$. 
\end{proof}

The following result is essentially~\cite[Prop 1]{Tits}:
\begin{lem}\label{lem:dimmax} In a formed space $(E,q)$, all maximal isotropic subspaces have the same dimension.
\end{lem}

\begin{proof} Suppose that there are two maximal isotropic subspaces $A,B\subset V$ with $\dim A<\dim B$.  Both $A$ and $B$ contain $\R(E)$, since e.g.~$A+\R(E)$ is isotropic.  Since they are also isotropic, we have
\[ A\cap\K(E)=\R(E)=B\cap\K(E). \]
Choose an injection $f:A\to B$ which is the identity on $\R(E)$.
Then we have a bijective map $A\to\im(f)$ which is trivially an isometry, and restricts to the identity on $\R(E)$.  So by Witt's Lemma \cite[Thm 3.4]{Forms} it may be extended to a bijective isometry 
$\tilde{f}$ of $E$. But then $\tilde{f}\inv(B)$ is an isotropic subspace properly containing $A$, contradicting maximality.
\end{proof}

We compute the dimensions of the buildings:
\begin{lemma}\label{lem:building dimensions}
Let $(E,q)$ be a formed space, and let $U\in\PPi(E)$.  Then
both $\PPi(E)$ and $\PPi(E,U)$ are graded posets, with dimensions
\begin{align*}
\dim\PPi(E) = \g(E)-1, \ \  \ \textrm{and} \ \ \ \dim\PPi(E,U) = \g(E)-\dim U+\dim\R(E) = \g(U^\perp).
\end{align*}
\end{lemma}

\begin{proof}
For $\PPi(E)$, this follows from the definitions and Lemma~\ref{lem:dimmax}.
For $\PPi(E,U)$, observe that maximal elements of $\PPi(E,U)$ are in particular maximal isotropic subspaces of $E$, having dimension $\g(E)+\dim\R(E)$.
By Lemma~\ref{lem:isotropic complement}, 
the minimal elements have the form $L\op\R(E)$ where $E=L\op U^\perp$.  By Lemma~\ref{lem:dimension of perp} and using that $\R(E)\leq U$, 
we have 
$$\dim(U^\perp)=\dim(E)-\dim(U)+\dim(\R(E))$$
So the minimal elements have dimension
$\dim U$, and 
\[ \dim\PPi(E,U) = \g(E)+\dim\R(E)-\dim U. \]
The last equality in the statement is given by Proposition~\ref{prop:induced form}. 
\end{proof}

The following result shows that the lower intervals of the buildings  and relative buildings of isotropic subspaces are buildings and relative buildings of vector spaces, and that the upper intervals are isotropic buildings of smaller formed spaces. 

\begin{lemma}\label{lem:intervals}
Let $(E,q)$ be a formed space, and let $W,U\in \PPi(E)$, and $W'\in\PPi(E,U)$. Then
$$\begin{array}{l|l}
(1) \ \PPi(E)_{<W}\cong \PP(W/\R(E))\ &\  (3)\ \PPi(E,U)_{<W'}\cong \PP(W'/\R(E),(U^\perp\cap W')/\R(E))\\
(2) \ \PPi(E)_{>W}\cong \PPi(W^\perp)\  &\  (4) \ \PPi(E,U)_{>W'}\cong \PPi((W')^\perp).
\end{array}$$
Moreover
$$
\rank(W) = \dim W-\dim\R(E)-1 \ \ \ \textrm{and} \ \ \ \rank(W') = \dim(W'\cap U^\perp)-\dim\R(E) = \dim W'-\dim U.
$$
\end{lemma}
\begin{proof}
Statement (1) follows from identifying the right hand side with the set of proper subspaces of $W$ properly containing $\R(E)$.
For (2), $\PPi(W^\perp)$ is the set of isotropic subspaces of $W^\perp$ properly containing $W+\R(E)$ (by Lemma~\ref{lem:radical of perp}), which gives an inclusion of the right hand side into the left hand side in (2).  This inclusion is surjective because any isotropic subspace containing $W$ lies in $W^\perp$, which proves (2).
Statement (3) follows if we can check that for $X<W'$, 
\[ X+U^\perp=E \ \ \ \textrm{if and only if}\ \ \  X+(U^\perp\cap W')=W'. \]
One direction holds because 
$ (X+U^\perp)\cap W'=X+(U^\perp\cap W')$, 
and the other holds since $W'+U^\perp=E$.
Finally (4) follows from (2), since $\PPi(E,U)\subset\PPi(E)$ is upper-closed.

The ranks computations follow from the lower intervals computations and their earlier computation of the dimensions of the lower intervals (Theorem~\ref{thm:GL building properties}), 
using Lemma~\ref{lem:dimension of perp} to compute $\dim(W'\cap U^\perp)$.
\end{proof}

\subsection{Cohen-Macaulay property}\label{sec:isoCMboth}
In this section, we prove that the building and relative building of a formed space are Cohen-Macaulay.
The arguments presented here, most particularly in the relative case,  are adapted from Vogtmann~\cite[Sec~1]{Vogtmann}.   (Vogtmann acknowledges Igusa for some parts of the argument, so our acknowledgements go to him too, by transitivity.) 
The goal of the paper \cite{Vogtmann} is to prove homological
stability in the case of non-degenerate orthogonal groups over a field
of characteristic 0, using the same buildings as those considered
here. However in the process of proving that these buildings are
Cohen-Macaulay, Vogtmann also uses buildings associated to certain
degenerate formed space. The paper \cite{Vogtmann} does not use the
language of formed spaces, and sometimes uses coordinates that are
specific to the example considered. Also a different convention is used
for the interaction with the radical in the degenerate case. But,
modulo this, the argument presented here is a rather direct generalization of the one given in \cite{Vogtmann}.\footnote{Vogtmann only obtains a slope 3 stability for the groups $\OO_{n,n}(\F)$ in \cite{Vogtmann}. The reason for this is that she does not use the optimal spectral sequence argument.}

First we handle the posets $\PPi(E)$. 

\begin{thm}\label{thm:PE is C-M}  
Let $E=(E,q)$ be a formed space.
Then $\PPi(E)$ is Cohen-Macaulay. 
\end{thm}

\begin{proof}
We will prove the theorem by induction on the genus $\g(E)$, starting from $\g(E)=0$, in  which case the result trivially holds.  
So let $E=(E,q)$ be a formed space of genus $\g(E)\ge 1$. By Lemma~\ref{lem:modR(E)}, we may assume that $\R(E)=0$. 

The lower and upper intervals
\begin{align*}
\PPi(E)_{>W} \cong \PPi(W^\perp) \ \ \ \textrm{and}\ \ \ \PPi(E)_{<W} \cong \PP(W/\R(W))
\end{align*}
are already known to be Cohen-Macaulay by Theorem~\ref{thm:GL building properties} and by the induction hypothesis (as $W>0$, and $\g(W^\perp)=\g(E)-\dim W$ by Proposition~\ref{prop:induced form}).
So it suffices to check that $\PPi(E)$ is spherical (of dimension $\g(E)-1$).
Fix a nonzero isotropic vector $u\in E$, which exists by our assumption on the genus.  
Consider the subposet
\[ \pP=\set{W\in\PPi(E)}{W\cap u^\perp>0}. \]
Its realization is contractible, since there is a homotopy between the identity map of $\pP$ and the constant map at $\langle u\rangle$, given via monotone maps by
\[ W \geq W\cap u^\perp \leq W\cap u^\perp+\langle u\rangle \geq \langle u \rangle \]
(using Lemma~\ref{lem:monotonic} for each inequality). 
On the other hand, every $W\in\PPi(E)$ of dimension at least 2 is contained in $\pP$, since $u^\perp$ is a hyperplane in $E$.  Hence the complement $\PPi(E)-\pP$ is discrete, consisting only of minimal elements.
Therefore, by Lemma~\ref{lem:discrete complement}, there is a homotopy equivalence
\[ \abs{\PPi(E)}/\abs{\pP} \simeq \bigvee_{L\in \PPi(E)-P} \Sigma\abs{\PPi(E)_{>L}} \]
Since $\PPi(E)_{>L} = \PPi(L^\perp)$ 
and
$\g(L^\perp) = \g(E)-1$ 
(by Proposition~\ref{prop:induced form}),
we conclude using the induction hypothesis
that the quotient $\abs{\PPi(E)}/\abs{\pP}$
is homotopy-equivalent to a wedge of spheres of dimension $\g(E)-1$.  This proves the theorem, since $\abs{\pP}$ is contractible.
\end{proof}

We now turn to the case of the posets $\PPi(E,U)$.

\begin{thm}\label{thm:PEU is CM}   
For all formed spaces $E$ and all $U\in\PPi(E)$, the poset $\PPi(E,U)$ is Cohen-Macaulay.
\end{thm}

We will prove the theorem by induction on $\dim E$ and we make two different arguments depending on whether $\dim U/\R(E)$ is one or greater than one.  We treat the latter case first.

\subsubsection{Case $\PPi(E,U)$ with $\dim(U/\R(E))>1$} 
\begin{lemma}[Induction step 1]\label{lem:C-M for big U}
Suppose that $\PPi(E,U)$ is Cohen-Macaulay for all formed spaces $E$ with dimension less than $d$ and for all $U\in\PPi(E)$.  Then $\PPi(E,U)$ is Cohen-Macaulay for all formed spaces $E$ with dimension $d$ and for all $U\in\PPi(E)$ with $\dim(U/\R(E))>1$.
\end{lemma}

To prove this first induction step, we will need a variant of the building of isotropic subspaces. For $U\in \PPi(E)$, let 
\[ \PPi'(E,U) = \set{W< E}{W\ \text{ isotropic}\text,\quad  W\cap \R(E)=0\text,\quad W+U^\perp= E}. \]
So the elements $W$ of $\PPi'(E,U)$ are required to be disjoint from the radical instead of containing it. 

\begin{lem}\label{lem:P'} Suppose $\R(E)=0$, $U\in \PPi(E)$ of dimension at least $2$ and $L<U$ is codimension 1 in $U$. Then 
$\PPi'(L^\perp,U)$ is Cohen-Macaulay of the same dimension as $\PPi(L^\perp,U)$.
\end{lem}

Note that $L\leq U$ implies $U\leq L^\perp$, and $U$ contains $\R(L^\perp)=L$, so the above posets make sense. 

\begin{proof}
Define a map
\[ p:\PPi'(L^\perp,U)\to\PPi(L^\perp,U) \]
by setting 
$p(W)=W\op L$. 
This is a strictly increasing, in fact rank preserving, poset map,  the rank in the source being $\dim (W\cap U^\perp)$ and in the target 
$\dim((W\op L)\cap U^\perp)-\dim L$. 
Using~\cite[Cor.~9.7]{QuillenHomotopy}, it follows that the domain will be Cohen-Macaulay of the same dimension as the codomain, if for each $Y\in\PPi(L^\perp,U)$, the poset fiber 
$p\inv(\PPi(L^\perp,U)_{\leq Y})$ under $Y$, is Cohen-Macaulay. 
Now 
\begin{align*}
 p\inv\left(\PPi(L^\perp,U)_{\leq Y}\right) &= \set{W<L^\perp}{W\cap L=0,\ W+L\le Y,\ W+U^\perp=L^\perp} \\
&= \set{W<Y}{W\cap L=0,\ W+(U^\perp\cap Y)=Y} 
\end{align*}
identifies with the general linear building denoted $^LT^{U^\perp\cap Y,Y}$ in \cite[Sec 1]{Vogtmann}. We have $\dim L\ge 1$ and $\dim(U^\perp\cap Y)=\dim Y-1$ because  $U^\perp<L^\perp$ has codimension one by Lemma~\ref{lem:dimension of perp}
(since $U\cap\K(E)=\R(E)=0$).  Hence we can apply~\cite[Cor.~1.5]{Vogtmann}, which gives that $^LT^{U^\perp\cap Y,Y}=p\inv\left(\PPi(L^\perp,U)_{\leq Y}\right)$ 
is Cohen-Macaulay, as required.  
\end{proof}

\begin{proof}[Proof of Lemma~\ref{lem:C-M for big U}]
Suppose $E$ has dimension $d$ and $U\in\PPi(E)$ is such that $\dim(U/\R(E))>1$.  We may assume without loss of generality that
$\R(E)=0$, 
since $\PPi(E,U)\cong\PPi(E/\R(E),U/\R(E))$ (Lemma~\ref{lem:modR(E)}). 

As above, choose $L<U$ of codimension one.  
Because $\dim U>1$ by assumption,
$L$ is nonzero, which  implies that $L^\perp<E$ is a proper subspace by Lemma~\ref{lem:dimension of perp}.
Hence we are given, by assumption, that $\PPi(L^\perp,U)$ is Cohen-Macaulay. By Lemma~\ref{lem:P'}, this implies that also $\PPi'(L^\perp,U)$ is Cohen-Macaulay.
Consider the poset map
\begin{align*}
j:\PPi(E,U) \to \PPi'(L^\perp,U),  \ \ \ \ \ \ W \mapsto W\cap L^\perp.
\end{align*}
The map makes sense because
\[ U^\perp+(W\cap L^\perp) = (U^\perp+W)\cap L^\perp = L^\perp, \]
 since
 $U^\perp<L^\perp$, and $W\cap L=0$,
because  $W+L^\perp=E$, so $\K(E) = W^\perp\cap(L+\K(E))$ contains $W\cap L$, which is isotropic so lies in $\R(E)=0$.
Moreover $j$ is strictly increasing,  
in fact rank preserving, as the rank of $W\in \PPi'(L^\perp,U)$ is $\dim(W\cap U^\perp)$, just like in $\PPi(E,U)$. 

Let $Y\in \PPi'(L^\perp,U)$.
We claim that the corresponding upper poset fiber of $j$
has the form
\[ j\inv\left(\PPi'(L^\perp,U)_{\geq Y}\right) \iso \PPi(E',U') \]
where $E'<E$ and $U'\leq U$ are certain subspaces, with $E'$ proper.
This poset will then by hypothesis be Cohen-Macaulay, 
showing by~\cite[Cor.~9.7]{QuillenHomotopy}
that $\PPi(E,U)$ is indeed Cohen-Macaulay, which will prove the result.

We start by associating an $E'$ and $U'$ to any given  $Y\in \PPi'(L^\perp,U)$. 
As  $Y\op L$ is an isotropic subspace of $L^\perp$,  Lemma~\ref{lem:isotropic complement},
shows that there exists an isotropic space $Z$ such that
\[ L^\perp = (Y\op L)^\perp\op Z. \]
Then
\[ E=Y^\perp\op Z, \]
since $Y^\perp\cap L^\perp = (Y\op L)^\perp$ showing that $Y^\perp\cap Z=0$, 
and $Y^\perp+L^\perp=E$
as $\dim Y^\perp=\dim E-\dim Y$, $\dim L^\perp=\dim E-\dim L$ and $\dim (Y^\perp\cap L^\perp)=\dim E-\dim Y-\dim L$ by Lemma~\ref{lem:dimension of perp} 
and the previous computation, using the fact that $\R(E)=0$ so that $Y\cap \K(E)=0=L\cap \K(E)$. 
Hence $Y\op Z$ is nondegenerate by Lemma~\ref{lem:isotropic complement}, 
and $$E = Y\op Z\op  (Y\op Z)^\perp.$$
Let 
\[ E' = (Y\op Z)^\perp \ \ \ \textrm{and} \ \ \ U' = U\cap E'. \]
Note that  $E'$ is a proper subspace of $E$, since $Y>0$.
We have \[ U'\geq L>0 \]
since $Y,Z\leq L^\perp$.
Because $\R(E')=\R(E)=0$ by Lemma~\ref{lem:perp of nondegenerate},
it makes sense to speak of $\PPi(E',U')$.

Now we define poset maps
\[ f: j\inv\left(\PPi'(L^\perp,U)_{\geq Y}\right) \leftrightarrows \PPi(E',U'): g \]
by setting \[ f(W)=W\cap E' \+ g(X)=X\op Y. \]
The map $f$ makes sense because
$\R(E')=0$
and
\[ W+{U'}^\perp \geq W+U^\perp = E \geq E',  \]
while $(W\cap E')\cap (U')^\perp=(W\cap U^\perp)\cap E'$ since any element of $E'$ is $U\cap (Y\op Z)$. In particular, $W\cap E'<E'$. (This also shows that $j$ is degree preserving.)
The map $g$ makes sense because
\[ X< E'\leq Y^\perp \]
showing that $X\op Y$ is an isotropic strict subspace of $E$, and
\[ (X\op Y)+U^\perp = E\]
 because $Y+U^\perp = L^\perp \geq Z$, 
so
\begin{align*}
X+Y+U^\perp &= X+Y+Z+U^\perp  = X+ E'^\perp + U^\perp= X+{U'}^\perp \geq E',
\end{align*}
while $Y+Z+E'=E$. 

It remains to check that the two maps are inverses of each other. Starting with $X\in \PPi(E',U')$, we have $f\circ g(X)=(X\op Y)\cap E'=X$. And starting with $W\in  j\inv\left(\PPi'(L^\perp,U)_{\geq Y}\right)$, 
we have $W\ge Y$, which implies that $W\le Y^\perp=Y\op E'$ as $E=Z\op Y^\perp$. Hence 
\[ g\circ f(W)=(W\cap E')\op Y =W. \qedhere \]
\end{proof}

\subsubsection{Case $\PPi(E,U)$ with $\dim(U/\R(E))=1$} 
We assume now that $U$ be an isotropic subspace of $E$ such that $\R(E)<U$ has codimension one. 
In this situation, we can pick a $u\in U$ such that
\[ U=\R(E)\op\langle u\rangle. \]
Then $U^\perp=u^\perp<E$ is a hyperplane.

If it happens that $U^\perp=U$, then $\PPi(E,U)$ is zero-dimensional, so Cohen-Macaulay.
Hence we may assume $U^\perp>U$, and fix some isotropic vector
\[ v\in U^\perp-U. \]
(This is possible by Proposition~\ref{prop:induced form}, which insures that $\g(U^\perp)\ge 1$ given that $\g(E)\ge 2$; if $\g(E)=1$ there is nothing to prove.) 
For each constant $\lambda\in\F$, define
\[ v_\lambda = v-\lambda u. \]
This $v_\lambda$ is an isotropic vector not contained in $\R(E)$, so $v_\lambda^\perp$ is likewise  a hyperplane in $E$.  Also,
\[ U\leq v_\lambda^\perp. \]
The hyperplanes $u^\perp$ and $v_\lambda^\perp$ are distinct, because $v_\lambda\notin U=u^{\perp\perp}$.

Now for $U$, $u$, $v$ and $\lambda$ as above, define
\[ \pP_\lambda = \pP_\lambda(E,u,v)
= \set{W\in\PPi(E,U)}{(W\cap v_\lambda^\perp)+U^\perp=E}. \]
The idea of the proof of the Cohen-Macaulay property of $\PPi(E,U)$ in the present case, is that $\PPi(E,U)$ is covered by the subposets $\pP_\lambda$ for various $\lambda$, while they and their intersections are already known to be Cohen-Macaulay by induction. We prove these facts in the next three lemmas. 

\begin{lemma}\label{lem:union of Plambda}
For all $W\in\PPi(E,U)$, there is some $\lambda$ such that $W\in \pP_\lambda$.
\end{lemma}

\begin{proof}
Since $W+U^\perp=E$, there exists a complement to $U^\perp$ lying in $W$, generated by some vector $w$.  Set
\[ \lambda = \frac{\omega(w,v)}{\omega(w,u)}, \]
which is possible since $w\notin U^\perp$.  Then $w\in v_\lambda^\perp$, so $W\in \pP_\lambda$.
\end{proof}

\begin{lemma}\label{lem:Plambda}
Under the above assumptions, there is a homotopy equivalence
\[ \pP_\lambda \simeq \PPi(v_\lambda^\perp, U+\langle v_\lambda\rangle). \]
\end{lemma}
Note that the right hand side makes sense because, by Lemma~\ref{lem:radical of perp}, 
\[ \R(v_\lambda^\perp) = \R(E)+\langle v_\lambda\rangle \leq U+\langle v_\lambda\rangle. \]

\begin{proof}
We define maps
\[ f: \pP_\lambda \longleftrightarrow \PPi(v_\lambda^\perp,U+\langle v_\lambda\rangle) :g \]
by setting
\[ f(W) = W\cap v_\lambda^\perp + \langle v_\lambda\rangle \ \ \ \textrm{and}\ \ \  g(V)=V. \]
The map $f$ lands in the claimed codomain because $f(W)$ is isotropic,  $\R(v_\lambda^\perp)=\R(E)+\langle v_\lambda\rangle< f(W)<v_\lambda^\perp$, the last inequality following from the fact that $v_\lambda^\perp$ is not isotropic because it would otherwise be equal to $\R(E)+\langle v_\lambda\rangle$, contradicting, through Witt's lemma, our assumption that $U\neq U^\perp$. 
Finally, $(W\cap v_\lambda^\perp)+U^\perp=E$ implies
$(W\cap v_\lambda^\perp)+(U^\perp\cap v_\lambda^\perp)=v_\lambda^\perp$. 
The map $g$ lands in the claimed codomain because $v_\lambda\in V$ (as $v_\lambda\in\R(v_\lambda^\perp)$), so 
$V+U^\perp$ contains $v_\lambda^\perp$, and hence contains $v_\lambda^\perp+U^\perp$, which is $E$ since
$v_\lambda^\perp$ and $U^\perp$ are two distinct hyperplanes.

The maps $f$ and $g$ are both poset maps as they both preserve inclusions.
The compositions is $f\circ g(V)=V$, because $v_\lambda\in V$.
The other composition satisfies
\[ g\circ f(W)\geq W\cap v_\lambda^\perp \leq W \]
in $\pP_\lambda$, since $W\in \pP_\lambda$ gives that also  $W\cap v_\lambda^\perp\in \pP_\lambda$. This gives a homotopy 
$g\circ f\sim\id_{\pP_\lambda}$.  Therefore $f$ is a deformation retraction.
\end{proof}

\begin{lemma}\label{lem:Plambda cap Pmu}
If $\lambda\neq\mu$, then
\[ \pP_\lambda\cap \pP_\mu = \PPi(E,U+\langle v\rangle), \]
as subposets of $\PPi(E,U)$.
In particular, the intersection does not depend on $\lambda$ or $\mu$.
\end{lemma}

\begin{proof}
If $W\in \PPi(E,U+\langle v\rangle)$, then $W$ contains a complement to $\langle u,v\rangle^\perp$, that is, the map
\[ f:W\to\F^2,\quad w\mapsto(\omega(u,w),\omega(v,w)) \]
is surjective.
In particular, for any $\alpha\in\F$, there exists $w\in W$ with
\[ (\omega(u,w),\omega(v,w)) = (1,\bar\alpha). \]
This $w$ lies in $W\cap v_\alpha^\perp$ but not in $U^\perp$, showing that $W\in \pP_\alpha$, for all $\alpha$.

Conversely, suppose $W\in \pP_\lambda\cap \pP_\mu$.  Then there exist $a,b\in W$ with
$a\in v_\lambda^\perp$, $b\in v_\mu^\perp$, and neither $a$ nor $b$ in $U^\perp$.  Scaling appropriately, we may assume
$\omega(a,u)=\omega(b,u)=1$.
Then
\[ \omega(a,v)=\lambda \+ \omega(b,v)=\mu. \]
Since
\[ \det\begin{bmatrix}1&\lambda\\1&\mu\end{bmatrix}\neq0, \]
this shows that
$W\xrightarrow{\flat} {}^\sigma E^*\to {}^\sigma \langle u,v\rangle^*$
is surjective, i.e.~that $W$ contains a complement to $\langle u,v\rangle^\perp$, so
$W\in \PPi(E,U+\langle v\rangle)$. 
\end{proof}

For each $\lambda\in\F$, let
\[ i_\lambda:\PPi(E, U+\langle v\rangle)\rar \pP_\lambda \]
be the inclusion map, recalling from Lemma~\ref{lem:Plambda cap Pmu} that $\PPi(E,\langle U,v\rangle)$ is the intersection of $\pP_\lambda$ with $\pP_\mu$ for any $\mu\neq\lambda$.

\begin{lemma}\label{lem:Plambda Pmu surjectivity}
Let $\lambda,\mu\in\F$ be distinct.  Then\\
(1)~the map
$\tilde H_*(\pP_\lambda\cap \pP_\mu)\xrightarrow{({i_\lambda}_*,{i_\mu}_*)} \tilde H_* \pP_\lambda \op \tilde H_* \pP_\mu$
is surjective;

\smallskip

\noindent (2)~the image of the set map
$\pi_1(\pP_\lambda\cap \pP_\mu) \to \pi_1 \pP_\lambda \ast \pi_1 \pP_\mu$ 
sending $x$ to ${i_\lambda}_*(x){i_\mu}_*(x)\inv$ contains $\pi_1 \pP_\lambda$.
\end{lemma}

\begin{proof}
First, observe that $v_\lambda^\perp$ and $v_\mu^\perp$ are distinct hyperplanes in $E$.
Indeed, if $v_\lambda^\perp=v_\mu^\perp$, then $v_\lambda\in\langle v_\mu\rangle+\R(E)$ by Lemma~\ref{lem:perpperp}, contradicting that $\langle u,v\rangle\cap\R(E)=0$. 
Hence we can pick an isotropic vector 
$c\in v_\mu^\perp$
such that
$\omega(c,v_\lambda)=1$.
This is possible by applying Lemma~\ref{lem:isotropic complement} to $\langle v_\lambda\rangle\leq v_\mu^\perp$, since $v_\lambda^\perp\cap v_\mu^\perp$ has codimension one in $v_\mu^\perp$.
The above properties imply that
$c\notin u^\perp$
since
\[ v_\mu^\perp\cap u^\perp = \langle u,v\rangle^\perp \leq v_\lambda^\perp. \]
Consequently, 
$\langle c\rangle+\R(E)\in \pP_\mu$.
Now define
$\phi:v_\lambda^\perp\to v_\lambda^\perp$
by
\[ \phi(x)=x-\omega(c,x)v_\lambda. \]
This is an isometry of $v_\lambda^\perp$ (though not of $E$).
Notice that
$\im(\phi)\leq c^\perp$
and
$\phi|_{\R(E)}=id_{\R(E)}$.
Also
$x\in u^\perp$ if and only if $\phi(x)\in u^\perp$.

We now define a map
$ j_\lambda:\pP_\lambda \rar \PPi(E,U+\langle v\rangle)$
by
\[ j_\lambda(W) = \phi(W\cap v_\lambda^\perp)+\langle c\rangle. \]
This subspace is indeed isotropic, since $c$ is isotropic and orthogonal to $\im(\phi)$.
Also it contains $\R(E)$.
Lastly, since $W\in \pP_\lambda$, $W\cap v_\lambda^\perp$ contains a vector $w$ which is not in $u^\perp$.
Then $\phi(w)$ is in $v_\lambda^\perp$ but not in $u^\perp$.
As $\phi(w)$ and $c$ are distinct, given that $c\notin v_\lambda^\perp$, this implies that $\langle \phi(w), c\rangle$ is a complement to $\langle u,v\rangle^\perp$.
Hence $j_\lambda(W)\in \PPi(E,\langle U,v\rangle)$ as claimed.

Now, $j_\lambda$ is a poset map.
We claim that $i_\lambda\circ j_\lambda$ is homotopic to the identity map of $\pP_\lambda$, and that
$i_\mu\circ j_\lambda$ is null-homotopic.
For the latter claim, observe that
\[ \phi(W\cap v_\lambda^\perp)+\langle c\rangle \geq \langle c\rangle+\R(E), \]
while $\langle c\rangle+\R(E)\in \pP_\mu$.
For the former claim, observe that
\begin{align*}
j_\lambda(W) \ \geq\  \phi(W\cap v_\lambda^\perp) \ \leq\  W\cap v_\lambda^\perp+\langle v_\lambda\rangle \ \geq\  W\cap v_\lambda^\perp \ \leq\  W.
\end{align*}
All of the intermediate terms are contained in $\pP_\lambda$, because both $W\cap v_\lambda^\perp$ and $\phi(W\cap v_\lambda^\perp)$ contain a vector complementary to $u^\perp$.  This completes the claims.

Exchanging the roles of $\lambda$ and $\mu$ also produces a map 
$j_\mu:\pP_\mu \to \PPi(E,U+\langle v\rangle)$
such that $i_\mu\circ j_\mu$ is homotopic to the identity map of $\pP_\mu$, and that
$i_\lambda\circ j_\mu$ is null-homotopic.  These together prove part (1) of the lemma, since
\[ \begin{bmatrix} {i_\lambda}_* & {i_\mu}_* \end{bmatrix}
\begin{bmatrix} {j_\lambda}_* \\ {j_\mu}_* \end{bmatrix}
= \begin{bmatrix} I & 0 \\ 0 & I \end{bmatrix}. \]

For part (2), let $y\in\pi_1 \pP_\lambda$ and $x= {j_\lambda}_*(y)$.  Then, using the homotopy relations established above, we get that 
${i_\lambda}_*(x){i_\mu}_*(x)\inv$ maps to $y$.
\end{proof}

Now we can handle the case where $\dim(U/\R(E))=1$:
\begin{lemma}[Induction step 2]\label{lem:C-M for small U}
Suppose that $\PPi(E,U)$ is Cohen-Macaulay for all formed spaces $E$ and all $U\in\PPi(E)$ such that either $\dim E<d$, or $\dim E=d$ and $\dim(U/\R(E))>1$.  Then $\PPi(E,U)$ is Cohen-Macaulay for all formed spaces $E$ and all $U\in\PPi(E)$ such that $\dim E=d$ and $\dim(U/\R(E))=1$.
\end{lemma}

\begin{proof}
Once again we may assume that $\R(E)=0$.   
The upper and lower intervals
\begin{align*}
\PPi(E,U)_{<W} = \PP(W/\R(E),(W\cap U^\perp)/\R(E)) \ \ \ \textrm{and}\ \ \ \PPi(E,U)_{>W} = \PPi(W^\perp)
\end{align*}
are already known to be Cohen-Macaulay by Theorem~\ref{thm:GL building properties} and Theorem~\ref{thm:PE is C-M}.
Hence we need only show that $\PPi(E,U)$ is spherical of dimension
\[ \dim\PPi(E,U) = \g(U^\perp) = \g(E)-\dim U = \g(E)-1 \]
(using Proposition~\ref{prop:induced form} for the last equality).
By our hypothesis and Lemma~\ref{lem:Plambda},
each subposet $\pP_\lambda$ is homotopy equivalent to a wedge of spheres of dimension
\[ \g(U+\langle v_\lambda\rangle^\perp) = \g(E)-2 \]
(which is one less than its actual dimension).
The same holds for the intersection $\pP=\pP_\lambda\cap \pP_\mu$,
for any $\lambda\neq \mu$ using now Lemma~\ref{lem:Plambda cap Pmu}.

By Lemma~\ref{lem:union of Plambda}, we have an equality of sets 
\[ \PPi(E,U)=\bigcup_{\lambda\in\F}\pP_\lambda. \]
Since $\pP_\lambda$ is upper-closed in $\PPi(E,U)$, this actually implies
$\abs{\PPi(E,U)}=\bigcup_{\lambda\in\F}\abs{\pP_\lambda}$
as well.

Assuming $\dim\PPi(E,U)\geq1$, it is therefore path connected, being a union of path-connected spaces with nonempty intersection.

When $\dim\PPi(E,U)\geq2$, we need to show it is 1-connected.  If $\dim\PPi(E,U)\geq3$, each $\pP_\lambda$, and their intersection $\pP$, is 1-connected.  By Seifert-van Kampen's theorem~\cite[Thm.~1.20]{Hatcher} (choosing a basepoint within $\pP$), $\PPi(E,U)$ is then 1-connected.
If $\dim\PPi(E,U)=2$, the $\pP_\lambda$'s and $\pP$ are still path-connected, so Seifert-van Kampen still applies to compute $\pi_1(\PPi(E,U))$, and gives zero because of Lemma~\ref{lem:Plambda Pmu surjectivity}(2).

Now it suffices to show that $\PPi(E,U)$ has its reduced homology concentrated in degree $\g(E)-1$.
We begin with $\abs{\pP_0}\cup\abs{\pP_1}$.  Its Mayer-Vietoris sequence  establishes that
$\tilde H_i(\abs{\pP_0}\cup\abs{\pP_1})=0$ for all $i$ except $i=\g(E)-1$ and $\g(E)-2$.
In those degrees, we get
\begin{align*}
0\to \tilde H_{\g(E)-1}(\abs{\pP_0}\cup\abs{\pP_1}) &\to
\tilde H_{\g(E)-2}(\abs{\pP_0}\cap\abs{\pP_1})\xrightarrow{\phi}
\tilde H_{\g(E)-2}\abs{\pP_0} \op \tilde H_{\g(E)-2}\abs{\pP_1} \\
&\to \tilde H_{\g(E)-2}(\abs{\pP_0}\cup\abs{\pP_1})\to 0.
\end{align*}
The map $\phi$ is surjective
by Lemma~\ref{lem:Plambda Pmu surjectivity}(1), establishing that the reduced homology of 
$\abs{\pP_0}\cup\abs{\pP_1}$
is concentrated in degree $\g(E)-1$. 

Now we argue with transfinite induction that
\[ P_S=\bigcup_{\lambda\in S}\abs{\pP_\lambda} \]
has its reduced homology concentrated in degree $\g(E)-1$ for any $S\subset\F$ of cardinality at least 2.  This will complete the proof, since
$P_\F=\abs{\PPi(E,U)}$.
It is necessary to check that
\[ P_{S\cup\lambda} = P_S\cup\abs{\pP_\lambda} \] has the desired property if $P_S$ does, and that
the union of a chain of $P_S$'s has the desired property if all of its terms do.  The latter is true because homology commutes with directed colimits of sub-simplicial complexes.
The former is true by a Mayer-Vietoris sequence:
\begin{align*}
0\to \tilde H_{\g(E)-1}(P_S)\to
\tilde H_{\g(E)-1}(P_S\cup\abs{\pP_\lambda}) &\to
\tilde H_{\g(E)-2}(P_S\cap\abs{\pP_\lambda})\xrightarrow{\phi} 
\tilde H_{\g(E)-2}\abs{\pP_\lambda} \\
&\to \tilde H_{\g(E)-2}(P_S\cup\abs{\pP_\lambda})\to 0.
\end{align*}
We have again used that the map $\phi$ is surjective by Lemma~\ref{lem:Plambda Pmu surjectivity}(1).
\end{proof}

\begin{proof}[Proof of Theorem~\ref{thm:PEU is CM}]
We prove the theorem by induction on the dimension of $E$. The poset $\PPi(E,U)$ can only exist if $\dim(E)\ge 2$, so that there may exist $U\in \PPi(E)$. If $\dim(E)=2$ and $\PPi(E)\neq\emptyset$, then $\R(E)=0$ and any $U\in \PPi(E)$  necessarily has dimension 1. In this case $\PPi(E,U)$ has dimension 0 by Lemma~\ref{lem:building dimensions} and the Cohen-Macaulay condition follows trivially. 
The induction step is given by  Lemmas~\ref{lem:C-M for big U} and~\ref{lem:C-M for small U} combined. 
\end{proof}

\section{Groups acting on the buildings}\label{sec:groupsandbuildings}

In this section, we define families of groups acting on the buildings and relative buildings of subspaces and isotropic subspaces. We study the relationships between these groups as well as properties of the actions. 

\begin{Def}\label{def:G and A}
(1) Let $V$ be a vector space and $V_0\leq V$ a subspace.  Define
\[ \A(V,V_0)=\set{g\in\G(V)}{g|_{V_0}=\textrm{id}_{V_0}} \]
to be the subgroup of the general linear group $\G(V)$ of automorphisms fixing $V_0$ pointwise.

\smallskip

\noindent (2)
Let $E=(E,q)$ be a formed space and $U\in\PPi(E)$. 
Define 
\[ \Ai(E,U) := \set{g\in\Gi(E)}{g|_{U}=\textrm{id}_{U}} \]
to be the subgroup of the bijective isometries $\Gi(E)$ of automorphisms fixing $U$ pointwise.
\end{Def}

The group $\A(V,V_0)$ just defined is closely related to the group 
\[ \A^T(V,V_0)=\set{g\in \GL(V)}
{gV_0=V_0\textrm{ and }g\textrm{ induces }\id_{V/V_0} \textrm{ on } V/V_0} \]
of Section~\ref{sec:vanishing of relative coinvariants}.
Indeed, these are isomorphic via duality, justifying the notation $\A^T$:
\begin{lemma}\label{lem:A and AT}
Let $V_0\leq V$ be vector spaces.  Dualizing gives an isomorphism
\[ \A^T(V,V_0) \cong \A(V^*,(V/V_0)^*),\quad g\mapsto (g^*)\inv. \]
\end{lemma}
\begin{proof}
The map $\G(V)\to\G(V^*)$ sending $g\mapsto (g^*)\inv$ is an isomorphism, 
so we just need to check that
$g\in \A^T(V,V_0)$ if and only if $(g^*)\inv\in \A(V^*,(V/V_0)^*)$.
But, since $V_0$ is the intersection of the kernels of all elements in $(V/V_0)^*$,
$g$ preserves $V_0$ if and only if $g^*$ preserves $(V/V_0)^*$.
Furthermore, $g$ induces the identity map modulo $V_0$ if and only if 
$g^*$ equals the identity map on $(V/V_0)^*$.
\end{proof}

\begin{rem}\label{rem:AAT}
By definition, forgetting the form on $E$ allows us to consider $\Ai(E,U)$ as a subgroup of $\A(E,U)$.
But \cite[Lem~3.8]{Forms} shows that $\Ai(E,U)$ is a subgroup of $\A^T(E,U^\perp)$ as well. So there is a diagram of group inclusions: 
\vspace{-2mm}
$$\xymatrix@R-2pc{& \A(E,U)\ar@{^(->}[dr] \ \ \ & \\
\Ai(E,U)\ \ \ \ar@{^(->}[ur]\ar@{^(->}[dr] && \G(E)\\
& \A^T(E,U^\perp)\ar@{^(->}[ur] & 
}$$
\end{rem}

For vector spaces $V_0\le V$, the group $\G(V)$ acts on $\PP(V)$ and the group $\A(V,V_0)$ acts on $\PP(V,V_0)$ 
of Definition~\ref{def:GL building}. 
For a formed space $E$,
$\Gi(E)$ acts on the poset $\PPi(E)$,
and $\Ai(E,U)$ acts on the poset $\PPi(E,U)$ of Definition~\ref{def:iso buildings}.
This is because an isometry of $E$ preserves the isotropic subspaces and preserves $\R(E)$, and also preserves $U^\perp$ if it preserves $U$.

\begin{prop}\label{prop:transitive action}
Let $V_0\leq V$ be  vector spaces,  $E=(E,q)$ a formed space and $U\in\PPi(E)$.
The actions of  $\G(V)$ on $\PP(V)$, of $\A^T(V,V_0)$ on $\PP(V,V_0)$, of $\G(E)$ on $\PPi(E)$ and of $\Ai(E,U)$ on $\PPi(E,U)$ are all transitive on the elements of any given dimension.
\end{prop}

\begin{proof}
Transitivity of the action of $\GL(V)$ on the elements of a given rank in  $\PP(V)$  follows from the fact that any isomorphism of a subspace $W$ can be extended to an isomorphism of the whole space using any choice of complement to $W$ inside $V$. For the action of $A^T(V,V_0)$ on ~$\PP(V,V_0)$, recall that the rank of $W\in \PP(V,V_0)$ is $\dim(W\cap V_0)$. Given two elements $W,W'$ of the same rank, we pick an isomorphism $W\cap V_0\to W'\cap V_0$ and extend it to a isomorphism of $V_0$. Now extend this isomorphism to the rest of $V$ by picking a complement $W_0^{(')}$ of $W^{(')}\cap V_0$ inside $W^{(')}$ and identify these complements via their canonical isomorphism to $V/V_0$.   
For the action of $\Gi(E)$ on $\PPi(E)$, transitivity is given by Witt's Lemma \cite[Thm 3.4]{Forms} and for the action of $\Ai(E,U)$ on $\PPi(E,U)$, it is given by \cite[Cor 3.9]{Forms}. 
\end{proof}

Witt's Lemma \cite[Thm 3.4]{Forms} gives that any desired linear automorphism of an isotropic subspace $W\in\PPi(E)$ preserving the radical can be achieved by an  isometry of a  $E$.  But in fact something much stronger is true, as we will see now. We start by fixing some notation.

For $X$ an element in a building $\pP$ of subspaces of a vector space, and $\GA$ a subgroup of the general linear group acting on $\pP$, we denote by $\GA_X$ the stabilizer of $X$, that is the subgroup of elements of $\GA$ preserving $X$ as a set, and $\GA_{id_X}$ the subgroup of elements fixing $X$ elementwise.

\begin{prop}\label{prop:restriction split}
Let $E=(E,q)$ be a formed space.\\
(1) For any $W\in\PPi(E)$, the map
\[ \Gi(E)_W \to \G(W)_{\R(E)}\times\Gi(W^\perp/W) \]
sending $f$ to its restriction to $W$ and its induced map on $W^\perp/W$,
is split surjective. 

\smallskip

\noindent
(2) For any $U\in\PPi(E)$ and $W\in\PPi(E,U)$, the analogous map
\[ \Ai(E,U)_W \to \A^T(W,W\cap U^\perp)_{id_{\R(E)}}\times\Gi(W^\perp/W) \]
is split surjective. 
\end{prop}

The map in the second part of the statement makes sense because,
by \cite[Lem 3.8]{Forms} (see also Remark~\ref{rem:AAT}), 
restriction from $E$ to $W$ defines a homomorphism from $\Ai(E,U)_W$ to 
\[ \A(W,W\cap U)\cap \A^T(W,W\cap U^\perp), \] 
which is equal to $\A^T(W,W\cap U^\perp)_{id_{\R(E)}}$ given 
that $W\cap U=\R(E)$
by \cite[Lem 3.10]{Forms}.  

The proof of the proposition will make use of the following lemma: 

\begin{lemma}\label{lem:decomp}
Let $E=(E,q)$ be a formed space and $W\in\PPi(E)$.
For any complement $C$ to $W^\perp$ in $E$, 
there exists $E'\leq E$ such that
\begin{enumerate}
\item $E=W\op C\op E'$,
\item $W^\perp=W\op E'$,
\item $E'$ is orthogonal to $W\op C$.
\end{enumerate}
In particular, $E'$ projects isomorphically and isometrically onto $W^\perp/W$.
\end{lemma}

\begin{proof}
Using Lemmas~\ref{lem:capperp} and~\ref{lem:perpperp} and the fact that 
$C+W^\perp=E$, we get 
\[ C^\perp\cap(W+\K(E))=C^\perp\cap (W^\perp)^\perp=(C+W^\perp)^\perp=\K(E) \]
and so
$C^\perp\cap W=\R(E)$. 
Also
$C\cap W^\perp=0$
implies
\[ C^\perp+W=C^\perp+(W^\perp)^\perp=(C\cap W^\perp)^\perp=E. \]
Let $E'$ be a complement to $\R(E)$ in $(W\op C)^\perp$.
We have
\[ W\cap (W\op C)^\perp\leq W\cap C^\perp=\R(E), \]
so $E'$ is disjoint from $W$.
In fact,
$W\op E'=W^\perp$, 
since the left is contained in the right and
\[ W^\perp=W^\perp\cap(C^\perp+W)=(W+C)^\perp+W. \]
It then follows that
$E=(W\op C)\op E'$, 
with $W\op C$ orthogonal to $E'$.
\end{proof}

\begin{proof}[Proof of Proposition~\ref{prop:restriction split}]
We will prove part (2); part (1) is similar (simply set $U=\R(E)$, and skip the last paragraph).
Using that $U\cap W=\R(E)$ by \cite[Lem 3.10]{Forms}, 
and applying Lemma~\ref{lem:isotropic complement} 
 to an arbitrary complement $Z_0$ of $\R(E)$ in $U$, we can 
choose a linear complement $C$ of $W^\perp$, 
\[ C\op W^\perp=E \]
which is isotropic, and such that  $U\leq C+\R(E)$.

Let $E'$ be a subspace of $E$ with the properties guaranteed in Lemma~\ref{lem:decomp}: 
$E=W\op C\op E'$, $W\op E'=W^\perp$, and 
$E'$ projects isomorphically and isometrically onto $W^\perp/W$ (where the latter is given the induced form of Proposition~\ref{prop:induced form}).
Identifying $E'$ with this quotient then defines a homomorphism
\[\Gi(W\op C)_W\times\Gi(W^\perp/W) \rar  \Gi(E)_W. \]
Furthermore, since $U\leq C+\R(E)\leq C\op W$,
this homomorphism restricts to a map
\[\Ai(W\op C,U)_W\times\Gi(W^\perp/W) \rar  \Ai(E,U)_W. \]
This map and the one in the statement fit in the following commuting triangle: 
\[\xymatrix{ \Ai(E,U)_W \ar[rr]&& \A^T(W,W\cap U^\perp)_{id_{\R(E)}}\times\Gi(W^\perp/W) \ar@{-->}@/^1pc/[d]\\
&& \Ai(W\op C,U)_W\times\Gi(W^\perp/W) \ar[llu]\ar[u] }\]
where the map $\Ai(W\op C,U)_W\to \A^T(W,W\cap U^\perp)_{id_{\R(E)}}$ is the restriction. We will construct a section of the top horizontal map by constructing one of the vertical map.

Observe that the map
\[ \phi:C\to {^\sigma(W/\R(E))^*},\quad v\mapsto\omega_q(-,v) \]
is injective as $C\cap W^\perp=0$.
It is bijective by a dimension count using Lemma~\ref{lem:dimension of perp}:
\[ \dim C = \codim W^\perp = \dim W-\dim(\R(E)). \]
We define a homomorphism
\[ \Psi:\A^T(W,W\cap U^\perp)_{id_{\R(E)}}\to \Ai(W\op C,U)_W\]
by sending $f$ to $f\op\tilde f$, where
\[ \tilde f=\phi\inv\circ {f^*}\inv\circ\phi. \]
(Here $f$ preserves $\R(E)$ and hence induces a map on $W/\R(E)$.)
The map $\Psi$ will define a homomorphism
if we can check that it is well-defined: that is, that $f\op\tilde f$ is an isometry of $W\op C$ that fixes $U$---by construction, $\Psi(f)$ preserves $W$.
Since both $W$ and $C$ are isotropic, $\Psi(f)$ is an isometry if and only if for all $w\in W$, $c\in C$,
\[ \omega_q(f(w),\tilde f(c)) =
\omega_q(w,c), \]
and $\tilde f$ is precisely defined so that this holds. 
Lastly we need to check that $\Psi(f)$ is the identity on $U$ if $f$ the identity modulo $W\cap U^\perp$ and the identity on $\R(E)$.  
Since $U=(U\cap C)\op\R(E)$, it suffices to show that $\tilde f$ is the identity on $U\cap C$.  Let $u\in U\cap C$.  Then
\[ \phi(u)|_{U^\perp}=0. \]
Since $f$ is the identity modulo $W\cap U^\perp$,
\[ f^*:{^\sigma(W/\R(E))^*}\to{^\sigma(W/\R(E))^*} \]
is the identity on the subspace of those elements which vanish on 
$(W\cap U^\perp)/\R(E)$.
But $\phi(u)$ is such an element.
It follows that $\tilde f(u)=u$, as needed.
\end{proof}

Finally, we will need the following result in Section~\ref{sec:stability}:

\begin{lemma}\label{lem:ker of restriction}
Let $E=(E,q)$ be a formed space, $U\in\PPi(E)$, and $W\in\PPi(E,U)$.  Then
\[ \Ai(U^\perp/U,W\cap U^\perp/\R(E)) \cong \ker\left(\Ai(E,U)_W\xrightarrow{res}\G(W)\right), \]
naturally with respect to isometries $E\to E'$ inducing a bijection on radicals.
\end{lemma}
Recall here that $U\cap W=\R(E)$ by \cite[Lem 3.10]{Forms}. 

\begin{proof}
The kernel of the restriction to $W$ in the statement is 
\begin{align*}
K 
&= \set{g\in\Gi(E)}{g|_{W+U}=id_{W+U}}
\end{align*}
There is a homomorphism 
\[ \phi:K\to\Ai(U^\perp/U,W\cap U^\perp/\R(E)) \]
because each $g\in K$ preserves $U$, and hence also $U^\perp$, and fixes $W$.
We claim that this map is an isomorphism; we will define an inverse.   

To describe a map in the other direction, choose a complement $C$ to $U^\perp$ in $E$, which is contained in $W$; this is possible since $W+U^\perp=E$.
Apply Lemma~\ref{lem:decomp}
to get a subspace $E'\leq E$ such that
$E=U\op C\op E'$, with 
$U^\perp=U\op E'$ and 
$E'$ is orthogonal to $U\op C$.
As $E'$ projects isomorphically and isometrically onto $U^\perp/U$, 
this determines an injection
\[ \psi:\Gi(U^\perp/U)\cong \Gi(E')\hookrightarrow\Gi(E), \]
by extending with the identity map of $U\oplus C$.
By construction $\psi(f)$ is the identity on $U$ (which contains $\R(E)$).
When $f$ is the identity on $W\cap U^\perp$, $\psi(f)$ is also the identity on $W$, since
\[ W=W\cap(U^\perp\op C)=(W\cap U^\perp)\op C. \]
Therefore, $\psi$ restricts to a homomorphism
\[ \Ai(U^\perp/U,W\cap U^\perp)\to K. \]
One then verifies that $\psi$ is an inverse to $\phi$.  Hence the kernel of restriction to $W$ is as claimed.
\end{proof}

\section{Homological stability}\label{sec:stability}

In this section, we consider the  group inclusions
\begin{equation}\label{equ:GLstab}
 \G(V)\rar \G(V\op \F) \ \ \ \textrm{and}\ \ \ \A(V,V_0)\rar \A(V\op \F,V_0) 
\end{equation}
as well as 
\begin{equation}\label{equ:Gistab}
\ \ \ \ \Gi(E)\to \Gi(E\op \HH) \ \ \ \textrm{and}\ \ \ \Ai(E,U)\rar \Ai(E\op \HH,U) 
\end{equation}
given by extending an automorphism to be the identity on the added summand $\F$ or $\HH$, where $\HH$ is   the hyperbolic space of Definition~\ref{def:hyperbolic plane}. 
These maps  are our main {\em stabilization maps}. 
We will also need the inclusion
\begin{equation}\label{equ:ATstab}
\A^T(V,V_0)\rar \A^T(V\op \F,V_0\op\F), 
\end{equation}
likewise by extending by the identity of the added summand; it corresponds to the above stabilization map
$\A(V,V_0)\to\A(V\op \F,V_0)$ via the identification of Lemma~\ref{lem:A and AT}.

Our main result for this section is the following stability result for
these maps, stated in terms of the vanishing in a range of the
corresponding relative homology groups. The bounds are given in terms of the dimension of the vector spaces and the genus, and dimension of the radical, of the formed space, as defined in Section~\ref{sec:forms} (Definition~\ref{def:kerneletc}). 

\begin{thm}\label{thm:mainstability}
(1) Let $V$ be a finite-dimensional vector space over a field $\F\neq \F_2$, and $V_0\le V$ a subspace.
Then
$$\begin{array}{ll}
H_d(\G(V\op \F),\G(V))=0  & \textrm{for}\ \ d\le \dim V\\ 
H_d(\A(V\op \F,V_0),\A(V,V_0))=0 \hspace{10mm} &\textrm{for}\ \ d\le \dim V-\dim V_0. \hspace{10mm}
\end{array}$$
Moreover, if $\F=\F_{\ell^r}$ for $\ell$ a prime with $\ell^r\neq 2$, 
$$\begin{array}{ll}
H_d(\G(V\op \F_{\ell^r}),\G(V);\F_\ell)=0  & \textrm{for}\ \ d\le \dim V+r(\ell-1)-2\\ 
H_d(\A(V\op \F_{\ell^r},V_0),\A(V,V_0);\F_\ell)=0 \hspace{10mm} &\textrm{for}\ \ d\le \dim V-\dim V_0+r(\ell-1)-2. \hspace{10mm}
\end{array}$$

\noindent
(2) Let $E=(E,q)$ be a formed space over a field $\F\neq \F_2$, and let $U\in\PPi(E)$. Then 
$$\begin{array}{ll}
H_d(\Gi(E\op \HH),\Gi(E))=0 & \textrm{for}\ \ d\le \g(E)\\
H_d(\Ai(E\op \HH,U),\Ai(E,U))=0  &\textrm{for}\ \ d\le \g(E)-(\dim(U)-\dim\R(E)).
\end{array}$$ 
\end{thm}

The proof of the theorem is given in Section~\ref{sec:proof}. To prepare for it, we first assemble in Section~\ref{sec:props} the properties of the groups and buildings that will be needed for the proof, 
building on Section~\ref{sec:groupsandbuildings}, and construct in Section~\ref{sec:SS} the relevant spectral sequences, building on the short Appendix~\ref{app:CM}. 

\begin{rem}\label{rem:range}
(1)~Note that $r(\ell-1)-2>0$ whenever $q=\ell^r>4$, and equals $0$ for $q=3$ or $4$. So the second part of (1) in the theorem gives an improvement for $\G(\F_{q})$ for all $q>4$. This improvement was first found by 
Galatius--Kupers--Randal-Williams \cite{GKRW}. Our proof here is really different, though we use one lemma from their paper, about the vanishing of the homology  $H_*(A^T(V,V_0);\St(V,V_0)\otimes \F_\ell)$ in the range $*\le  r(\ell-1)-2$ (see Lemma 5.2 in that paper). This lemma is an independent computation, that does not use the rest of the paper. 

Note that the map $H_2(\GL_2(\F_4)) \to H_2(\GL_{3}(\F_4))$ is not injective, so the stability range cannot be improved in that case.  
Indeed, the latter group being 0 with $\F_2$ coefficients by the theorem and
Quillen's vanishing result \cite{QuillenK}, while one can check that
the first is not.  Indeed, the $16$--dimensional representation $\rho$
induced by the $C_2$ subgroup generated by
$\begin{pmatrix}1&1\\0&1\end{pmatrix}$,  permuting the elements of
$\F_4^2$, splits as a sum of 10 trivial represpentations and 6 copies of the sign representations $\sigma$. Hence the Steiffel-Withney class 
$$w(\rho)=w(6\sigma)=w(\sigma)^6=(1+t)^6=1+t^2+t^4+t^6,$$
with $t$ the generator in $H^1(C_2)=H^1(\bbR P^\infty)$. As it has a non-trivial component $t^2$ of degree 2,  and is the restriction to a subgroup of a representation of $\GL_2(\F_4)$, 
this shows that $H^2(\GL_2(\F_4);\F_2)$ must be non-zero. But then the same holds for $H_2(\GL_2(\F_4);\F_2)$ by the universal coefficient theorem\footnote{This argument was suggested to us by Oscar Randal-Williams.}. 

\noindent
(2)~It is natural to wonder whether an analogous improvement of the offset also holds in the case of the automorphism groups of formed spaces for finite fields at the characteristics. 
The argument given here in the case of the linear groups does not transfer to these other groups because of one main differences between the buildings used in the two cases: for the linear groups, there is a natural augmented version of the building, $\bP(V)$ and $\bP(V,V_0)$, that is contractible, and this is the building we use for the proof. Such an augmented version (with nice properties!) does not exist in the isotropic case.  
An argument in the set-up of \cite{GKRW} instead could potentially work (if such an improvement is possible!), but requires the study the connectivity of new buildings. 
\end{rem}

\subsection{Properties of the groups and buildings}\label{sec:props}

Let $V$ be a finite-dimensional vector space and $V_0$ a subspace.
Recall from Definitions~\ref{def:GL building} and \ref{def:GL relative building bar} the building and relative building
\begin{align*}
 \PP(V)=\{0<W<V\},  \ \ \ \ &\PP(V, V_0) = \set{W<V}{W+V_0=V}, \\
 \bP(V)=\{0<W\leq V\}, \ \ \ \  &\bP(V, V_0) = \set{W\leq V}{W+V_0=V}.
\end{align*}
Note that, just as we have two choices $\A$ and $\A^T$ for the definition of the relative groups, there are two choices for the relative buildings, the other choice being
\[ \bP^T(V,V_0):=\set{0\leq W< V}{W\cap V_0=0}. \]
Sending $W$ to $(V/W)^*$ defines an isomorphism 
$$\bP^T(V,V_0)\stackrel{\cong}\rar \bP(V^*,(V/V_0)^*)$$
which is order preserving if we define the order on $\bP^T(V,V_0)$ to be given by reversed inclusion. 
Both $\A(V,V_0)$ and $\A^T(V,V_0)$ act on both posets $\bP(V,V_0)$ and $\bP^T(V,V_0)$.
To prove the stability theorem, we will use the action of $\A^T(V,V_0)$
on $\bP(V,V_0)$ and, dually, the action of $\A(V,V_0)$ on $\bP^T(V,V_0)$. 

\begin{lem}\label{lem:P and PT}
The isomorphism $\bP^T(V,V_0)\cong \bP(V^*,(V/V_0)^*)$ given above is equivariant with respect to the action of  $\A(V,V_0)$ on $\bP^T(V,V_0)$ and of $\A^T(V,V_0)$
on $\bP(V,V_0)$, where the groups are identified via the isomorphism of Lemma~\ref{lem:A and AT}. 
\end{lem}

\begin{proof}
We need to check that for any $W\in \PP^T(V,V_0)$ and $g\in \A(V,V_0)$, we have that $(V/g(W))^*=(g^{-1})^*(V/W)^*$, which holds because 
$f\in (g^{-1})^*(V/W)^*$ if and only if $g^*f\in (V/W)^*$, that is if and only if $f\circ g$ vanishes on $W$, and is the case if and only if $f$ vanishes on $g(W)$. 
\end{proof}

For $E=(E,q)$ a formed spaced and $U$ an isotropic subspace,
recall from Definition~\ref{def:iso buildings} the building and relative building
\begin{align*}
\PPi(E)&=\{\R(E)< W< E\ |\ W \ \textrm{isotropic}\}, \\
\PPi(E,U)&=\{\R(E)<W< E\ | \ W \ \textrm{isotropic}, W+U^\perp=E \}
\end{align*}
where $\R(E)$ is the radical of $E$ of Definition~\ref{def:kerneletc}. 
Note that $\PPi(E,\R(E))=\PPi(E)$.
There are no distinct ``augmented'' variants of $\PPi(E)$ and $\PPi(E,U)$, because $E$ itself is never an isotropic subspace unless 
$E=\R(E)$, in which case all variants of the building are empty.

The following result summarizes properties of the buildings that will be needed in the stability argument. 

\begin{prop}\label{prop:buildings}
Let $V$ be a finite dimensional vector space and $V_0\leq V$ be a subspace.  Let $E=(E,q)$ be a formed space and $U\in\PPi(E)$. 

\smallskip

\noindent
(1)~The posets $\bP(V),\bP(V,V_0),\PPi(E)$ and $\PPi(E,U)$ are Cohen-Macaulay with:  

\begin{center}
\begin{tabular}{l|l|l|l}
 Poset $P$ &  $\dim(P)$ & rank($W$) for $W\in P$& lower intervals\\
\hline
&&&\\
 $\bP(V)$      & $\dim V-1$    & $\dim W-1$ & $\bP(V)_{<W} \ \ \ \ \cong\ \  \PP(W)$\\
  $\bP(V,V_0)$ &  $\dim V_0$    & $\dim(W\cap V_0)$ & $\bP(V,V_0)_{<W}  \, \cong\ \  \PP(W,W\cap V_0)$\\
  $\PPi(E)$    &  $\g(E)-1$     & $\dim W-\dim\R(E)-1$ &$\PPi(E)_{<W} \ \ \ \ \cong \ \ \PP(W/\R(E))$\\
 $\PPi(E,U)$  &  $\g(U^\perp)$ & $\dim W-\dim U$ &$\PPi(E,U)_{<W} \, \cong \PP\big(W/\R(E),(W\cap U^\perp)/\R(E)\big)$. 
\end{tabular}
\end{center}
(Note that $\g(U^\perp)=g(E)-\dim U+\dim \R(E)$ by Lemma~\ref{prop:induced form}.)

\smallskip

\noindent
(2)~The natural inclusions of vector spaces (resp.~formed spaces) induce rank-preserving maps of posets 
\begin{center}
\begin{tabular}{rclcrcl}
$\bP(V)$ &$\rar$&$\bP(V\op \F)$ & \ \ \  \ and\ \ \   \  &  $\PPi(E)$&$\rar$&$ \PPi(E\op \HH)$\\
$\bP(V,V_0)$ &$\rar$&$ \bP(V\op \F,V_0\op \F)$ & &  $\PPi(E,U)$&$\rar$& $\PPi(E\op \HH,U)$,\\
\end{tabular}
\end{center}
In each case, the dimension of the codomain is exactly one more than that of the domain. \\
Moreover, the maps are equivariant with respect to the stabilization maps 
\begin{center}
\begin{tabular}{rclcrcl}
$\G(V)$ &$\rar$&$ \G(V\op \F)$ &\ \ \ \ and \ \ \ \ & $\Gi(E)$ &$\rar$&$ \Gi(E\op \HH)$\\
$\A^T(V,V_0)$ &$\rar$&$ \A^T(V\op \F,V_0\op\F)$ && $\Ai(E,U)$ &$\rar$&$ \Ai(E\op \HH,U)$.\\
\end{tabular}
\end{center}
\end{prop}

\begin{proof}
Statement~(1) is given by 
Theorems~\ref{thm:barCM}, \ref{thm:GL building properties}, \ref{thm:PE is C-M} and~\ref{thm:PEU is CM}, and
Lemmas~\ref{lem:building dimensions} and~\ref{lem:intervals}.

For (2), we have that the natural inclusion, taking $W$ in the left hand poset to itself as an element of the right hand poset, is a poset map. The map is rank-preserving by (1), where 
in the case of $\PPi(E)$, we use that
$\R(E\oplus\HH)=\R(E)$ by Lemmas~\ref{lem:properties of HH}, 
 and in the case of $\PPi(E,U)$, also that  
\[ W\cap U^\perp = W\cap (U^\perp\cap E), \]
as subspaces of $E\oplus\HH$, which is true since $W\leq E$, noting also that $U^\perp\cap E$ identifies with the orthogonal complement of $U$ as a subspace of $E$.

The claim about dimensions follows from part (1), together with the fact that
$\g(E\op\HH)=\g(E)+1$ by Lemma~\ref{lem:properties of HH},
and also the observation that, as subspaces of $E\oplus\HH$,
\[ U^\perp = (U^\perp\cap E)\oplus\HH, \]
so that $\g(U^\perp)=\g(U^\perp\cap E)+1$ by Lemma~\ref{lem:properties of HH}.
The equivariance of the maps follows from the compatibility of the stabilization maps defined on the groups and the buildings. \qedhere
\end{proof}

The following proposition gives a description of the stabilizers of the action of the groups 
$\G(V),\Gi(E),\A^T(V,V_0)$ and $\Ai(E,U)$ on their associated buildings.
In each case, the stabilizer is an extension with kernel another group of one of these types, making an inductive argument possible. 
In the cases of $\G(V)$ and $\A^T(V,V_0)$, the kernel is actually a group of type $\A(-,-)$ rather than $\A^T(-,-)$, but these are isomorphic by Lemma~\ref{lem:A and AT}.

Recall that, for a group $G$ acting on a building $P$, we use the notation $G_W$ for the stabilizer of $W\in P$, and $G_{id_W}$ for the subgroup of elements that restrict to the identity on $W$. 

\begin{prop}\label{prop:stabilizers}
Let $V_0\leq V$ be finite-dimensional vector spaces.
Let $E=(E,q)$ be a formed space and $U\in\PPi(E)$. \\
(1) The groups
$\G(V)$, $\A^T(V,V_0)$, $\Gi(E)$, $\Ai(E,U)$
act (respectively) on the posets
$\bP(V)$,$\bP(V,V_0)$, $\PPi(E)$, $\PPi(E,U)$, 
transitively on the elements of a given rank.

\smallskip

\noindent 
(2) For an element $W$ in each of these posets, there is a split short exact sequence
\[ \begin{array}{ccrccclcc}
1&\rar& 
\A(V,W) &\rar& \G(V)_W  &\sta{res\ }\rar& \G(W) &\rar& 1
, \\
1&\rar& 
\A(V_0,W\cap V_0) &\rar& \A^T(V,V_0)_W &\sta{res\ }\rar& \A^T(W,W\cap V_0) &\rar& 1
, \\
1&\rar& 
\Ai(E,W) &\rar& \Gi(E)_W  &\sta{res\ }\rar& \G(W)_{\R(E)} &\rar& 1
, \\
1&\rar& 
\Ai(U^\perp/U,W\cap U^\perp) &\rar& \Ai(E,U)_W 
  &\sta{res\ }\rar& \A^T(W,W\cap U^\perp)_{id_{\R(E)}} &\rar& 1.
\end{array} \]
Each of these short exact sequences commutes with the stabilization homomorphisms (\ref{equ:GLstab}), (\ref{equ:Gistab}), (\ref{equ:ATstab}) in the left two terms of the sequence, 
and the identity on the right term. 
When $W=W_0$ is of rank 0, the second and fourth sequences degenerate to isomorphisms
\begin{equation}\label{equ:G=A}
\begin{aligned}
\G(V_0) \sta{\cong}\rar \A^T(V,V_0)_{W_0} \ \ \ \ \textrm{and}\ \ \ \  \Gi(U^\perp/U) \sta{\cong}\rar \Ai(E,U)_{W_0}.
\end{aligned}
\end{equation}

\smallskip

\noindent
(3) 
Let $X$ be a rank 0 element in the poset $\PP^T(V,W), \PP^T(V_0,W\cap V_0), \PPi(E,W)$ or $\PPi(U^\perp/U,W\cap U^\perp)$ associated to the kernel groups 
in each of the four exact sequences in (2). 
The splitting in (2) may be chosen so that it commutes with the stabilizer of $X$ in the kernel. 
(By (\ref{equ:G=A}), this stabilizer identifies with $\G(X)$ in the general linear cases and  with $\Gi(W^\perp/W)$ in the isotropic cases.) 
\end{prop}

\begin{proof}
Statement (1) follows from Proposition~\ref{prop:transitive action}. 

The identifications of the kernels in the four sequences in (2) follows directly from the definitions of the groups in the first three cases, and is given by Lemma~\ref{lem:ker of restriction} in the last case. 
Surjectivity of the right hand map will follow from the existence of a splitting. We will in all cases construct a splitting satisfying the extra condition in (3), proving (2) and (3) at the same time. 

In the first case, a rank 0 element $X$ in $\bP^T(V,W)$ is a complement of $W$ in $V$. Writing $V=W\op X$ defines the required splitting $\G(W)\times \G(X)\to \G(V)$, noting that $\G(X)$ identifies indeed canonically with the stabilizer of $X$ for the action of $\A(V,W)$ on $\bP^T(V,W)$. 
In the second case, a rank 0 element $X$ in $\bP^T(V_0,W\cap V_0)$ is a complement to $W\cap V_0$ in $V_0$. Picking a complement $W'$ of $W\cap V_0$ inside $W$, we have $V=W'\op (W\cap V_0)\op X$. 
This gives an inclusion
\[ \A^T(W,W\cap V_0)\times\G(X)\rar \A^T(W\op X,V_0), \]
likewise giving the required splitting.

A rank 0 element $X$ in $\PPi(E,W)$ can be written as $X=C\op \R(E)$ with $C$ a complement of $W^\perp$ in $E$. As in the proof of 
Proposition~\ref{prop:restriction split}, we write $E=W\op C\op E'$ with $E\cong W^\perp/W$. The lemma shows that this gives a splitting 
$$\G(W)_{\R(E)} \x \Gi(W^\perp/W) \rar \Gi(E)$$
of the restriction map. As elements of $\Ai(E,W)$ are the identity
modulo $W^\perp$ (by \cite[Lem 3.8]{Forms}), 
we have that  
 $\Gi(W^\perp/W)$ identifies with the stabilizer subgroup $\Ai(E,W)_X$, which finishes the proof in this case.   
The result is similarly given by Proposition~\ref{prop:restriction split} in the last case too.

The compatibility with the stabilization maps is a direct check. 
\end{proof}

Finally we will need the following retract information:
\begin{lemma}\label{lem:buildingstab1}
Let $V$ be a finite-dimensional vector space and $E=(E,q)$ a formed space.
The natural inclusions
$$\G(V) \rar \G(\F\op V)_{\F} \  \ \ \textrm{and}\ \ \ \Gi(E)\rar \Gi(\HH\op E)_L$$
are split injective, compatibly with the stabilization maps.
Here $L\leq\HH$ is any choice of isotropic line in the formed space $\HH$ of Definition~\ref{def:hyperbolic plane}.
\end{lemma}

\begin{proof}
The former splitting arises from viewing $V$ as $(\F\op V)/\F$, and the latter from
$E\iso L^\perp/L$.
For this last equality, we need the check that
$L\op E = L^\perp$
as subspaces of $\HH\op E$.  The left side is contained in the right; the reverse inclusion holds since
$L^\perp\cap\HH=L$
as $\HH$ is 2-dimensional and nondegenerate by Lemma~\ref{lem:properties of HH}. 
\end{proof}

\begin{lemma}\label{lem:buildingstab2}
Let $V_0\leq V$ be vector spaces, $E$ a formed space, and $U\in\PPi(E)$. Let $W_0$ be a rank 0 element in $\bP(V,V_0)$ (resp.~$\PPi(E,U)$). 
The inclusions 
\[ \G(V_0)\cong\A^T(V,V_0)_{W_0}\rar \A^T(V,V_0) \ \ \ \ \textrm{and}\ \ \ \ \Gi(U^\perp/U)\cong\Ai(E,U)_{W_0}\rar \Ai(E,U) \]
associated to the isomorphisms (\ref{equ:G=A}) are split, 
compatibly with the stabilization maps.
\end{lemma}

\begin{proof}
In the first case, $W_0$ is a complement to $V_0$ in  $V$ and the splitting 
is given by restriction to $V_0$.
In the second case, $W_0$ is a complement to $U^\perp$ in $E$ and the splitting 
sends $f\in\Ai(E,U)$ to its induced map on $U^\perp/U$.
\end{proof}

\subsection{Spectral sequences}\label{sec:SS} 
For $V_0\leq V$  finite-dimensional vector spaces, recall from Section~\ref{sec:join} the Steinberg module 
\[ \St(V)=\tilde H_{\dim\PP(V)}(\PP(V)) \+ \St(V,V_0)=\tilde H_{\dim\PP(V,V_0)}(\PP(V,V_0)), \]
 the top homology groups of the posets $\PP(V)$ and $\PP(V,V_0)$.

\begin{rem}
We use the convention that $\tilde H_{-1}(\emptyset)=\bbZ$,
so that $ \St(\F^1)=\bbZ$. 
\end{rem}

\begin{thm}\label{thm:SS} 
Consider a tuple $(\GA,\GA',\pP,\pP')$ of one of the following four types: 
$$\begin{array}{lllll}
(1)& \GA=\G(V)      & \GA'=\G(V\op \F)            & \pP=\bP(V)     & \pP'=\bP(V\op \F),\\
(2)& \GA=\Gi(E)     & \GA'=\Gi(E\op \HH)          & \pP=\PPi(E)    & \pP'=\PPi(E\op \HH),\\
(3)& \GA=\A^T(V,V_0) & \GA'=\A^T(V\op \F,V_0\op \F) & \pP=\bP(V,V_0) &\pP'=\bP(V\op \F,V_0\op \F),\\
(4)& \GA=\Ai(E,U)     & \GA'=\Ai(E\op \HH,U)          & \pP=\PPi(E,U)  & \pP'=\PPi(E\op\HH,U).
\end{array}$$
where $V$ is an $\F$--vector space of dimension at least 1, $0<V_0< V$ a subspace, $E=(E,q)$ a formed space,
and $U\in\PPi(E)$.
Then there is a spectral sequence converging to zero in all degrees in cases (1) and (3), and in degrees $p+q\leq \dim P$ in cases (2) and (4),  with
\[ E^1_{pq}=\begin{cases}
H_q(\GA',\GA)&p=-1,\\
H_q(\GA'_{W_p},\GA_{W_p};\St_p)&0\leq p\leq \dim \pP\\
H_q(\GA'_{W_p};\St_p)&p=\dim \pP+1\\
0 & p>\dim \pP+1.
\end{cases} \]
Here $W_p\in \pP'$ is any element of rank $p$, with $W_p\in \pP$ if $p\leq\dim \pP$.
Also
\[ \St_p = \begin{cases}
\St(W_p)                                       & \text{in case (1)}, \\
\St(W_p/\R(E))                                 & \text{in case (2)}, \\
\St(W_p,W_p\cap V_0)                           & \text{in case (3)}, \\
\St\big(W_p/\R(E),(W_p\cap U^\perp)/\R(E)\big) & \text{in case (4)}.  
\end{cases} \]
Moreover, if $\F\neq \F_2$,  we have  $$E^1_{\dim P+1,0}=0$$
in all cases. In case (1) and (3), if $\F=\F_{\ell^r}$ with $\ell^r\neq 2$ and we take homology with coefficients in $\F_\ell$, then we additionally have 
\begin{align*}E^1_{\dim P+1,q}=0 \ \ \ \  &  \textrm{for all}\ q \ \ \text{in case (1)},\\
&\textrm{for all}\ q\le r(\ell-1)-2 \ \ \text{in case (3)}. 
\end{align*}
\end{thm}

These spectral sequences come from the action of the groups  $\GA$ and $\GA'$ on chain complexes of modules associated to their corresponding posets $\pP$ and  $\pP'$, as given by the following theorem: 

\begin{thm}\label{les} 
Let $(\GA,\pP)$ be as in case (1), (2), (3) or (4) in Theorem~\ref{thm:SS}. 
Then there is a chain complex $C_\bullet=C_\bullet(\GA,\pP)$ of $\GA$-modules with 
$$C_p=\left\{\begin{array}{ll}
 \induction{\St_p}{\GA_{W_p}}{\GA}  & \ \  0\le p\le \dim(\pP), \\
\bbZ & \ \  p=-1
\end{array}\right.$$
where $W_p\in \pP$ has rank $p$ and $\St_p$ is as in Theorem~\ref{thm:SS}. 
This chain complex is exact in all cases and degrees, except for
\[ H_{\dim P}(C_\bullet)=\tilde H_{\dim \pP}(\pP) \]
in cases (2) and (4).
\end{thm}

\begin{proof}[Proof of Theorem~\ref{les}]
From Theorems~\ref{thm:barCM}, \ref{thm:PE is C-M} and~\ref{thm:PEU is CM}
we have that $\bP(V)$, $\bP(V,V_0)$ 
$\PPi(E)$, and $\PPi(E,U)$ are all Cohen-Macaulay. 
Let $\pP$ be one of the posets $\bP(V)$, $\PPi(E)$, $\bP(V,V_0)$ or $\PPi(E,U)$. 
Applying Theorem~\ref{thm:cpxofHCM} to $\pP$ in each case gives the above associated chain complexes, using the identification of the lower intervals $\pP_{<W_p}$ given in Proposition~\ref{prop:buildings}. The homology of the chain complex is the reduced homology of $\pP$, which is trivial in the 
cases of $\bP(V)$ and $\bP(V,V_0)$ since they are contractible, and non-trivial only in its top degree in the cases of $\PPi(E)$ and $\PPi(E,U)$ since they are Cohen-Macaulay.
\end{proof}

\begin{proof}[Proof of Theorem~\ref{thm:SS}]. 
Let $\GA,\GA',\pP,\pP'$ be as in the theorem.
We have $\dim \pP'=\dim \pP+1$ by 
Proposition~\ref{prop:buildings}(2).
Consider the complexes $C_\bullet=C_\bullet(\GA,\pP)$ 
and $C'_\bullet=C_\bullet(\GA',\pP')$ 
be the complexes of $\GA'$ and $\GA$--modules given in Theorem~\ref{les}. 
Choose projective resolutions $Q_\bullet$ and $Q'_\bullet$ of $\bbZ$ over $\GA$ and $\GA'$. Then for every $p\ge -1$, stabilization by $\F$ (resp.~$\HH$) induces a map of complexes 
$$Q_\bullet\otimes_{G}C_p\rar Q'_\bullet\otimes_{G'}C'_p$$
and 
$$H_i(\GA',\GA;C',C)=H_i(\Cone(Q_\bullet\otimes_{\GA}C_p\to Q'_\bullet\otimes_{G'} C_p')).$$
Assembling these cones for every $p$,
we get a double complex 
$$\Cone_{p,q}=Q_{q-1}\otimes_{\GA}C_p\ \textstyle{\bigoplus} \ Q'_q\otimes_{\GA'} C_p'$$
with horizontal differential that of $C_\bullet$ and $C'_\bullet$, and vertical differential that of the cone. 
There are two spectral sequences associated to this double complex. The first one has 
$$^1E:\ \ E^1_{pq}=Q_{q-1}\otimes_{\GA}H_p(C_\bullet)\  \textstyle{\bigoplus}\   Q'_q\otimes_{\GA'} H_p(C'_\bullet).$$
Now by Theorem~\ref{les}, $H_*(C_\bullet')=0=H_*(C_\bullet)$ in cases (1) and (3). In cases (2) and (4), $H_p(C_\bullet)=0$ 
when $p<\dim \pP$, and $H_p(C'_\bullet)=0$ when $p\le \dim \pP$. Since there are no nonzero terms in the range $p+q\le \dim \pP$, the spectral sequence converges to 0 in those degrees. 

Hence the other spectral sequence associated to the same double complex converges to 0 in the same range of degrees. It has $E^1$--term 
$$^2E:\ \ E^1_{pq}=\begin{cases}
H_q(\GA',\GA)&p=-1,\\ 
H_i(\GA',\GA;C'_p,C_p)=H_i(\GA',\GA;\induction{\St_p}{\GA'_{W_p}}{\GA'},\induction{\St_p}{\GA_{W_p}}{\GA}) & 0\leq p\leq \dim\pP\\
H_q(\GA';C'_p)=H_i(\GA';\induction{\St_p}{\GA'_{W_p}}{\GA'}) &p=\dim \pP+1
\end{cases}
$$
This can be rewritten as in the statement of the theorem using Shapiro's lemma and its relative version. 

Finally we consider the terms 
 $$E^1_{p,q}=H_0(\GA'_{W_{p}};\St_p) \ \ \ \textrm{when}\ p=\dim \pP+1.$$
In case (1) and (3), $W_p=V\op \F$ as it is a maximal subspace, so $\GA'_{W_{p}}$ is the whole group $\GL(V\op \F)$ or $\A^T(V\op\F,V_0\op \F)$, and  $E^1_{p,0}$ is the corresponding Steinberg coinvariants when $q=0$. These vanish by Theorem~\ref{thm:relcoinvars} and its absolute version (see eg., \cite[Thm.~1.1]{APS}). 

 When $\F=\F_{\ell^r}$, taking homology with coefficients in $\F_\ell$, we moreover 
have in case (1) $E^1_{p,q}=H_q(\G(V\op \F);\St(V)\otimes \F_\ell)=0$
for all $q$ (see eg., \cite[Sec 5.1]{GKRW}), giving the additional
vanishing statement in that case, while in case (2), 
 $E^1_{p,q}=H_q(\A^T(V\op \F,V_0\op \F);\St(V\op \F,V_0\op \F)\otimes \F_\ell)=0$ for all $q\le r(\ell-1)-2$ by \cite[Lem 5.2]{GKRW}).

For cases (2) and (4), the action of $\GA'_{W_{p}}$ on 
\[ \St_p = \begin{cases}
\St(W_p/\R(E))                                 & \text{in case (2)}, \\
\St\big(W_p/\R(E),(W_p\cap U^\perp)/\R(E)\big) & \text{in case (4)}.  
\end{cases} \]
is through the quotients
\[ \begin{cases}
\G(W_p)_{\R(E)}                        & \text{in case (2)}, \\
\A^T(W_p,W_p\cap U^\perp)_{id_{\R(E)}} & \text{in case (4)}.
\end{cases} \]
of Proposition~\ref{prop:stabilizers},
which surject onto
\[ \begin{cases}
\G(W_p/\R(E))                           & \text{in case (2)}, \\
\A^T(W_p/\R(E),(W_p\cap U^\perp)/\R(E)) & \text{in case (4)}.
\end{cases} \]
The corresponding coinvariants therefore reduce to coinvariants of the same form as in the linear cases, and vanish by the same theorems. 
\end{proof}

\subsection{Proof of Theorem~\ref{thm:mainstability}}\label{sec:proof}

Recall from equation~(\ref{equ:G=A}) and Lemma~\ref{lem:buildingstab2} that there are split inclusions of groups
\begin{equation*}
\G(V_0)\cong \A^T(V,V_0)_{W_0} \to A^T(V,V_0) \ \ \textrm{and}\ \  \Gi(U^\perp/U)\cong \Ai(E,U)_{W_0}\to \Ai(E,U)
\end{equation*}
for $W_0$ a rank 0 element of $\bP(V,V_0)$ (resp.~$\PPi(E,U)$), which are compatible with the stabilization maps. 
Hence we get maps of pairs 
\begin{equation}\label{equ:G=Aagain}
\begin{aligned}
\big(\G(V_0\op \F),\G(V_0)\big)&\rar \big(\A^T(V\op \F,V_0\op \F),\A^T(V,V_0)\big) \\ 
\big(\Gi(U^\perp/U\op \HH),\Gi(U^\perp/U)\big)&\rar \big(\Ai(E\op \HH,U),\Ai(E,U)\big)
\end{aligned}
\end{equation}
inducing injective maps in homology. 

For the rest of the section, set $\al=0$ unless we are in the linear case, with $\F=\F_{\ell^r}$ and we take coefficients in $\F_\ell$, in which case we set 
$$\al=r(\ell-1)-2.$$
We will prove the result by showing that the following two statements hold for every $d\ge 0$: 
\begin{enumerate}
\item[(I$_d$)] Let $(G,G',A,A')$ be either of the following 
$$\begin{array}{lllll}
(1)& G=\G(V_0)&G'=\G(V_0\op \F)&A=\A^T(V,V_0) & A'=\A^T(V\op \F,V_0\op \F)\\ 
(2)&G=\Gi(U^\perp/U) &G'=\Gi(U^\perp/U\op \HH)& A=\Ai(E,U) & A'=\Ai(E\op \HH,U),
\end{array}$$
with $\pP=\bP(V,V_0)$ in the first case and $\pP=\PPi(E,U)$ in the second case. Then the homomorphism (\ref{equ:G=Aagain}) induces an isomorphism 
\[ H_d(G',G)\stackrel{\cong}\rar H_d(A',A) \ \ \ \textrm{if}\  d\leq \dim \pP+1+\al; \]
\item[(II$_d$)] Let $(G,G')$ be either of the following
$$\begin{array}{lll}
(1)&G=\G(V) & G'=\G(V\op \F)  \\ 
(2)&G=\Gi(E) & G'=\Gi(E\op \HH)
\end{array}$$
with $\pP=\bP(V)$ in the first case and $\pP=\PPi(E)$ in the second case. 
Then 
\[ H_d(G',G)=0 \ \ \ \textrm{if} \  d\leq \dim \pP+1+\al. \]
\end{enumerate}

We start by showing that (I$_d$) and (II$_d$) together imply
\begin{enumerate}
\item[(I$_d$')]  Let $(A,A')$ be either of the following 
$$\begin{array}{lll}
(1)&A=\A^T(V,V_0) & A'=\A^T(V\op \F,V_0\op \F)\\ 
(2)& A=\Ai(E,U) & A'=\Ai(E\op \HH,U),
\end{array}$$
with $\pP=\bP(V,V_0)$ in the first case and $\pP=\PPi(E,U)$ in the second case. Then 
\[ H_d(A',A)=0 \ \ \ \textrm{if} \  d\leq \dim \pP+\al. \]
\end{enumerate}

\begin{proof}[Proof that (I$_d$) and (II$_d$) imply (I'$_d$).]
Indeed, as $d\leq \dim \pP+\al<\dim \pP+1+\al$, we can apply (I$_d$) in each case to get an isomorphism  $H_d(G',G)\cong H_d(A',A)$ for $(G,G')$ as in (I$_d$). 
Now to apply (II$_d$) in case (1) we need $\dim \bP(V,V_0)\le \dim \bP(V_0)+1$, which holds as both sides of the inequality equal $\dim V_0$.  
In case (2), we need $\dim\PPi(E,U)\le \dim \PPi(U^\perp/U)+1$, which holds as both sides equal $\g(U^\perp)$ by Propositions~\ref{prop:buildings} and~\ref{prop:induced form}.
\end{proof}

Before proving that (I) and (II) hold, we start by checking that they do together prove Theorem~\ref{thm:mainstability}. 

\begin{proof}[Proof that (I$_d$) and (II$_d$) imply Theorem~\ref{thm:mainstability} in degree $d$.]
(II$_d$) gives the non-relative part of the statement in the theorem, as $\dim \pP+1=\dim V$ in case (1) and  $\dim \pP+1=\g(E)$ in case (2). 
The relative part of the statement is given by (I'$_d$) using also that $A(V,V_0)\cong A^T(V^*,(V/V_0)^*)$ in the general linear case, together with the computation  
$\dim\bP(V^*,(V/V_0)^*)=\dim(V/V_0)^*=\dim V-\dim V_0$. In the isotropic case, it follows using the computation $\dim\PPi(E,U)=\g(E)-\dim U+\dim \R(E)$ 
of Proposition~\ref{prop:buildings}. 
\end{proof}

We will prove the statements (I$_d$) and (II$_d$) together, by induction on the degree $d$. 
The two statements trivially hold for $d=0$.  So we assume for induction that both statements hold for all $G,G',A,A'$ for all $d<q$, and we will prove that (I$_q$) and (II$_q$) then also hold.  We start by proving  (I$_q$).

\begin{proof}[Proof that (I$_d$) and (II$_d$) for all $d<q$ imply (I$_q$).]
Let $G,G',A,A'$ be the groups and $\pP,\pP'$ the associated posets, as in (I$_q$) and Theorem~\ref{thm:SS}.
Theorem~\ref{thm:SS}(2) associates a spectral sequence to the tuple $(A,A',\pP,\pP')$, which converges to zero in degrees $p+q\leq \dim \pP$, in fact in all degrees in case (1), 
and has
\[ E^1_{pq}=\begin{cases}
H_q(A',A)&\textrm{if }p=-1,\\
H_q(A'_{W_p},A_{W_p};\St_p)&\textrm{if }0\leq p\leq \dim \pP\\
H_q(A'_{W_p};\St_p)&\textrm{if }p=\dim \pP+1\\
0 &\textrm{if }p>\dim \pP+1
\end{cases}. \]
where $W_p\in \pP$ (or $\pP'$) has rank $p$, and
$$\St_p=\St(W_p,W_p\cap V_0) \ \ \ \textrm{or}\ \ \  \St_p=\St(W_p/\R(E),W_p\cap U^\perp/\R(E))$$ in the first and second cases respectively.

Recall that the map $(G',G)\to (A',A)$ in the statement is induced by the inclusion of the stabilizer of an element $W_0$ of rank 0 in $P$, using the identification  
$G^{'} \cong A^{'}_{W_0}$ from (\ref{equ:G=A}). Hence 
the differential $d^1:E^1_{0,q}\to E^1_{-1,q}$ coincides with the map we are interested in, and (I$_q$) can be rephrased as saying that this differential is an isomorphism 
as long as 
\[ q\leq \dim \pP+1+\al. \]
By Lemma~\ref{lem:buildingstab2}, the pair of inclusions
\[ (G',G)\to (A',A) \]
is split. Hence the induced map in homology is automatically injective, so we only need to check its surjectivity.  Since $E^\infty_{-1,q}=0$ 
(as $q-1\le \dim \pP$ in case (2) as $\al=0$ in that case), 
and no differentials can leave this position, it suffices to check that no differentials except for $d^1$ can hit this position.
This will follow if we can show that
\[ E^1_{p,q-p}=0\textrm{ when }1\leq p\leq \dim \pP+1, \]
recalling that $ E^1_{p,q-p}=0$ if $p>\dim \pP+1$ for any $q$. The case $p=\dim \pP+1$ is given by Theorem~\ref{thm:SS} as $E^1_{p,q-p}=E^1_{\dim \pP+1,q-p}$ with $q-p=0$ in all cases with $\al=0$, and $q-p\le \al$ in the finite field case at the characteristic with $\al=r(\ell-1)-2$.

\medskip

So we consider the case $p\le \dim \pP$, in which case 
\[ E^1_{p,q-p} = H_{q-p}(A'_{W_p},A_{W_p};\St_p). \]
Write $W_p^0=W_p\cap V_0$ in the general linear case and $W_p^0=W_p\cap U^\perp$ in the isotropic case.
Recall from Proposition \ref{prop:stabilizers} that there are 
group extensions for $A_{W_p}$ and $A'_{W_p}$ 
\begin{align*}
1\rar A_0:=\A(V_0,W_p^0) \rar &\ A_{W_p} \rar \A^T(W_p,W_p^0) \rar 1\\
1\rar A_0':=\A(V_0\op \F,W_p^0) \rar &\ A'_{W_p} \rar \A^T(W_p,W_p^0) \rar 1
\end{align*}
in the general linear case
and 
\begin{align*}
1\rar A_0:=\Ai(U^\perp/U,W_p^0) \rar &\ A_{W_p} \rar 
  \A^T(W_p,W_p^0)_{id_{\R(E)}} \rar 1\\
1\rar A_0':=\Ai(U^\perp/U\op \HH,W_p^0) \rar & \ A'_{W_p} \rar 
  \A^T(W_p,W_p^0)_{id_{\R(E)}} \rar 1
\end{align*}
in the isotropic case. As these are compatible with the stabilization maps, we get an associated Leray-Hochschild-Serre spectral sequence of the form 
\begin{align*}
E'^2_{a,b}&= H_a\big(A^T(W_p,W_p^0)_{(id_{\R(E)})};H_b(A_0',A_0;\St_p)\big)\\
&= H_a\big(A^T(W_p,W_p^0)_{(id_{\R(E)})};H_b(A_0',A_0)\otimes \St_p\big), 
\end{align*}
converging to $H_{q-p}(A'_{W_p},A_{W_p};\St_p)$, 
where we have used that $\St_p$ is free abelian and that the action of $A_0'$ on it is trivial in both cases because the action is through the cokernel of the above exact sequence, where $A_0'$ is the kernel.  
We want to show that this $E'^2$ page is zero in total degree
\[ a+b=q-p\leq \dim \pP+1+\al-p \ \ \ \textrm{with}\ 1\le p\le \dim \pP. \]
We may here apply the induction hypothesis to the groups $H_b(A_0',A_0)$, because the degree
\[ b\leq q-p<q \]
is less than $q$. 

\medskip

In the general linear case, $A_0=\A(V_0,W_p^0)\cong \A^T(V_0^*,(V_0/W_p^0)^*)$ and by 
(I'$_b$), $H_b(A_0',A_0)=0$ as long as 
$b\le \dim V_0-\dim W_p^0+\al=\dim \pP+\al-p$. 
In the isotropic case (where $\al=0$ always), $A_0=\Ai(U^\perp/U,W_p^0)$ so  $H_b(A_0',A_0)=0$ by (I'$_b$) as long as
$b\le \dim \PPi(U^\perp/U,W_p^0)=g(U^\perp/U)-\dim W_p^0= \dim \pP-p$ in this case too, where we used that $\g(U^\perp/U)=\g(U^\perp)$ (Proposition~\ref{prop:induced form}). 

So in both cases $E'^2_{ab}=0$ whenever $b\le \dim \pP+\al-p$. 

The only remaining case is that
$b=\dim \pP+1+\al-p$, in which case $a=0$. Writing 
$$\pP(A_0)=\bP(V_0^*,(V_0/W_p^0)^*)\cong \bP^T(V_0,W_p^0)=\pP^T(A_0) \ \ \ \textrm{or}\ \ \ \pP(A_0)=\PPi(U^\perp/U,W_p^0)$$ in each respective case, where the isomorphism $\pP(A_0)\cong \pP^T(A_0)$ is that of Lemma~\ref{lem:P and PT}, 
the same computation as the one we just did gives that $b\le \dim \pP(A_0)+1+\al$ in that case, so we can apply (I$_b$). 
Consider first the linear case. 
 Let $X\in \pP^T(A_0)$ be a rank 0 element, that is a complement of $W_p^0$ inside $V_0$. Then its image $(V_0/X)^*$ in $\bP(V_0^*,(V_0/W_p^0)^*)$ is also rank 0 and (I$_b$) 
shows that there is an isomorphism 
\begin{align*}
H_b(G',G) \rar  H_b\big(\A^T(V_0^*\op F^*,(V_0/W_p^0)^*\op F^*),\A^T(V_0^*,(V_0/W_p^0)^*)\big)
\end{align*}
for
$$G=\A^T(V_0^*,(V_0/W_p^0)^*)_{(V_0/X)^*}\ \ \ \textrm{and}\ \ \ G'=\A^T(V_0^*\op \F^*,(V_0/W_p^0)^*\op \F^*)_{(V_0/X)^*} .$$
Recall from Lemma~\ref{lem:P and PT} that the isomorphism $\pP(A_0)\cong \pP^T(A_0)$ is equivariant with respect to the isomorphism $\A^T(V_0^*,(V_0/W_p^0)^*) \cong \A(V_0,W_p^0)=A_0$, 
so the stabilizers $G$ of $(V_0/X)^*\in \pP(A_0)$ identifies under this latter isomorphism with the stabilizer of $X\in \pP^T(A_0)$, and likewise for $G'$. 
Hence (I$_b$) gives an isomorphism 
$$H_b((A'_0)_X,(A_0)_X)\sta{\cong}\rar H_b(A'_0,A_0)$$
By Proposition~\ref{prop:stabilizers}(3), conjugation by $A^T(W_p,W_p^0)$ is the identity on the stabilizer subgroups $(A_0')_X$ and $(A_0)_X$.
Hence we get in this case 
\[ E'^2_{0,n-p} \cong H_0\big(A^T(W_p,W_p^0);\St(W_p,W_p^0)\big)\otimes H_{n-p}\big((A_0')_X,(A_0)_X\big) \]
which is zero because the first factor vanishes by Theorem~\ref{thm:relcoinvars}.

\medskip

In the isotropic case, let 
$X$ be a rank 0 element in $\PPi(U^\perp/U,W_p^0)$. Then (I$_b$) gives an isomorphism 
$$H_b((A'_0)_X,(A_0)_X)\sta{\cong}\rar H_b(A'_0,A_0)$$
while  Proposition~\ref{prop:stabilizers}(3) gives that $\A^T(W_p,W_p^0)_{id_{\R(E)}}$ acts trivially on $(A'_0)_X,(A_0)_X$. 
Hence we get that  
\[E'^2_{0,n-p} \cong H_0\big(A^T(W_p,W_p^0)_{id_{\R(E)}};\St(W_p/\R(E),W_p^0/\R(E))\big)\otimes H_{n-p}((A'_0)_X,(A_0)_X)\]  
which is also zero by Theorem~\ref{thm:relcoinvars}, as
\[ A^T(W_p,W_p^0)_{id_{\R(E)}}\to A^T\big(W_p/\R(E),W_p^0/\R(E)\big) \]
is surjective and the action on $\St(W_p/\R(E),W_p^0/\R(E)$ is through this surjection.
Hence all the necessary entries in the previous spectral sequence are also zero, which finishes the proof of (I$_q$). 
\end{proof}

\begin{proof}[Proof that (I$_d$) for all $d\le q$ and (II$_d$) for all $d<q$ imply (II$_q$).]
Fix now $G,G'$ as in (II$_q$), with associated posets $\pP$ and $\pP'$. Let $\Si G=\G(\F\op V)$, $\Si G'=\G(\F\op V\op \F)$ in the general linear case, and likewise
$\Si G=\Gi(\HH\op E)$ and $\Si G'=\Gi(\HH\op E\op \HH)$ in the isotropic case, and let $\Si \pP=\bP(\F\op V)$ or $\PPi(\HH\op E)$ be the poset associated to $\Si G$ in each case, and $\Si \pP'$ defined similarly. 
Theorem~\ref{thm:SS}(1) associates a spectral sequence to the data $(\Si G,\Si G',\Si \pP,\Si \pP')$ in each case, which converges to zero in degrees $p+q\leq \dim \Si \pP=\dim \pP+1$, and in all degrees in the general linear case.
It has
\[ E^1_{pq}=\begin{cases}
H_q(\Si G',\Si G)&\textrm{if }p=-1,\\
H_q(\Si G'_{W_p},\Si G_{W_p};\St_p)&\textrm{if }0\leq p\leq \dim \pP+1\\
H_q(\Si G'_{W_p};\St_p)&\textrm{if }p=\dim P+2\\
0 &\textrm{if }p>\dim P+2
\end{cases} \]
where $W_p\in \Si \pP$ (or $\Si \pP'$) has rank $p$ with 
$$\St_p=\St(W_p) \ \ \ \ \textrm{or}\ \ \ \  \St(W_p/\R(E))$$ in the general linear case and  isotropic case respectively, where we recall from Lemma~\ref{lem:properties of HH} that  $\R(\HH\op E)=\R(E)$. 
We have also used the fact given by Proposition~\ref{prop:buildings}(2) that $\dim \pP'=\dim \pP+1$.

Note that $W_p$ has dimension $p+1$ in the general linear case and dimension $p+1+\dim \R(E)$ in the isotropic case (see Proposition~\ref{prop:buildings}).  In both cases, $\St_0=\St(\F^1)$ is the constant module $\bbZ$.

\medskip

We assume $q\le \dim \pP+1+\al$ and need to show that $H_q(G',G)=0$. This group does not as such appear in the spectral sequence, 
but taking $W_0=\F\le \F\op V$ in the general linear case and $W_0=L\le \HH\le \HH\op E$ in the isotropic case, Lemma~\ref{lem:buildingstab1} shows that 
there is an injective map 
$$H_q(G',G)\to H_q(\Si G'_{W_0},\Si G_{W_0})=E^1_{0,q}.$$
We will show that the differential 
$$d^1: E^1_{0,q}= H_q(\Si G'_{W_0},\Si G_{W_0}) \rar E^1_{-1,q}=H_q(\Si G',\Si G)$$
is injective when $q\le \dim \pP+1$. 
It will follow that the composition 
$$s_1\colon H_q(G',G)\rar H_q(\Si G',\Si G)$$
is also injective in this range of degrees. This composition is known as the ``lower suspension'' and the same argument applied to $(\Si G',\Si G)$ will show that also the map 
$s_2: H_q(\Si G',\Si G)\to H_q(\Si^2 G',\Si^2 G)$ is injective in this range. Statement (II$_q$) will then follow if we show that the composition 
$$H_q(G',G)\stackrel{s_1}\rar H_q(\Si G',\Si G) \stackrel{s_2}\rar H_q(\Si^2 G',\Si^2 G)$$
is the zero map. This is a by now standard argument, which can be found e.g.~in
\cite[Prop.~4.22]{RWW}. 
We sketch it here for completeness. 

Consider the commutative diagram 
$$\xymatrix{\cdots \ar[r] & H_q(G') \ar[d]\ar[r]^-{f_1}& H_q(G',G)\ar[d]^{s_1} \ar[r]^-{f_2} & H_{q-1}(G)\ar[d]^-{l_2} \ar[r]^-{l_2'} & H_{q-1}(G') \ar@{-->}[dl]_{\phi_*} \cdots \\
\cdots \ar[r] & H_q(\Si G') \ar@{-->}[dl]_{\psi_*}\ar[d]_{l_1}\ar[r]^-{g_1}&H_q(\Si G',\Si G) \ar[d]^{s_2} \ar[r]^-{g_2} & H_{q-1}(\Si G)\ar[d] \ar[r] & \cdots\\
\cdots H_g(\Si^2 G) \ar[r]^-{l_1'} & H_q(\Si^2 G') \ar[r]^-{h_1}&H_q(\Si^2 G',\Si^2 G)\ar[r]^-{h_2} & H_{q-1}(\Si^2 G) \ar[r] & \cdots
}$$ 
where the vertical maps are all induced by the lower suspensions and the horizontal maps come from the long exact sequence associated to the ``upper suspension'', which is our stabilization map.

For convenience, write $(X,Z)=(E,\HH)$ in the isotropic case, or $(X,Z)=(V,\F)$ in the general linear case, so that $G$ (resp.~$G'$, $\Si G$) is the group of automorphisms of $X$ (resp.~$X\op Z$, $Z\op X$).
The natural isomorphism $X\op Z\to Z\op X$ induces an isomorphism $\phi:G'\to \Si G$ which commutes with the respective inclusions of $G$ into these groups. This shows that $l_2$ and $l_2'$ have the same kernel, so the composition 
\[ g_2\circ s_1=l_2\circ f_2=0 \]
in the diagram. Hence the image of $s_1$ lies in the image of $g_1$, and thus the image of $s_2\circ s_1$ lies in the image of $s_2\circ g_1=h_1\circ l_1$. 

Now, let $\psi:\Si G'\to\Si^2 G$ be the isomorphism induced by the isometry 
$Z\op X\op Z\to Z\op Z\op X$ sending $(z,x,z')\mapsto(z',z,x)$.  
This time the triangle
\[ \xymatrix{ & \Si G' \ar[d]_{\Si}\ar[dl]_{\psi} \\
\Si^2 G \ar[r]^-{\textrm{incl}} & \Si^2 G' } \]
does not quite commute.  But it does commute up to conjugation in the target: specifically, conjugation by the automorphism of
$Z\op Z\op X\op Z$ switching the first and last factors.  Therefore it induces a commutative triangle in homology, showing that the induced maps $l_1$ and $l_1'$ have the same image.
It follows that $h_1\circ l_1=0$.  Combining this with the previous paragraph, we can conclude that $s_2\circ s_1=0$ as claimed.

\medskip

It remains only to check injectivity of the differential $d^1:E^1_{0,q}\to E^1_{-1,q}$ in the above spectral sequence. We are assuming
\[ q\le \dim \pP+1+\al, \] 
so $E^\infty_{0,q}=0$. The differential $d^1$ is the only one that can leave that position. It will necessarily be injective once we have shown that no other differential hits that position. 
We will do this by showing that 
$$E^1_{p,q-p+1}=0 \ \ \ \textrm{for all} \ 1\le p\le \dim \pP+2.$$ 
When $p=\dim \pP +2=\dim \Sigma \pP +1$, we have $q-p+1\le \al$ 
and $E^1_{p,q-p+1}=0$ by Theorem~\ref{thm:SS} as either $q-p+1=\al=0$ or we are in a case where this homology group  vanishes for any $q$.  So we are left to consider the case $p\le \dim \pP+1$. 
We have 
$$E^1_{p,q-p+1}=H_{q-p+1}(\Si G'_{W_p},\Si G_{W_p};\St_p).$$
We use the Leray-Hochschild-Serre spectral sequence associated to the short exact sequences of Proposition~\ref{prop:stabilizers}. The $E^2$--term of that spectral sequence has the form 
$$E'^2_{a,b}=H_a\big(\G(W_p)_{(\R(E))};H_b(A',A;\St_p)\big)$$
converging to $H_{a+b}(\Si G'_{W_p},\Si G_{W_p};\St_p)$, 
with
\[ A=\A(\F\op V,W_p),\quad A'=\A(\F\op V\op \F,W_p) \]
in the general linear case, and 
\[ A=\Ai(\HH\op E,W_p),\quad A'=\Ai(\HH\op E\op \HH,W_p) \]
in the isotropic case. We want to show that $E'^2_{a,b}=0$ for all $a+b=q-p+1$ with $1\le p\le \dim \pP+1$ and $q\le \dim \pP+1+\al$.

\medskip 

We start by analysing the isotropic case (where $\al=0$). 
When $a=0$, we have 
\begin{align*}E'^2_{0,q-p+1}&=H_0\big(\G(W_p)_{\R(E)};H_{q-p+1}(A',A;\St(W_p/\R(E)))\big) \\
&=H_0\big(\G(W_p)_{\R(E)};H_{q-p+1}(A',A)\otimes\St(W_p/\R(E))\big). 
\end{align*}
By (I$_{q+1-p}$), which we may use as $q+1-p\le q$, there is an isomorphism 
$$H_{q-p+1}\big(\Gi(W_p^\perp/W_p\op \HH),\Gi(W_p^\perp/W_p)\big)\to H_{q-p+1}(A',A)$$ 
as, using Lemma~\ref{lem:properties of HH} and the dimension and rank calculations of Proposition~\ref{prop:buildings}, 
\begin{align*}
q-p+1 \le \dim \pP-p+2 &= g(E)-p+1 = g(\HH\op E) - p \\
&= g(\HH\op E) - \dim(W_p) + \dim\R(\HH\op E) + 1 = \dim\PPi(\HH\op E,W_p) + 1.
\end{align*}
Now $\Gi(W_p^\perp/W_p)$ and $\Gi(W_p^\perp/W_p\op \HH)$ are the stabilizers of rank 0 elements in the buildings associated to $A$ and $A'$, and 
 Proposition~\ref{prop:stabilizers}(3) shows that the action of $\G(W_p)_{\R(E)}$ on these stabilizer subgroups is trivial. Hence 
the $E^2$-term can be rewritten as 
\[ E'^2_{0,q-p+1}=H_0\big(\G(W_p)_{\R(E)};\St(W_p/\R(E))\big)\otimes H_{q-p+1}(A',A), \]
which is zero because $\G(W_p)_{\R(E)}$ surjects onto $\G(W_p/\R(E))$, whose coinvariants in the module $\St(W_p/\R(E))$ vanish by  \cite[Thm.~1.1]{APS}, after checking that
\[ \dim (W_p/\R(E))=p+1\geq2. \]

When $a\ge 1$, we have $b\le q-p<q$ and we may use (I$_b'$). 
We also have $b\le q-p\leq\dim\PPi(\HH\op E,W_p)$ just as above.
Hence $H_b(A',A)=0$ and hence $E'^2_{a,b}=0$ in that case. 
This shows that $E'^2_{a,b}=0$ for all $a+b=q-p+1$, which finishes the proof of (II$_q$) in the isotropic case. 

\medskip

In the general linear case, we have
$$A\cong \A^T(\F^*\op V^*,(\F\op V/W_p)^*), $$
and similarly for $A'$, fitting together into an isomorphism of pairs. 
Also,
\begin{align*}
\dim\bP(\F^*\op V^*,(\F\op V/W_p)^*) +\al
& = \dim V+1-(p+1)+\al = \dim \pP-p+1 +\al \geq q-p. 
\end{align*}
This shows just as in the previous case 
that $E'^2_{a,b}=0$ for all $a+b=q-p+1$ by (I'$_b$), as long as $a\ge 1$.

When $a=0$, we have 
\begin{align*}E'^2_{0,q-p+1}&=H_0\big(\G(W_p);H_{q-p+1}(A',A;\St(W_p))\big)\\
&=H_0\big(\G(W_p);H_{q-p+1}(A',A)\otimes \St(W_p)\big). 
\end{align*} 
By Lemma~\ref{lem:P and PT}, there is an isomorphism $\bP^T(\F\op V,W_p)\cong \bP(\F^*\op V^*,(\F\op V/W_p)^*)$, which is equivariant with respect to the isomorphism $A\cong \A^T(\F^*\op V^*,(\F\op V/W_p)^*)=A^T$, and likewise for the suspended versions. Pick a rank 0 element $X$ in $\bP^T(\F\op V,W_p)$, that is a complement of $W_p$ in $V$, 
with $(\F\op V/X)^*$ its image in $\bP(\F^*\op V^*,(\F\op V/W_p)^*)$. 
By (I$_{q-p+1}$), which applies since
\[ \dim\bP\left(\F^*\op V^*,(\F\op V/W_p)^*\right)+1+\al \ge q-p+1\]
as checked above, 
there is an isomorphism 
$$H_{q-p+1}\left(A^T_{(\F\op V/X)^*},A^T_{(\F\op V/X)^*}\right)\rar H_{q-p+1}\left((A^T)',A^T)\right).$$
Now under the isomorphism $\A^T\cong A$, the group $A^T_{(\F\op V/X)^*}$ identifies with the stabilizer of $X$, which is rank 0 in the building associated to $A$, and likewise for the stabilized version. The action of  $\G(W_p)$ is trivial on this stabilizer (as can be check directly, or via Proposition~\ref{prop:stabilizers}(3)). 
Hence we get a decomposition
\[ E'^2_{0,q-p+1}=H_0(\G(W_p);\St(W_p))\otimes H_{q-p+1}(A',A). \]
And this is zero because the coinvariants vanish by  \cite[Thm.~1.1]{APS} as $W_p$ as dimension $p+1\ge 2$.
This finishes the proof of (II$_q$), and of the theorem. 
\end{proof}

\begin{proof}[Proof of Corollary~\ref{cor:C}]
By \cite[Thm III.4.6]{FP},  the stable homology of the groups
$\Sp_{2n}(\F_{p^r})$ and $\UU_{n,n}(\F_{p^{2r}})$ vanishes at the characteristic, and that
of $\OO_{n,n}(\F_{p^r})$  vanishes at the characteristic when $p$ is odd. The result then follows from Theorem~\ref{thm:mainstability}. 
\end{proof}

\section{Euclidean form}\label{sec:eucli}
In this section, we will use our homological stability result
stabilizing by adding hyperbolic planes, to  deduce in some cases, a stability result for the unitary and orthogonal groups of the standard Euclidean form on $\F^n$, proving Theorems~\ref{thm:introeucli} and~\ref{thm:intronondg} as well as Corollary~\ref{cor:D}. The idea is to relate the stabilizations maps $\op E$, for $E$ a non-degenerate formed space, to the stabilization maps $\op \HH$, under specific assumptions on the field. 

\medskip

Let $\F$ be a field with involution $\sigma$.
We will assume that the pair $(\F,\sigma)$ satisfies one of the following properties:
\begin{enumerate}[(A)]
\item\label{A} There exists $a\in\F$ such that $\bar a a=-1$.
\item\label{B} There exist $a,b\in\F$ such that $\bar a a+\bar b b=-1$.
\end{enumerate}
Clearly (\ref{A}) implies (\ref{B}). Also (\ref{A}) always holds in
characteristic $2$ since in that case $-1=1=\bar 1\cdot 1$. 
When $\F=\mathbb{C}$ is the complex numbers, (\ref{A}) is satisfied when $\s$ is the identity but {\em not} when $\s$ is complex conjugation.

When $\s=\id$, Property (\ref{A}) (resp.~ (\ref{B})) says that $-1$ is a square (resp.~a sum of two squares) in $\F$. For finite fields of odd characteristic, we have the following:

\begin{lemma}\label{lem:sumsquares} 
Let $\F_{p^r}$ be a finite field of odd characteristic. 
Any nonsquare element in  $\F_{p^r}$ can be written as a sum of two squares. 
\end{lemma}

In particular, for any finite field $\F_{p^r}$, the pair $(\F_{p^r},\s=\id)$ satisfies Property (\ref{B}). 

\begin{proof}
Let $q=p^r$.  Exactly half of the nonzero elements of $\F_q$ are squares because the multiplicative group $\F_q-\{0\}$ is cyclic of even order.  Zero is also a square, so the total number of squares is $(q+1)/2$.  If the set of squares were closed under addition, these elements would form a subgroup of the additive group $\F_q$, but that cannot be the case because $q$ is not divisible by $(q+1)/2$.  Hence there exist elements $a,b\in\F_q$ such that $a^2+b^2$ is not a square.  But we can then get any other nonsquare by scaling appropriately, as the quotient of any two nonsquare elements is a square
(since $|\F_q^\times/{\F_q^\times}^2|=2$).
\end{proof}

\begin{lemma}\label{lem:norm}
Let $(\F_{p^{2r}},\s)$ be a finite field with the non-trivial involution $\s=(-)^{p^r}$. Then the norm map
\[ N:\F^\times\to(\F^\sigma)^\times,\quad a\mapsto \bar a a \]
is surjective. In particular $(\F_{p^{2r}},\s\neq \id)$  always satisfies Property (\ref{A}). 
\end {lemma}

Note that the norm map is not surjective in general, as exemplified for instance by $\F=\bbC$ and $\sigma$ complex conjugation. 

\begin{proof} Set $q=p^r$. 
The map $N$ is a group homomorphism of multiplicative groups given explicitly by the formula  $N(a)=a^{q+1}$.
Its domain is a cyclic group of order $q^2-1$.  So the image is cyclic of order
$\frac{q^2-1}{q+1}=q-1$, 
the same as the order of the codomain.
\end{proof}

The following result allows us to compare stabilization by $E$, for $E$  non-degenerate, with stabilization by $\HH$:

\begin{prop}\label{prop:++++}
Let $E=(E,q)$ be a nondegenerate formed space over  $(\F,\sigma)$.
\begin{enumerate}
\item If $(\F,\sigma)$ satisfies Property (\ref{A}), then $E^{\op 2}\iso\HH^{\op \dim E}$. 
\item If $(\F,\sigma)$ satisfies Property (\ref{B}), then $E^{\op 4}\iso\HH^{\op 2\dim E}$.
\end{enumerate}
\end{prop}

\begin{proof}
Let $E'=(E,-q)$.
By Proposition~\ref{prop:hat+-},
\[ E\op E'\iso\HH^{\op \dim E}. \]
If there exists $a\in\F$ such that $a\bar a=-1$, then multiplying by $a$ gives an isomorphism $E'\iso E$, proving the first case.
In the second case, Lemma~\ref{lem:--++} says that $E'\op E'\iso E\op E$.  Applying this to a sum of two copies of the above equation proves the result.
\end{proof}

Just as we did with $\HH$ earlier, we can consider stabilization by $E$, that is the map $$\Gi(F)\to \Gi(F\op E)$$ extending an isometry by defining it to be the identity on the added $E$.
We are now ready to prove Theorem~\ref{thm:intronondg} from the introduction, which says that, for any non-degenerate formed space $E=(E,q)$ and taking $F=E^{\op n}$, the induced map 
\[H_d(\Gi(E^{\op n}))\rar H_d(\Gi(E^{\op n+1})) \]
is injective in degrees $d\leq k(\lfloor n/2\rfloor-e)-1$,
and surjective in degrees $d\leq k(\lfloor(n-1)/2\rfloor-e)$, 
where $e=0$ in the case of Property (\ref{A}) and $e=1$ in the case of Property (\ref{B}), with $k=\dim E$.

\begin{proof}[Proof of Theorem~\ref{thm:intronondg}]
Let $E'=(E,-q)$.  Lemma~\ref{lem:--++}  and Proposition~\ref{prop:hat+-} give that 
\[ E'\op E'\iso E\op E \ \ \ \ \textrm{and}\ \ \ \ E\op E'\iso\HH^{\op k}. \]
Consequently, the stability map
$\Gi(E^{\op n})\rar\Gi(E^{\op n}\op\HH^{\op k})$
factors as
\[ \Gi(E^{\op n})\rar\Gi(E^{\op n}\op E)\rar\Gi(E^{\op n}\op E\op E')\cong \Gi(E^{\op n}\op\HH^{\op k}). \]
So by Theorem~\ref{thm:mainstability} (applied $k$ times),
the first map, which is the stabilization by $E$, 
induces an injection on homology in degrees $d<\g(E^{\op n})$.
Meanwhile, 
so long as $n\geq2$,
\[ E^{\op n}\iso E^{\op n-2}\op E'\op E', \]
so, using now that we also have $\HH\cong E'\op E$,  the stability map
\[ \Gi(E^{\op n-2}\op E')\rar\Gi(E^{\op n-2}\op E'\op\HH^{\op k})=\Gi(E^{\op n-2}\op E'\op E'\op E) \]
factors as
\[\xymatrix{ \Gi(E^{\op n-2}\op E')\ar[r]& \Gi(E^{\op n-2}\op E'\op E')\ar[r]\ar[d]_{\cong}&\Gi(E^{\op n-2}\op E'\op E'\op E) \ar[d]^\cong\\
& \Gi(E^{\op n}) \ar[r]& \Gi(E^{\op n+1}).  }\]
Hence by Theorem~\ref{thm:mainstability},
the stabilization map 
$\Gi(E^{\op n})\rar\Gi(E^{\op n+1})$ 
induces a surjection on homology in degrees $d\leq\g(E^{\op n-2}\op E')$.

Now we estimate the two genera mentioned above.
By Proposition~\ref{prop:++++},
\[ E^{\op n}\iso\HH^{\op k\lfloor n/2\rfloor}\op E^{\op e} 
\ \ \ \ \textrm{or}\ \ \ \ 
 E^{\op n}\iso\HH^{\op 2k\lfloor n/4\rfloor}\op E^{\op e}, \]
respectively in cases (\ref{A}) or (\ref{B}).
Using Lemma~\ref{lem:props of sum} and
the fact that any $x\in\bbR$ satisfies
$ 2\lfloor x/2\rfloor\geq\lfloor x\rfloor-1$,
we get
\[ \g(E^n)\geq k(\lfloor n/2\rfloor-e). \]
Similarly, for $n\geq3$,
\begin{align*}
E^{\op n-2}\op E' \ \iso\ E^{\op n-3}\op\HH^{\op k}\ \iso\ \HH^{\op k\left(\lfloor (n-3)/2\rfloor+1-e\right)}\op E^{\op e}
\end{align*}
So
\[ \g(E^{\op n-2}\op E')\geq k(\lfloor(n-1)/2\rfloor-e). \]
If $n=2$, then this last conclusion is vacuously true. 
If $n<2$, then all of the claims in the proposition statement are vacuous.
\end{proof}

We explain now how Theorem~\ref{thm:intronondg} gives a 
stability result (Theorem~\ref{thm:introeucli})
for the the unitary groups $\UU(n,\F,\sigma)$ and orthogonal groups $\OO(n,\F)$ defined using the standard Euclidean form.

\begin{Def}\label{def:Euclidean} 
Set $\Lambda=\Lambda_{min}$. 
The {\em Euclidean unitary group} is the group 
\[ \UU(n,\F,\sigma)=\Gi({\mathcal{E}}^{\op n}) \]
where ${\mathcal{E}}=(\F^1,q_{\mathcal{E}})$ with
\[ q_{\mathcal{E}}(a,b)=\bar ab,  \]
and the {\em Euclidean orthogonal group} is the special case 
\[ O(n,\F)=\UU(n,\F,\id). \]
\end{Def}

\begin{rem}\label{rem:eucli} (1) 
Note that $\omega_{q_{\mathcal{E}}}(a,b)=(1+\eps)\bar a b$, 
so that $q$ is nondegenerate as long as $\eps\neq -1$.
In particular this one-dimensional formed space $\mathcal{E}$ is always nondegenerate in the
orthogonal case ($\sigma=id$, $\eps=1$, $\Lambda=0$) and in the
unitary case ($\sigma\neq id$, $\eps=1$, $\Lambda=\F^\s$), 
provided that $\text{char}(\F)\neq2$.

\smallskip

\noindent
(2) (Unitary groups, $\operatorname{char}(\F)= 2$) In the proof of Theorem~\ref{thm:introeucli} below, we will show that the group 
$\UU(n,\F,\sigma)$ in characteristic 2 actually always identifies, as long as $\F\neq \F_2$, with the automorphism group of a non-degenerate formed space, with parameter $\eps\neq 1$, allowing to apply our results to these groups too.

\smallskip

\noindent
(3) (Orthogonal groups, $\operatorname{char}(\F)= 2$)  The formed space $\mathcal{E}$ in characteristic $2$ is degenerate in our sense, with kernel $\K(\mathcal{E})=\F$. However the radical $\R(\mathcal{E})=0$ and the formed space $\mathcal{E}$ is defined as ``non-degenerate'' in for example \cite{FP}.  For finite fields of even characteristic, \cite[Sec 7.8]{FP} states that, in the orthogonal case, there is exactly one formed space with trivial radical in each odd dimension, namely the formed space $\HH^{\op n}\op \mathcal{E}$. This formed space is degenerate in our sense, so does not satisfy the hypothesis of Theorem~\ref{thm:intronondg}, but it does however fit in our first stability result, namely Theorem~\ref {thm:stabilityintro'}, being a genus $n$ formed space.

\smallskip

\noindent
(4) (Symplectic groups) There are no ``Euclidian symplectic group'' since $\omega_{\mathcal{E}}$ is trivial when $\eps=-1$ and $\s=\id$, but also $Q_{\mathcal{E}}$ is trivial when $\Lambda=\F$. 
\end{rem}

\begin{proof}[Proof of Theorem~\ref{thm:introeucli}]
If $\F$ has characteristic not 2, we have seen in Remark~\ref{rem:eucli} that the Euclidian formed space $\mathcal{E}$ is non-degenerate in both the orthogonal and unitary case, and the result follows directly from Theorem~\ref{thm:intronondg} with $k=1$.  

If $\F\neq \F_2$ has characteristic 2 and $\s\neq \id$, choose any $\alpha\in\F-\F^\sigma$ and set $\eps=\bar\alpha/\alpha$. Then 
$\eps\bar\eps=1$ and $\eps\neq1=-1$.  From \cite[Lem 2.10]{Forms}, we have that  multiplication by $\frac{1}{\alpha}$ defines a natural isomorphism 
$\Forms(E,\s,1,\Lambda)\rar  \Forms(E,\s,\eps,\frac{1}{\alpha}\Lambda)$. Let $\mathcal{E}'=(\F^1,\frac{1}{\alpha}q_{\mathcal{E}})$ be the image of $\mathcal{E}$ under this map for $E=\F^1$, and note that  
$(\mathcal{E}')^{\op n}$ is likewise the image of $\mathcal{E}^{\op n}$ with $E=\F^n$. 
By naturality, the automorphism groups $\Gi(\mathcal{E}'^{\op n})$ and $\Gi(\mathcal{E}^{\op n})$ are the same subgroup of $\GL(E)$. The result then follows by applying Theorem~\ref{thm:intronondg} to the formed space $\mathcal{E}'$. 
\end{proof}

Another consequence of Property (\ref{B}) is that the stable isometry groups do not depend on the choice of nondegenerate form used to stabilize by.

\begin{thm}\label{thm:stable group}
Assume $(\F,\sigma)$ satisfies Property (\ref{B}).  Let $E=(E,q)$ be a nondegenerate formed space.
The stable group
\[ \colim_n\Gi(E^{\op n}) \]
does not depend on the choice of $E$, up to group isomorphism.
\end{thm}

\begin{proof}
By Proposition~\ref{prop:++++}, $E^{\op 4}\iso\HH^{\op 2\dim E}$.
Hence 
$\colim_n\Gi(E^{\op n})\iso\colim_m\Gi(\HH^{\op m})$, 
as the two colimits have cofinal subsystems which coincide.
\end{proof}

Finally, we prove Corollary~\ref{cor:D}. The proof does not actually directly use Theorem~\ref{thm:intronondg}, but rather deduces the result from  Theorem~\ref{thm:mainstability} using the same ideas. 

\begin{proof}[Proof of Corollary~\ref{cor:D}]
For (1) with $n=2m$ even, we have that $\UU(2m,\F_{p^{2r}})\cong \UU_{m,m}(\F_{p^{2r}})$. By Corollary~\ref{cor:C}, the $d$th  homology of this group thus vanishes for $0<d<m=\frac{n}{2}$, giving the result. 
For $n=2m+1$, we have that $\UU(2m+1,\F_{p^{2r}})\cong \Gi(\mathcal{E}\op \HH^{\op m})$ as there is only one isomorphic class of non-degenerate Hermitian form in each dimension \cite[II.6.7]{FP}. The result then follows from combining the vanishing theorem \cite[Thm III.4.6]{FP} with Theorem~\ref{thm:mainstability}, as $\g(\mathcal{E}\op \HH^{\op m})=m=\frac{n-1}{2}$ (Lemma~\ref{lem:properties of HH}), given that $\lfloor\frac{n-1}{2}\rfloor=\lfloor\frac{n}{2}\rfloor$ when $n$ is odd. 

For (2), there is also only one isomorphism class of non-degerenare quadratic form over finite fields in odd dimensions \cite[II.4.5]{FP} and the same argument as above gives the vanishing of 
$H_d(\OO(n,\F_{p^r}),\F_p)$ when $p$ is odd, using  the vanishing theorem \cite[Thm III.4.6]{FP}  (which only works in odd characteristic for orthogonal groups) and Theorem~\ref{thm:mainstability}. 

In even dimension $n=2m$ there are two non-isomorphic quadratic forms \cite[II.4.5]{FP}, with the corresponding orthogonal groups denoted $\OO^+(2m,\F_{p^r})$ and $\OO^-(2m,\F_{p^r})$, though with only one stable group \cite[Prop II.4.11]{FP}. The adding rules for isomorphism classes of quadratic forms \cite[II.4.7]{FP} imply that we can always write such a quadratic form as $\HH^{\op m-1}\op \mathcal{E}\op 
\mathcal{E}^{\pm}$, where the last summand is either the Euclidean form, or the form with automorphism group $\OO^-(1,\F_{p^r})$. 
Combining \cite[Thm III.4.6]{FP} and Theorem~\ref{thm:mainstability} gives a vanishing for $0<d<m-1=\frac{n-2}{2}=\lfloor\frac{n-1}{2}\rfloor$ when $n$ is even. 
\end{proof}

\begin{rem}\label{rem:better}
The group $\OO_{m,m}(\F_{p^r})$ is isomorphic to $\OO^+(2m,\F_{p^r})$ or $\OO^-(2m,\F_{p^r})$ depending on whether its discriminant $(-1)^m$ is a square or not in $\F_{p^r}$ (see \cite[II.4.5]{FP}). Note that  $H_d(\OO_{m,m}(\F_{p^r}))=0$ when $0<d<m$, giving an improvement by one degree on the vanishing of the corresponding group $\OO^+(2m,\F_{p^r})$ or $\OO^-(2m,\F_{p^r})$ accordingly. 
\end{rem}

\appendix

\section{Forms and isotropic subspaces}\label{app:isotropic}

In this appendix, we study basic properties of subspaces of formed
spaces. Section~\ref{sec:complements} is concerned with the study of
orthogonal complements and Section~\ref{sec:modform} with the construction of a canonical form on the quotient $Z^\perp/Z$ given any isotropic subspace $Z$ of a formed space.

\subsection{Properties of orthogonal complements}\label{sec:complements}

In Section~\ref{sec:forms}, we have associated to a formed space $E=(E,q)$ the maps
$$\omega_q:E\x E\to \F, \ \ \ Q_q:E\to \F, \ \ \ \textrm{and}\ \  \ \flat_q=\omega_q(-,-):E\to E^*.$$ 
Recall from Definition~\ref{def:kerneletc} that ``orthogonality'' is defined using $\omega_q$, where one
has to remember that $\omega_q$ is typically here {\em not} an inner product, so that orthogonality does not behave in the usual manner. 
In particular, orthogonal complements are not actual vector space complements! 
Recall also from Definition~\ref{def:kerneletc} the kernel $\K(E)=\ker \flat_q$ and the radical $\R(E)\le \K(E)$, the subspace where also $Q_q$ vanishes.

\begin{lemma}\label{lem:dimension of perp}
Let $(E,q)$ be a finite dimensional formed space, and $A\leq E$ a subspace.  Then
\begin{align*}
\dim(A^\perp)
&= \dim(E)-\dim(A)+\dim(A\cap\K(E)).
\end{align*}
\end{lemma}
\begin{proof}
The orthogonal complement $A^\perp$ is the kernel of the composition
$ E\xrightarrow{\flat_q} {^\sigma E^*}\xrightarrow{incl^*}A^*$, 
which is the same as the kernel of the composition
\[ E\to E/\K(E)\to {^\sigma(E/\K(E))^*}\xrightarrow{incl^*}{^\sigma(A/A\cap \K(E))^*}, \]
in which all three maps are surjective.  So
$\dim(A^\perp) = \dim(E)-\dim(A/A\cap \K(E))$.
\end{proof}

\begin{lemma}\label{lem:perpperp}
Let $(E,q)$ be a finite dimensional formed space, and $A\leq E$ a subspace. Then $(A^\perp)^\perp=A+\K(E)$.
\end{lemma}

\begin{proof}
Clearly $A+\K(E)$ is always contained in $(A^\perp)^\perp$.  We will check that they have the same dimension.
Applying Lemma~\ref{lem:dimension of perp} to both $A^\perp$ and $A$ gives
\begin{align*}
\dim(A^{\perp\perp}) &= \dim(E)-\dim(A^\perp)+\dim(\K(E)) \\
&=  \dim(E)-[\dim(E)-\dim(A)+\dim(A\cap\K(E))]+\dim(\K(E)) = \dim(A+\K(E)). \qedhere
\end{align*}
\end{proof}

\begin{lem}\label{lem:capperp}
Let $A,B\le E$ be two subspaces of a formed space $(E,q)$. Then
\[ (A+B)^\perp=A^\perp\cap B^\perp. \]
If in addition $\K(E)\leq A$, then
\[ (A\cap B)^\perp=A^\perp + B^\perp. \]
\end{lem}
\begin{proof}
The first part follows directly from the definitions.  
For the second statement, we have  from Lemma~\ref{lem:perpperp}
$$A^\perp + B^\perp =(A^\perp+B^\perp)+\K(E) = \big((A^\perp+B^\perp)^\perp\big)^\perp.$$
On the other hand 
$$(A\cap B)^\perp =((A\cap B) +\K(E))^\perp = ((A+\K(E))\cap (B+\K(E)))^\perp = (A^{\perp\perp}\cap B^{\perp\perp})^\perp.$$
The result then follows from 
applying the first statement to the subspaces $A^\perp$ and $B^\perp$. 
\end{proof}

\begin{lemma}\label{lem:perp of nondegenerate}
Let $(E,q)$ be a formed space and $A\leq E$ a nondegenerate subspace.  Then
$$ E=A\op A^\perp  \ \ \ \textrm{with}\ \ \ \K(E) = \K(A^\perp) \ \ \ \textrm{and}\ \ \ \R(E) = \R(A^\perp).$$
\end{lemma}

\begin{proof}
If $x\in A\cap A^\perp$, then $x\in\K(A)$, which is zero by assumption, so $A$ and $A^\perp$ intersect trivially.
As also $A\cap \K(E)=0$, by Lemma~\ref{lem:dimension of perp} we get that 
$ \dim A^\perp= \codim A $
and hence $E=A\op A^\perp$. 
Now $\K(E) \leq \K(A^\perp)$ since $\K(E) \leq A^\perp$, but also  $\K(E) \geq \K(A^\perp)$ since $E=A+A^\perp$. 
Finally, the radicals are both the kernel of $Q_q$ on $\K(E)$.
\end{proof}

Finally, the following lemma is crucial in Section~\ref{sec:buildings}: it gives a description of the minimal elements in $\PPi(E,U)$, it is used in computing the connectivity of $\PPi(E,U)$, and finally it is used in defining splittings of stabilizer groups.

\begin{lemma}\label{lem:isotropic complement}
Let $(E,q)$ be a formed space, and let $U\leq E$ be isotropic.  Then  there exists an isotropic subspace $Z$ such that 
$$E=U^\perp\oplus Z.$$ 
The space $Z$ can be chosen to contain any isotropic subspace $Z_0$ with $W^\perp\cap Z_0=0$, and has the property that $Z\cap \R(E)=0$.
Furthermore, if also $U\cap \R(E)=0$, then 
\begin{enumerate}
\item $U\oplus Z$ is non-degenerate; 
\item $E=U\oplus Z^\perp=(U\oplus Z)\oplus (U\oplus Z)^\perp$; 
\end{enumerate}

\end{lemma}
\begin{proof}
For the first statement, we assume without loss of generality that $\R(E)\cap U=0$: if not, replace $U$ with a complement in $U$ to $\R(E)\cap U$, which does not change $U^\perp$.
Let $C$ be any linear complement to $U^\perp$, containing $Z_0$.
Consider the map
\[ f:U\to {^\sigma\! C^*} \]
sending $w$ to $\omega_q(-,w)$.
It is injective since an element of the kernel is orthogonal to both $C$ and $U^\perp$, hence lies in $\K(E)\cap U=\R(E)\cap U=0$.
It is surjective since, if not, there would be a nonzero element of $C$ orthogonal to $U$, which does not exist.
Hence $f$ is bijective.

Let $x_1,\dots,x_n$ be a basis for $C$, such that $x_1,\dots,x_\ell$ is a basis for $Z_0$.
Consider the dual basis for
${^\sigma\! C^*}$; 
via $f$, it induces a basis $w_1,\dots,w_n$ for $U$ with the property that 
\[ \omega_q(x_i,w_j)=\delta_{ij}. \]
For each $j$, let $c_j\in \F$ be such that $Q_q(x_j)=[c_j]\mod \Lambda$ and set
\[ y_j=x_j-c_jw_j-\sum_{k=1}^{j-1}\omega_q(x_k,x_j)w_k. \]
Set 
$Z=\langle y_1,\dots, y_n\rangle$. 
Since all $w_j\in U\leq U^\perp$, $Z$ is another linear complement to $U^\perp$.
Since $Z_0$ is isotropic, $y_j=x_j$ for $j\leq\ell$, so $Z_0\leq Z$ still.
Using that $\omega_q(x_i,w_j)=\delta_{ij}$ and $\omega_q(w_i,w_j)=0$, one gets 
\[ \omega_q(y_i,y_j) = \omega_q(x_i,x_j)-\omega_q(x_i,\omega_q(x_i,x_j)w_i) = 0, \]
for all $i<j$,
and
\[ Q_q(y_j) = Q_q(x_j)+0-\omega_q(x_j,c_jw_j)= 0 \mod \Lambda \]
for all $j$, where we now also use that $Q_q(w_j)=0$.  As the vanishing of $q_q(y_j)$ implies that of $\omega_q(y_j,y_j)$, it follows that $Z$ is isotropic, as desired, proving the first part of the statement.

\medskip

For the second part of the statement, we start off by counting dimensions.
By Lemma~\ref{lem:dimension of perp},
\[ \dim(U^\perp)
= \codim U \ \ \ \textrm{and}\ \ \ \   \dim(Z^\perp)=\codim Z \]
since neither of them intersects $\K(E)$. 
It follows that 
\[ \codim Z=\dim U^\perp= \codim U \]
since $E=U^\perp\op Z$.

Now, we claim that $U$ is also a complement to $Z^\perp$.  If $x\in U\cap Z^\perp$, then $x$ is orthogonal to both $U^\perp$ and $Z$, hence to all of $E$, so $x\in\K(E)$, forcing $x=0$ as $U\cap \K(E)=0$.  Then the above dimension count shows that
\[ U\op Z^\perp = E, \]
proving the first part of statement (2). 
As  $\K(U\op Z)$ is contained in $Z^\perp\le E=U\op Z^\perp$, and hence in $Z$. Similarly,  $\K(U\op Z)$ is contained in $U$, so the kernel is zero as $U\cap Z=0$, proving statement (1).
The second part of (1) then follows from applying Lemma~\ref{lem:perp of nondegenerate}. 
\end{proof}

\subsection{The formed space $Z^\perp/Z$ for $Z$ isotropic}\label{sec:modform}

Given an isotropic subspace $Z\le E$, we will now construct a canonical form on $Z^\perp/Z$. We start by proving a few additional properties of orthogonal complements of isotropic subspaces.

\begin{lemma}\label{lem:radical of perp}
Let $(E,q)$ be a formed space and $Z\leq E$ an isotropic subspace.  Then
\[ \R(Z^\perp)=Z+\R(E) \]
\end{lemma}
\begin{proof}
Clearly $Z$ and $\R(E)$ are contained in $\R(Z^\perp)$, so we will check the reverse inclusion.  First,
\[ \K(Z^\perp) = Z^\perp\cap Z^{\perp\perp} = Z+\K(E), \]
using Lemma~\ref{lem:perpperp}.
Now, since $\omega_q$ vanishes on $Z+\K(E)$, $Q_q$ is additive there.  But $Q_q$ vanishes on $Z$.  So if $z\in Z$ and $k\in\K(E)$ such that
$0 = Q_q(z+k)=Q_q(k)$, then $k\in\R(E)$, so $z+k\in Z+\R(E)$ as claimed.
\end{proof}

\begin{lem}\label{lem:induced form1}
Let $Z<E$ be a subspace. Pulling back along the quotient map $E\to E/Z$ induces a bijection 
between forms $q$ on $E/Z$ with radical $R$ and forms $\hat q$ on $E$ with radical $R+Z$. 
In particular, if $(E,\hat q)$ is a formed space and $Z\le \R(E)$, there is a uniquely defined form on $E/Z$ such that the projection $E\to E/Z$ is an isometry; its radical is $\R(E)/Z$.
\end{lem}

It follows that pulling back via the quotient map $E\to E/Z$ determines a bijection between the isotropic subspaces of $E/Z$ and the isotropic subspaces of $E$ containing $Z$.

\begin{proof}
The projection map $p:E\to E/Z$ induces a map $p^*$ from forms on $E/Z$ to forms on $E$ by pull-back.  
It is injective: if $p^*q=p^*q'$, they have the same associated $\omega$ and $Q$, which implies that 
$\omega_{q}=\omega_{q'}$ and $Q_{q}=Q_{q'}$ because $p$ is surjective, which in turn implies $q'=q''$ by \cite[Thm 2.5]{Forms}. 

If $(E/Z,q)$ has kernel $K$ and radical $R$, we have that $(E,p^*q)$ has kernel $K+Z$ as $\flat_{p^*q}$ factors as 
$E\sta{p}\rar E/Z\sta{\flat_q}\rar (E/Z)^*\inc E^*$. The radical of  $(E,p^*q)$ is $R+Z$, computed similarly.  
Suppose that $(E,\hat q)$ is a formed space with $\R(E)=Z+R$ for some $R\le E$. We need to check that there is a form $q$ on $E/Z$ with radical $(R+Z)/Z$ with $p^*q=\hat q$. 
Pick an arbitrary linear complement
\[ Z\oplus L=E. \]
The quotient map $p$ induces an isomorphism $L\cong E/Z$. Define $q$ to be the form $q|_L$, seen as a form om $E/Z$ via this isomorphism.  
Using that $Z\le \R(E)$, one checks that $\omega_{p^* q}=\omega_{\hat q}$ and $Q_{p^*q}=Q_{\hat q}$, which shows that $p^*q=\hat q$ by \cite[Thm 2.5]{Forms}. 
\end{proof}

Recall from Definition~\ref{def:kerneletc} that the genus $\g$ of a formed space is the dimension of a maximal isotropic space minus the dimension of the radical.  

\begin{prop}\label{prop:induced form}
Let $(E,q)$ be a formed space and $Z\leq E$ an isotropic subspace.
There is a unique form on $Z^\perp/Z$ such that the quotient map
\[ \pi:Z^\perp\to Z^\perp/Z \]
is an isometry.  Furthermore, its radical is the image under the quotient map of $Z+\R(E)$ 
and 
its genus is 
\[ \g(Z^\perp/Z)=\g(Z^\perp)=\g(E)+\dim\R(E) -\dim(Z+\R(E)). \]
\end{prop}

\begin{proof}
As $Z\le \R(Z^\perp)$, we can apply Lemma~\ref{lem:induced form1} to get a unique from on $Z^\perp/Z$ induced by $(Z^\perp,q|_{Z^\perp})$, with the quotient map inducing an isometry. 
Using in addition Lemma~\ref{lem:radical of perp}, we get that $\R(Z^\perp)=Z+\R(E)$, showing that its radical is as stated. 

For the genus computation, the first equality follows as $Z^\perp$ and $Z^\perp/Z$ agree after modding out their radicals.
Now, a maximal isotropic subspace of $Z^\perp$ must contain $\R(Z^\perp)=Z+\R(E)$. 
It must then be a maximal isotropic subspace of $E$ as well as it contains $Z$, so any larger one is necessarily contained in $Z^\perp$. So its dimension is $\g(E)+\dim\R(E)$.  The second equality follows.
\end{proof}

\section{Cohen-Macaulay posets}\label{app:CM}
Recall that a poset 
is Cohen-Macaulay if every interval in it is spherical (see Definition~\ref{def:CM}).
In this section, we show how, for a Cohen-Macaulay poset $\pP$, one can
use the homology of lower intervals to construct a chain complex
computing the reduced homology of $\pP$. Such a construction appears in Quillen's notebooks \cite{QuillenNotes}, and is used in for example \cite[Sec 2]{CharneyDed} and \cite[Sec 2]{Vogtmann}, but we do not know of a reference for it including proofs. 

\begin{thm}\label{thm:cpxofHCM}
Let $\pP$ be a Cohen-Macaulay poset and denote by $\pP_r\subset \pP$ the subset of elements of rank $r$.  Then there is a chain complex $\tilde C_\bullet(\pP)$ with terms
\[ \tilde C_r=\bigoplus_{X\in \pP_r}\tilde H_{r-1}(\pP_{<X}) \]
and $\tilde C_{-1}=\bbZ$, and whose homology is the reduced homology $\tilde H_*(\pP)$. 
\end{thm}
In particular, this chain complex is exact except in degree $\dim \pP$.
Note that, as $\tilde H_{-1}(\emptyset)=\bbZ$, we have 
\[ \tilde C_0(\bar \pP)=\bbZ \pP_0. \]

\begin{lemma}\label{lem:discrete complement}
Let $\mathcal{Q}\leq \pP$ be posets such that $\pP-\mathcal{Q}$ is discrete, i.e.~has no nontrivial order relations.  There is a homeomorphism
\[ |\pP|/|\mathcal{Q}|\simeq \bigvee_{X\in \pP-\mathcal{Q}} \Sigma \abs{\pP_{<X}*\pP_{>X}} \]
which is natural with respect to maps of such pairs, where $\Sigma$ denotes the unreduced suspension.
\end{lemma}

\begin{proof}
The assumption that $\pP-\mathcal{Q}$ is discrete gives that the simplices of $|\pP|$ that are not contained in $|\mathcal{Q}|$ are precisely those that contain precisely one element of $\pP-\mathcal{Q}$. Hence we have that 
\[ |\pP|/|\mathcal{Q}|\simeq \Big(|\mathcal{Q}| \bigcup_{X\in \pP-\mathcal{Q}} \abs{\pP_{<X}*\{p\}*\pP_{>X}}\Big) /|\mathcal{Q}|.\]
Noting that, for each $X\in \pP-\mathcal{Q}$, we have that $\abs{\pP_{<X}*\{X\}*\pP_{>X}}\cap |\mathcal{Q}|=\abs{\pP_{<X}*\pP_{>X}}$, we get the formula given in the statement. 
Naturality follows as the above decomposition is natural. 
\end{proof}

For all $r\geq0$, define
\[ \pP_{\leq r}=\set{X\in \pP}{\rank X\leq r}. \]
In particular, $\pP_r=\pP_{\leq r}\minus \pP_{\leq r-1}$.  Since $\pP_r$ is discrete (two distinct comparable elements cannot have the same rank), we get:
\begin{cor}
Let $\pP$ be graded poset. Then there is a natural (with respect to rank-preserving poset maps)
homotopy equivalence
\[ |\pP_{\leq r}|/|\pP_{\leq r-1}| \to\bigvee_{X\in \pP_r}
\Sigma|\pP_{<X}|. \]
\end{cor}

\begin{proof}[Proof of Theorem~\ref{thm:cpxofHCM}] 
The rank defines a finite filtration of the realization of $\pP$ by closed CW-inclusions.  There is an associated spectral sequence beginning with
\[ E^1_{pq}=\tilde H_{p+q}(|\pP_{\leq p}|/|\pP_{\leq p-1}|) \]
and converging to $H_{p+q}(|\pP|)=H_{p+q}(\pP)$. 
By the previous result,
\[ \tilde H_j(|\bar \pP_{\leq p}|/|\bar \pP_{\leq p-1}|)=
\bigoplus_{X\in \pP_p}\tilde H_{j-1}(|\pP_{<X}|) \]
which, by the Cohen-Macaulay property of $\pP$, is zero except when $j-1=p-1$.
Hence the spectral sequence is zero except along the row $q=0$, and so we have $E_2=E_\infty$, and also $E_\infty$ canonically identifies with the abutment.  In other words, it is merely a chain complex.
\end{proof}

\bibliographystyle{plain}
\bibliography{biblio}

\end{document}